\documentclass{article}

\usepackage[english]{babel}
\usepackage{amsthm}
\usepackage{amsmath}
\usepackage{amsfonts}
\usepackage{amssymb}
\usepackage{makeidx}
\usepackage{graphicx}
\usepackage[left=2cm,right=2cm,top=2cm,bottom=2cm]{geometry}
\usepackage[colorlinks]{hyperref}
\usepackage{xcolor}
\usepackage[justification=centering]{caption}
\usepackage{subcaption}
\usepackage{algorithm}
\usepackage[noend]{algpseudocode}
\usepackage{cleveref}
\crefname{equation}{Eq.}{Eqs.}
\usepackage{balance}
\usepackage{mathtools}
\DeclareMathOperator*{\argmax}{arg\,max}
\DeclareMathOperator*{\argmin}{arg\,min}
\DeclareMathOperator*{\Argmax}{Arg\,max}
\DeclareMathOperator*{\Argmin}{Arg\,min}

\renewcommand{\geq}{\geqslant}

\newcommand{\markupdraft}[2]{
	\ifthenelse{\equal{#1}{display}}{#2}{}
	\ifthenelse{\equal{#1}{color}}{\color{#2}}{}
}

\newcommand{\newcolored}[3][]{{\markupdraft{color}{#2}#3}
	\ifthenelse{\equal{#1}{}}{}{\markupdraft{display}{{\color{yellow!70!black}[#1]}}}} 

\newcommand{\del}[2][]{{\markupdraft{display}{{\color{yellow!99!black}[removed: "#2"[#1]]}}}} 
\newcommand{\new}[2][]{\newcolored[#1]{blue!75!black}{#2}}
\newcommand{\nnew}[2][]{\newcolored[#1]{red!85!black!90!blue}{#2}}

\newcommand{\indraftonly}[1]{{#1}}  
\renewcommand{\indraftonly}[1]{}\renewcommand{\markupdraft}[2]{}  
\renewcommand{\del}[2][]{}

\newcommand{\anne}[1]{\indraftonly{{\color{magenta}Anne: #1}}}
\newcommand{\niko}[1]{\indraftonly{\color{green!60!black}~#1$_{-\!\mathrm{Niko}}$}}

\newcommand{\armand}[1]{\indraftonly{{\color{cyan}Armand: #1}}}

\newcommand{\todo}[1]{\indraftonly{\bf {\color{red}TODO: #1}}}

\newcommand{\done}[1]{\indraftonly{\bf {\color{gray}DONE: #1}}}

\newcommand{\blank}[1]{}

\newcommand{\vertiii}[1]{{\left\vert\kern-0.25ex\left\vert\kern-0.25ex\left\vert #1 
		\right\vert\kern-0.25ex\right\vert\kern-0.25ex\right\vert_2}}

\newcommand*{\1}{\text{\usefont{U}{bbold}{m}{n}1}}

\newcommand{\rank}{\mathrm{rank}\,}

\newcommand{\bR}{\mathbb{R}}

\newcommand{\bN}{\mathbb{N}}
\newcommand{\bE}{\mathbb{E}}

\newcommand{\bP}{\mathbb{P}}

\newcommand{\cF}{\mathcal{F}}
\newcommand{\cB}{\mathcal{B}}

\newcommand{\cX}{\mathsf{X}}
\newcommand{\cY}{\mathsf{Y}}

\newcommand{\cV}{\mathsf{V}}

\newcommand{\cC}{\mathcal{C}}
\newcommand{\cO}{\mathcal{O}}

\newcommand{\cU}{\mathsf{U}}

\newcommand{\mueff}{\mu_{\mathrm{eff}}}
\newcommand{\diag}{\mathrm{diag}}

\newtheorem{proposition}{Proposition}
\newtheorem{theorem}{Theorem}
\newtheorem{corollary}{Corollary}
\newtheorem{lemma}{Lemma}

\usepackage{xparse}

\newcommand{\R}{\mathbb{R}}

\newcommand{\Sd}{\mathcal{S}^d}
\newcommand{\Sdp}{\mathcal{S}^d_{+}}
\newcommand{\Sdpp}{\mathcal{S}^d_{++}}

\newcommand{\wi}[1][i]{{w^m_{#1}}}
\newcommand{\vwm}{\ensuremath{\mathbf{w}_m}}
\newcommand{\wic}[1][i]{{w^c_{#1}}}

\NewDocumentCommand{\yt}{ O{t} }{Y_{#1}}
\NewDocumentCommand{\yit}{ O{i} O{t} }{\left[Y_{#2}\right]_{#1}}

\newcommand{\pUk}[1]{\nu_U^{#1}} 
\newcommand{\pUd}{\pUk{d}}
\newcommand{\density}[1]{p_U^{#1}}
\newcommand{\densityd}{\density{d}}
\newcommand{\CSA}{\Gamma_{\text{CSA}}^1}
\newcommand{\CSAd}{\Gamma_{\text{CSA}}^2}

	\newcommand{\Fz}{F_z}
	\newcommand{\Fq}{F_q}
	\newcommand{\FK}{F_{\ncovmat}}
	\newcommand{\Fr}{F_r}
	\newcommand{\Fsig}{F_p}
	\newcommand{\varnorm}{z,p,q,\nnew{\rncovmat},r}

\usepackage{enumitem} 
\setlist[enumerate]{wide = 0pt, leftmargin=*}

\numberwithin{proposition}{section}
\numberwithin{theorem}{section}
\numberwithin{definition}{section}
\numberwithin{corollary}{section}
\numberwithin{lemma}{section}
\numberwithin{remark}{section}
\numberwithin{example}{section}
\numberwithin{equation}{section}

\newtheoremstyle{assumption}{\topsep}{\topsep}{\itshape}{0pt}{}{.}{ }{{\bf\thmname{#1}}{\thmnumber{#2}}\textnormal{\thmnote{ (#3)}}}
\newtheoremstyle{assumptionbis}{\topsep}{\topsep}{}{0pt}{}{}{ }{\ensuremath{({\bf\thmname{#1}}_{\thmnumber{#2}}\textnormal{\thmnote{ (#3)}})}}

\theoremstyle{assumption}
\newtheorem{assumptionH}{{H}}

\crefformat{assumptionH}{{\bf H}{\normalfont #2#1#3}}

\newtheorem{assumptionF}{\textbf{F}}

\crefformat{assumptionF}{{\bf F}{\normalfont #2#1#3}}

\newtheorem{assumptionP}{\textbf{N}}

\crefformat{assumptionP}{{\bf N}{\normalfont #2#1#3}}

\newtheorem{assumptionG}{\ensuremath{\boldsymbol{\Gamma}}}

\crefformat{assumptionG}{{\ensuremath{\boldsymbol{\Gamma}}}{\normalfont #2#1#3}}

\newtheorem{assumptionR}{\textbf{R}}

\crefformat{assumptionR}{{\bf R}{\normalfont #2#1#3}}

\newtheorem{assumptionRho}{\ensuremath{\boldsymbol{\rho}}}

\crefformat{assumptionRho}{{\ensuremath{\boldsymbol{\rho}}}{\normalfont #2#1#3}}

\newcommand{\w}{\mathbf{w}} 
\newcommand{\Fpstar}{F_{c}^{\text{p}}}
\newcommand{\Fpsigma}{F^{\text{p}}_{c_\sigma}}
\newcommand{\Fpc}{F^{\text{p}}_{c_c}}
\newcommand{\Utt}{U_{t+1}^{s_{t+1}}}
\newcommand{\Fx}{F^{\text{m}}_{c_m}}

\newcommand{\Fc}{F^{\text{C}}_{c_1,c_\mu}}

\newcommand{\covmat}{\mathbf{C}}
\newcommand{\ncovmat}{\boldsymbol{\Sigma}}
\newcommand{\rncovmat}{{\hat{\ncovmat}}}
\newcommand{\Z}{z}
\newcommand{\X}{m}

\newcommand{\Id}{\mathbf I_d}

\newcommand{\wrt}{with respect to}
\newcommand{\wilog}{without loss of generality}

\newcommand{\lsc}{lower semicontinuous}

\begin{document}
	
		
		
		
		\title{Irreducibility of nonsmooth state-space models with an application to CMA-ES
				%
			}
		
		
		\author{Armand Gissler, Shan-Conrad Wolf, Anne Auger, Nikolaus Hansen}
		
		
		\maketitle
		
		\begin{abstract}
			We analyze a stochastic process resulting from the normalization of states in the zeroth-order optimization method CMA-ES.
			On a specific class of minimization problems \del{known as}\nnew{where the objective function is} scaling-invariant, this process defines a time-homogeneous Markov chain\del{, and we can study its convergence.} whose convergence at a geometric rate can \nnew{imply the}\del{yield\del{ to} guarantees}\niko{is there more than one guaranty of linear convergence?}\del{ of} linear convergence of CMA-ES.
			However, the analysis of the \new{intricate}\del{complex} updates for this process \del{presents}\nnew{constitute} a great mathematical challenge.
 We establish that this Markov chain is an irreducible and aperiodic T-chain.
			These contributions represent a first major step for the convergence analysis towards a stationary distribution.
	We rely for this analysis on conditions for the irreducibility of nonsmooth state-space models on manifolds.
			\new{To obtain our results,}\del{Besides,} we extend these \new{conditions to address}\del{results to conclude to\niko{my dictionaries don't contain "conclude" with the word "to" after}} the irreducibility in different hyperparameter settings that define different Markov chains, 
				and to include nonsmooth state spaces.

		\end{abstract}
		
		
		
%
%
%
		

	

	{\tableofcontents}

	\section{Introduction}

	The convergence of stochastic processes is at the core of many algorithms in various domains. Well-known examples include Markov chain Monte-Carlo (MCMC) algorithms~\cite{brooks2011handbook} like the Metro\-polis-Hastings algorithm~\cite{metropolis1953equation,hastings1970monte} that aim to sample a target distribution $\pi$ by generating a Markov chain with stationary probability measure $\pi$. Fast convergence of the Markov chain towards $\pi$  is one important\del{ desired} property for the underlying algorithms. It can be described qualitatively as the geometric ergodicity of the Markov chain, i.e., convergence at a geometric rate towards $\pi$, a question that has been widely studied \cite{gallegos2023equivalences,roberts2004general}.
	We focus here on an application of stochastic processes to the domain of numerical stochastic optimization which is closely connected to MCMC. We analyze indeed a Markov chain underlying 
	the so-called covariance matrix adaptation evolution strategy (CMA-ES)~\cite{hansen2001completely,hansen2004evaluating}, a widely used stochastic derivative-free optimization algorithm \cite{rodriguez2006hybrid,colutto2010cmaes,bieler2014robust,fujii2018exploring,akiba2019optuna,maki2020application,tanabe2021level,aboyeji2024covariance}\footnote{As of March 2024, two Python implementations of CMA-ES received together more than 60 millions downloads.} that can tackle difficult optimization problems which are notably nonconvex, multimodal and ill-conditioned.
	The algorithm minimizes a function $f: \mathbb{R}^n \to \mathbb{R}$ by sampling Gaussian vectors whose mean and covariance matrix are adapted iteratively.
	The adaptation of the parameters of the Gaussian distribution has been carefully designed, combining several independent principles \cite{hansen1996adapting,hansen2001completely,hansen2003reducing,ostermeier1994stepsize}.
	Ample empirical evidence shows that the algorithm converges geometrically fast~\cite{hansen2001completely,hansen2003reducing,hansen2014principled,hansen2015evolution}---in optimization referred to as linear convergence---towards the optimum on large classes of functions and the covariance matrix learns the inverse Hessian~\cite{hansen2011cmaes} up to a scalar factor on strictly convex quadratic problems. Yet, establishing a convergence proof of CMA-ES that reflects its working principle (i.e., without modifying the algorithm to enforce convergence) is still an open and difficult theoretical question.

	In this context, we extend a methodology that was already successful to analyze stepsize adaptive algorithms~\cite{auger2005convergence,auger2013linear,bienvenue2003global,auger2016linear,toure2023global} to prove the convergence of CMA-ES by exploiting its mechanisms and reflecting its working principle, including the learning of second-order information.
The methodology is based on the definition of a normalized Markov chain that models the algorithm when minimizing a scaling-invariant function, a function class that includes non quasi-convex functions~\cite{toure2021scaling}.
		As we will explain, if this Markov chain is stable---in the sense that it converges to a stationary distribution geometrically fast and satisfies a Law of Large Numbers---then the linear convergence of the algorithm follows. With more work, the learning of the inverse Hessian on strictly convex-quadratic functions should follow as well.
		In order to obtain such stability properties, the irreducibility of the process (the definitions will be recalled in the paper) is a necessary condition.
		On top of establishing the irreducibility of this Markov chain, we prove that it is an aperiodic T-chain, paving the way to a convergence analysis by means of a geometric drift condition.

Because of the 	intricacy of the algorithm, the irreducibility cannot be easily established by simply investigating the transition kernel of the Markov chain. Instead, we rely on recent results connecting the irreducibility of a Markov chain defined on a smooth manifold to the stability of an underlying control model.
	More precisely, we view the Markov chain as a nonlinear state-space model
\begin{equation}\label{eq:NLSS-model-intro}
			\phi_{t+1} = F(\phi_t,\alpha(\phi_t,U_{t+1}))
		\end{equation}
where $\{U_{t+1}\}_{t\in\bN}$ is an independent and identically distributed (i.i.d.) process valued in a measured space $\cU$, $F\colon\cX\times\cV\to\cX$ is a locally Lipschitz update function between smooth manifolds $\cX, \cV$ and $\alpha\colon\cX\times\cU\to\cV$ is a measurable, possibly discontinuous function. 
When $F$ is nonsmooth, we call \eqref{eq:NLSS-model-intro} a nonsmooth state-space model. The connections that we rely on between the irreducibility, aperiodicity and T-chain property of the Markov chain and an underlying deterministic control model have been recently established~\cite{gissler2024irreducibility}, relaxing the assumptions in previous work~\cite{chotard2019verifiable} that the state space of the chain is an open subset of an Euclidean space and $F$ is continuously differentiable. This latter work was already a generalization of the case where $\alpha(\phi_t,U_{t+1})=U_{t+1}$ and $F$ is smooth, i.e., infinitely differentiable~\cite{meyn1991asymptotic}.

While part of the methodology we follow relies on the results presented in \cite{gissler2024irreducibility}, we introduce here two other generic and central techniques for the analysis.

Like in many practically used algorithms (in contrast to toy algorithms), different update mechanisms can be turned on and off in CMA-ES by some specific hyperparameter settings (like learning rates) resulting in different algorithm variants with varying number of state variables.
Our aim is to analyze all algorithm variants without repeating the similar mathematical analysis for each of them.
Hence, in order to have a single proof, we introduce the notions of {\it projected} and {\it redundant} Markov chains. 
		Specifically, we consider a Markov chain $\{(\phi_t,\xi_t)\}_{t\in\bN}$ valued in the manifold $\cX\times\cY$ with
		\begin{equation}\label{eq:redundant-NLSS-model-intro}
			(\phi_{t+1},\xi_{t+1}) = \tilde{F}((\phi_t,\xi_t),\tilde\alpha((\phi_t,\xi_t),U_{t+1}))
		\end{equation}
		where $\{\phi_t\}_{t\in\bN}$ obeys \eqref{eq:NLSS-model-intro}, 
			and $\tilde F\colon\cX\times\cY\times\cV\to\cX\times\cY$ and $\tilde\alpha\colon\cX\times\cY\times\cU\to\cX\times\cY$ 
			satisfy the same assumptions as $F$ and $\alpha$, respectively. 
			\del{Besides w}We suppose then\del{for the projection $\Pi_\cX$} that
		\begin{equation}
			\Pi_\cX \circ \tilde F \left( (\phi,\xi) , \tilde \alpha ((\phi,\xi),u) \right) = F(\phi,\alpha(\phi,u))
		\end{equation}
		for every $(\phi,\xi,u)\in\cX\times\cY\times\cU$, where $\Pi_\cX\colon\cX\times\cY\to\cX$ is the canonical projection of $\cX\times\cY$ on $\cX$. The Markov chain $\{(\phi_t,\xi_t)\}_{t\in\bN}$ is said to be redundant, whereas $\{\phi_t\}_{t\in\bN}$ is said to be projected.
		We derive similar tools as in \cite{gissler2024irreducibility} to analyze the projected Markov chain $\{\phi_t\}_{t\in\bN}$ by investigating the redundant control model \eqref{eq:redundant-NLSS-model-intro}.

\paragraph{Contributions} Overall the contributions of this paper are twofold. 

		On the one hand, we provide two generic tools to analyze the irreducibility, aperiodicity and topological properties of complex nonsmooth state-space models.
			First, we extend the methodology to investigate Markov chains following \eqref{eq:NLSS-model-intro} with locally Lipschitz updates on smooth manifolds in order to be able to deduce irreducibility, aperiodicity and T-chain property from a redundant chain to a projected chain.
		Second, we show how to transfer the analysis of nonsmooth state-space models following \eqref{eq:NLSS-model-intro} from smooth manifolds to \emph{nonsmooth} manifolds, as long as they can be continuously transformed into smooth manifolds.
	
	On the other hand, using the developed tools, we establish the irreducibility, aperiodicity, and T-chain property of a Markov chain defined by the normalization of states  of CMA-ES when minimizing a scaling-invariant function.
	Our results include most of the relevant hyperparameter settings,
	some of them described by separate Markov chains.
		The proven properties constitute an essential step for a proof of linear convergence of CMA-ES.

	\paragraph{Organization}
	In \Cref{sec:main}, we present the update equations behind CMA-ES and define a class of normalized Markov chains associated to the algorithm when minimizing scaling-invariant functions. In \Cref{sec:IrrApp}, we state our first main result that these Markov chains are irreducible, aperiodic T-chains.
	In \Cref{sec:NLSS}, we state and prove results on the irreducibility, aperiodicity and topological properties of nonlinear state-space models.
	In \Cref{sec:proofs}, we apply the results exposed in \Cref{sec:NLSS} to the normalized Markov chain defined earlier ant prove the main result of \Cref{sec:IrrApp}. For the sake of readability, some proofs are delayed and presented in \Cref{app:proofs-steadily} and \Cref{app:proofs-controllability}.

	\paragraph{Notations}
	Throughout this paper, we use the following notations:
	$\bN$, $\bN^*$, $\bR$, $\bR_{+}$, $\bR_{++}$
	for the sets of nonnegative integers, positive integers, real numbers, nonnegative real numbers, and positive real numbers, respectively.
	Unless stated otherwise, for $n\in\bN^*$ and any vector $x\in\bR^n$, $\|x\|$ denotes the Euclidean norm of $x$.
	The set of real symmetric matrices of size $d\times d$ is denoted $\Sd$, and its subsets of positive semi-definite matrices and of positive definite matrices are denoted $\Sdp$ and $\Sdpp$, respectively.
	Given a positive integer $n$, $\mathfrak{S}_n$ represents the set of permutations of $\{1,\dots,n\}$, and its cardinality is denoted $n!$.
	The differential application of a function $F$ at a point $x$ is denoted $\mathcal{D}F(x)$, and the Clarke derivative of $F$ at $x$ is denoted $\partial F(x)$.
	We use the notations $\Argmin f$ and $\Argmax f$ for the sets of global minima and global maxima of $f$, respectively. When unique global minimum and maximum exist, we denote them as $\argmin f$ and $\argmax f$, respectively.
	For any sequence $\{v_k\}_{k\in\bN^*}$ and any $k\in\bN^*$, we set $v_{1:k}=(v_1,\dots,v_k)$. 
	For a topological space $\cX$, we denote $\cB(\cX)$ the Borel $\sigma$-field of $\cX$, which makes $\cX$ a measured space. If $\mu$ is a measure on $\cB(\cX)$ and $\nu$ is a measure on $\cB(\cY)$, we denote $\mu\otimes\nu$ the product measure of $\mu$ and $\nu$, which is a measure on $\cB(\cX\times\cY)$. Likewise, for $k\in\bN^*$, we denote $\mu^{\otimes k}$ the measure product of $\mu$ by itself $k$ times, as a measure on $\cB(\cX^k)$.

	\section{\del{Main results}\del{Irreducibility, aperiodicity and topological properties of a Markov chain underlying CMA-ES}%
		\new{Definition of Markov chains arising from \nnew{a}\del{the} normalization of CMA-ES}}\label{sec:main}
	We present in this section the CMA-ES algorithm and define \del{associated }normalized Markov chains---candidates to be stable---associated to the algorithm. We explain the connection between the stability of these Markov chains and the convergence of the algorithm, motivating thus why the irreducibility, aperiodicity and topological properties of the Markov chains that we study in the paper \new{are}\del{is} an important part for obtaining a convergence proof of CMA-ES.
	\del{
	\del{As an application of the theory of stochastic processes to linearly convergent algorithms for black-box optimization, the main result of this section}%
	This section consists in\del{ proving the irreducibility together with topological properties of a discrete-time Markov process underlying the CMA-ES algorithm.}
	\new{the definition of a normalized process underlying CMA-ES.}
	 In this section, we \del{first }introduce in \Cref{sec:algo-cmaes} the intuition and mathematical equations of \del{the }CMA-ES\del{ algorithm}, then we \del{motivate}\new{define} in \Cref{sec:def-normalized-chain} \del{the consequence of the irreducibility of }an underlying \new{normalized} \del{(discrete-time) }stochastic process\del{ about the convergence of CMA-ES}.
	 \new{We prove then that this process is a time-homogeneous Markov chain when the objective function is scaling-invariant.} 
	\del{and finally in}\new{In the next} \Cref{sec:IrrApp}, we state our main theorem \new{that this process is an irreducible aperiodic T-chain}.
	This results in one milestone towards establishing the linear convergence of the algorithm\new{, as explained in \Cref{sec:def-normalized-chain}}.
	}
	
		\subsection{Presentation of CMA-ES}\label{sec:algo-cmaes}
		
		The covariance matrix adaptation evolution strategy (CMA-ES) is an iterative algorithm which aims to approximate a problem solution 
		\begin{equation}
			\tag{P} \label{eq:optim-problem}
			x^* \in \underset{x\in\bR^d}{\Argmin}~ f(x) 
		\end{equation}
		where $d\in\bN^*$ is the dimension of the problem, and $f\colon\bR^d\to\bR$ is \del{called }the objective function. 
		A vector $x^*\in\bR^d$, solution to \eqref{eq:optim-problem}, is called a global \del{optimum}\nnew{minimum} of $f$.
		The CMA-ES attempts to approach $x^*$
			by \nnew{successively}
			sampling, \nnew{for iterations $t\in\bN$,} new candidate solutions
			from a multivariate normal probability distribution $\mathcal{N}(\X_t,\sigma_t^2\covmat_t)$.
			\del{, where $t\in\bN$ denotes the current iteration of the algorithm.}%
		The vector $\X_t\in\bR^d$ is the current mean
		of the distribution and we specifically desire that $f(\X_t)$ converges to the essential infimum of $f$.
		The positive real number $\sigma_t>0$ is the current stepsize, and the symmetric positive definite matrix $\covmat_t\in\Sdpp$ is referred to as the current covariance matrix.
		For our analysis, we generalize the assumption that the\del{ sampling} distribution of the candidate solutions is multivariate normal.
		
		The parameters of the sampling distribution are updated using two cumulation paths $p_t^\sigma,p_t^c\in\bR^d$, 
		which implement a weighted moving average of the steps followed by the mean.
		
		More precisely, the algorithm works as follows. 
		At iteration $t\in\bN$, given $\X_t\in\bR^d$, $\sigma_t>0$, $\covmat_t\in\Sdpp$, and $p_t^\sigma,p_t^c\in\bR^d$, 
		we generate independent identically distributed (i.i.d.) samples $U^1_{t+1},\dots,U^\lambda_{t+1}$ following \del{the}\nnew{a} sampling distribution\done{}\niko{we have silently changed the meaning of "sample distribution", before (like in the previous paragraph) it was used for the parametrized distribution with which we sample the candidate solutions.} $\pUd$ in $\bR^d$ \del{(for CMA-ES it is the standard normal distribution in $\bR^d$)}\done{we refer as CMA-ES with any $\pUd$} 
		and independently of $(\X_t,\sigma_t,\covmat_t,p_t^\sigma,p_t^c)$. 
		\new{Usually, the distribution $\pUd$ is the standard normal distribution in $\bR^d$. However, throughout the paper we will refer to CMA-ES as the algorithm presented in this section with a general and abstract sampling distribution $\pUd$.}
		We compute then $\lambda$ candidate solutions 
		\begin{equation}
			\label{eq:candidates-cma}
		x_{t+1}^i\coloneqq \X_t + \sigma_t \sqrt{\covmat_t} U_{t+1}^i\text{ for }i=1,\dots,\lambda \enspace,
		\end{equation}
		and rank them \wrt\ their $f$-values. Formally, we define a permutation $s_{t+1}\in\mathfrak{S}_\lambda$ satisfying%
		\begin{equation}\label{eq:permutation-def}
			f\left(x_{t+1}^{s_{t+1}(1)}\right)\leqslant  \dots \leqslant f\left(x_{t+1}^{s_{t+1}(\lambda)}\right) \enspace.
		\end{equation}
		When
		$f\left(x_{t+1}^i\right)  = f\left(x_{t+1}^j\right) $,	
			we impose for uniqueness
		$s_{t+1}^{-1}(i) < s_{t+1}^{-1}(j)$ if $i<j$\del{ in $\{1,\dots,\lambda\}$}.
		We say that we have \emph{neutral selection} when, \del{we assume }\nnew{instead of \eqref{eq:permutation-def},}\del{ that}
		\del{When }the permutation $s_{t+1}$ is independent of the samples $U_{t+1}^i$ for all $t\in\bN$\del{, w}. 
		This is the case, for example, when the permutation is fixed for all $t$, or when $f(x_{t+1}^i)$ is independent of its argument or independent of $U_{t+1}^i$.

		The mean is moved towards the best solutions,
		\del{The mean}and is updated by applying the function $\Fx$ defined as
		\begin{equation}\label{eq:Fm}
			\begin{array}{rl}
				\Fx\colon (m,v) \in \bR^d\times\bR^d & \mapsto m + c_m v \enspace, \del{\to\bR^d \\}
			\end{array}
		\end{equation}
		given a fixed learning rate $c_m>0$ (by default $c_m=1$). \del{More p}Precisely,
		the \del{update}\nnew{mean} obeys 
		\begin{equation}\label{eq:m}
			\X_{t+1} =
			\Fx\left(\X_t,\sigma_t\sqrt{\covmat_t}\sum_{i=1}^\mu \wi U_{t+1}^{s_{t+1}(i)}\right)
			\new{=  \X_t + c_m \sigma_t\sqrt{\covmat_t}\sum_{i=1}^\mu \wi U_{t+1}^{s_{t+1}(i)}} \enspace,
		\end{equation}
		where $\wi[1]\geqslant\dots\geqslant\wi[\mu]>0$ are weights such that $\sum_{i=1}^\mu\wi =1$, 
		and $\sqrt{\covmat_t}$ is \nnew{the} symmetric positive definite square root of $\covmat_t$.

	We introduce the function to update the paths $p_t^\sigma,p_t^c\in\bR^d$. Given a decay factor $c\in(0,1]$, $\Fpstar$ is defined as
	\begin{equation}\label{eq:p}
		\begin{array}{rl}
			\Fpstar \colon (p,v) \in \bR^d\times\bR^d & \mapsto (1-c) p +\sqrt{c(2-c)\mueff} v
		\end{array}
	\end{equation}
	where $\mueff = 1/ \|\vwm\|^{2}$, with $\vwm=(\wi[1],\dots,\wi[\mu])^\top$. 
	The closer the decay factor $c$ is to zero, 
	the more the updated path depends on the previous path due to the term $ (1-c) p$. 
	Conversely, when $c=1$, the updated path is collinear to \del{$v$ }and \new{only depends on $v$}.
	We set two decay factors, $c_\sigma,c_c\in(0,1]$, and use \eqref{eq:p} to update two cumulation paths, one for updating the stepsize and the other for the rank-one update of the covariance matrix (see below).
		\del{The cumulation paths obey}%
		We \nnew{update}\del{set}
		\begin{equation}\label{eq:ps}
			p_{t+1}^\sigma = \new{ (1-c_\sigma)p_t^\sigma + \sqrt{c_\sigma(2-c_\sigma)\mueff} \sum_{i=1}^\mu \wi U_{t+1}^{s_{t+1}(i)}} = \Fpsigma\left(p_t^\sigma,\sum_{i=1}^\mu \wi U_{t+1}^{s_{t+1}(i)}\right) \enspace,
		\end{equation}	
		and
		\begin{equation}\label{eq:pc}
			p_{t+1}^c = \new{ (1-c_c)p_t^c + \sqrt{c_c(2-c_c)\mueff} \sqrt{\covmat_t} \sum_{i=1}^\mu \wi U_{t+1}^{s_{t+1}(i)}} = \Fpc\left(p_t^c,\sqrt{\covmat_t}\sum_{i=1}^\mu \wi U_{t+1}^{s_{t+1}(i)}\right) \enspace.
		\end{equation}
		The second argument in the RHS of \eqref{eq:pc} is the same as in \eqref{eq:m} disregarding stepsize $\sigma_t$. \cref{eq:ps} additionally drops $\sqrt{\covmat_t}$. Consequently, when \new{$p_0^\sigma$ and} $U^i_{t+1}$ are \new{standard Gaussian}, then, 
		\new{under neutral selection,} $p_{t+1}^\sigma$ is a standard Gaussian vector too
		and, in particular, the length of $p_{t+1}^\sigma$ does not depend on its direction.\niko{Also, the variances are unchanged even when we do not assume isotropy!}
		The path $p_{t+1}^c$ from \eqref{eq:pc} 
		\new{maintains under neutral selection the covariance matrix $\covmat_t$ when $p_t^c$ has covariance matrix $\covmat_t$. The path} is commensurable with updating $\covmat_t$ and its expected length can strongly depend on its direction.
		\niko{Remark: this suggests that we may want to transform $p_t^c$ in \cref{eq:pc} like $\sqrt{\covmat_{t}}\sqrt{\covmat_{t-1}}^{-1}p_{t}^c$. It should be inconsequential in practice, but nevertheless the correct update?
		With this transformation, we have $p^c = \sqrt{C}p^\sigma$ if $c_c = c_\sigma$!?}%

		The stepsize is updated using the path $p_{t+1}^\sigma$. 
		Considering an abstract measurable function  $\Gamma\colon\bR^d\to\bR_{++}$ that we call stepsize change, the update reads
		\begin{equation}\label{eq:stepsize-update}
			\sigma_{t+1} = \sigma_t \times \Gamma\left(  p_{t+1}^\sigma \right) \enspace.
		\end{equation}
		A standard stepsize change used in CMA-ES is the cumulative stepsize adaptation (CSA) where $\Gamma$ equals
		\begin{equation}\label{eq:CSA}
			\CSA(p) = \exp\left( \frac{c_\sigma}{d_\sigma} \left( \frac{\|p\|}{\bE\|\pUd\|} -1 \right) \right)  \enspace,
		\end{equation}
		where $\bE\|\pUd\|\coloneqq\bE\|\xi\|$ for a random variable $\xi$ distributed under $\pUd$.
		\done{define $\|\pUd\|$  }%
		When $\pUd$ is the standard normal distribution, \eqref{eq:CSA} increases the stepsize when $\| p^\sigma_t \|$ is larger than to be expected under \new{neutral} selection (assuming that $p_0^\sigma \sim \pUd $)
		and decreases the stepsize when $\| p^\sigma_t \|$ is smaller.
		When consecutive steps are taken in a similar direction, 
		the \new{expected} path is long \new{while} the same progress could be made in fewer iterations with larger steps.
		When consecutive steps are negatively correlated, the \new{expected} path is short and a smaller stepsize is advisable.
		A smooth alternative to \eqref{eq:CSA} \new{implementing the same idea }is~\cite{arnold2004performance}
		\begin{equation}\label{eq:CSA-mod}
			\CSAd(p) = \exp\left( \frac{c_\sigma}{2d_\sigma} \left( \frac{\|p\|^2}{\bE\|\pUd\|^2} -1 \right) \right)  \enspace.
		\end{equation}
		These two stepsize changes rely on the choice of damping parameter $d_\sigma>0$, which is chosen $\approx 1 + 2 \sqrt{\mueff/d}$ in the first case and
		\new{$\approx 1 + 2 \mueff/d$} in the second case.
		\del{The performance of both updates within CMA-ES is usually Yet}%
		\nnew{Empirically, $\CSA$ and $\CSAd$ show}\del{present} similar \del{empirical }\nnew{performance when used with CMA-ES}~\cite{gissler2023evaluation}.
		While the function $\CSA$ is the default stepsize change, 
		previous theoretical works on ES \nnew{also have} investigated\del{ only} $\CSAd$~\nnew{\cite{arnold2004performance, toure2023global}}.

		Last, we introduce the update function for the covariance matrix, which depends on the choice of learning rates $c_1,c_\mu\geqslant0$ such that $c_1+c_\mu\in [0,1]$:
		\begin{equation}\label{eq:FC}
			\begin{array}{rl}
				\Fc \colon \Sdpp\times \bR^d \times \Sdp & \to \Sdpp \\
				(\covmat,p,\mathbf M) & \mapsto (1-c_1-c_\mu) \covmat + c_1 pp^\top + c_\mu \mathbf M \enspace,
			\end{array}
		\end{equation}
		and the covariance matrix \nnew{is} update\nnew{d via}\del{ then obeys}
		\begin{multline}\label{eq:c-update}
			\covmat_{t+1} =  (1-c_1-c_\mu) \covmat_t + c_1 p_{t+1}^c [p_{t+1}^c]^\top + c_\mu \sqrt{\covmat_t} \sum_{i=1}^\mu \wic \left[U_{t+1}^{s_{t+1}(i)}\right]\left[U_{t+1}^{s_{t+1}(i)}\right]^\top \sqrt{\covmat_t}   \\
			= \Fc\left( \covmat_t , p_{t+1}^c , \new{\sqrt{\covmat_t}} \sum_{i=1}^\mu \wic \left[U_{t+1}^{s_{t+1}(i)}\right]\left[U_{t+1}^{s_{t+1}(i)}\right]^\top\new{\sqrt{\covmat_t}} \right)  \enspace, 
		\end{multline}
		where we define weights $\wic[1]\geqslant\dots\geqslant\wic[\mu]>0$ such that $\sum_{i=1}^\mu\wic =1$.
		Moreover, we assume throughout the paper that $0<c_1+c_\mu<1$.
		\nnew{This assumption will be essential\del{ later} in the proofs of \Cref{l:path-step-3}, \Cref{c:steadily-attracting-state-for-CMA} and \Cref{p:full-rank-control-matrix-cma}.}
		The \nnew{setting}\del{case of} $c_1+c_\mu=1$ is however \nnew{used}\del{acceptable} in practice when $\mu$ is large 
		and we believe that\del{could be included in our analysis}\del{it can also be analyzed} with further work
		\nnew{\del{all }our results could be proven for this case as well.}

		The term $c_1 pp^\top$ is called the rank-one update, 
		whereas the term $c_\mu \mathbf M$ is the rank-mu update since in \eqref{eq:c-update} we replace $\mathbf M$ by a matrix of rank $\min(\mu,d)$ almost surely 
		which satisfies a maximum likelihood condition for the best samples of the last iteration~\cite[Proposition 7]{hansen2014principled}.
		In practice, also negative weights are used for the rank-mu update of the covariance matrix~\cite{jastrebski2006improving,hansen2010benchmarking}. 
		However, since the updated covariance matrix \new{must be} positive definite, \new{the norm of the} vectors corresponding to negative weights must be \new{controlled}. 
		We do not consider negative weights in the present paper.
		
		\subsection{Assumptions}\label{sec:assumptions}
		Our\del{ convergence} analysis of CMA-ES relies on analyzing the stability of normalized Markov chains underly\-ing the algorithm. 
		The construction of these Markov chains assumes that the objective function is scaling-invariant.
		A function $f\nnew{\colon\bR^d\to\bR}$ \nnew{is said to be} \textit{scaling-invariant \wrt\ $x^*\in\bR^d$} when for all $x,y\in\bR^d$ and $\rho>0$:
		\begin{equation}\label{eq:SI}
			f(x+x^*)\leqslant f(y+x^*) \Leftrightarrow f(\rho x+x^*) \leqslant  f(\rho y+x^*) \enspace.
		\end{equation}
	
	The class of scaling-invariant functions has been of interest for the convergence analysis of different variants of ES~\cite{auger2016linear,toure2023global}, 
	and is related to the class of positively homogeneous functions~\cite{toure2021scaling}. 
	We make an additional technical assumption on the level sets of the objective function to avoid ties in \eqref{eq:permutation-def}, 
	\del{thereby ensuring almost surely the uniqueness of the permutation $s_{t+1}$,}%
	\del{see \Cref{p:ordering-normalized-offsprings} below}%
	\new{which will be useful to define lower semi-continuous density functions in \Cref{l:density-alpha}}. 
	Overall, we will use the following assumptions:
	\begin{assumptionF}
		\label{F1:negligible-level-sets} The objective function $f\colon\bR^d\to\bR$ is a strictly increasing transformation of a continuous function with Lebesgue negligible level sets.
	\end{assumptionF}
	\begin{assumptionF}
		\label{F2:scaling-invariant} The objective function $f\colon\bR^d\to\bR$ is scaling-invariant \wrt\ a point $x^*\in\bR^d$.
	\end{assumptionF}
%
	Instead of assuming \Cref{F1:negligible-level-sets}, 
	we can suppose \wilog\ that the function $f$ is continuous with Lebesgue negligible level sets
	since CMA-ES is invariant to increasing transforma\-tions of the objective function~\cite{auger2016these}.
	%
		Assumption \Cref{F2:scaling-invariant} is central in this analysis since it is required to define an \new{equivalent,} normalized Markov chain via \eqref{eq:normalized-chain-def} below.

		In order to \nnew{go beyond scaling-invariant functions,}\del{relax the assumption,} it might be possible to adopt another approach, considering for instance recent works on the convergence of evolution strategies that prove a drift condition on the state variables of the algorithm and \nnew{hence}\del{manage to prove} the convergence on composites of strongly convex functions with strictly increasing functions (for however so far the (1+1)-ES selection scheme only)~\cite{akimoto2022global,morinaga2024convergence}.
%

		\del{Additionally}\nnew{Furthermore}, the sampling distribution $\pUd$ should satisfy the following assumption 
		that allows in particular to characterize a density for the ranked candidate solutions, see \Cref{l:density-alpha}.
		\begin{assumptionP}
			\label{P1:density} The probability distribution $\pUd$ admits a continuous \del{positive}density $\densityd(\cdot)$ \wrt\ the Lebesgue measure on $\bR^d$ which is positive everywhere on $\bR^d$.
		\end{assumptionP}
%
		This assumption is satisfied by a multivariate standard normal distribution as used in 
		CMA-ES\del{ algorithm}. \done{check where we call CMA-ES without gaussian}
		We have moreover the following assumptions on the stepsize change $\Gamma$. 
		\begin{assumptionG}
			\label{G1:C1} The stepsize change $\Gamma\colon\bR^d\to\bR_{++}$ is a locally Lipschitz map and is differentiable at every nonzero vector of $\bR^d$.
		\end{assumptionG}
		\begin{assumptionG}
			\label{G2:unbounded}  \del{We have}\nnew{Given $c_c$ the cumulation parameter for the path in \eqref{eq:ps}, the}\del{The} function $\Gamma$ satisfies $\liminf\Gamma(p)>(1-c_c)^{-1}$ for $\|p\|$ to $+\infty$.\done{check}
		\end{assumptionG}
		\begin{assumptionG}
			\label{G3:G(0)<1} \del{We have}\new{The function $\Gamma$ satisfies} $\Gamma(0)<1$.
		\end{assumptionG}
		Assumption \Cref{G1:C1} is required to apply the results stated in \Cref{sec:deterministic-control-model} and 
		in particular to ensure that the analyzed process satisfies \nnew{the condition} \Cref{A5:C1-update} \del{from}\nnew{in} \Cref{sec:deterministic-control-model}
		\del{in order }%
		to obtain the irreducibility and aperiodicity of the Markov chain defined in \eqref{eq:normalized-chain-def}.
		\nnew{Assumptions} \Cref{G2:unbounded} and \Cref{G3:G(0)<1} are used in \Cref{p:steadily-attracting-state-for-CMA,p:full-rank-control-matrix-cma}, respectively.
		
		Assumptions \Cref{G1:C1}--\Cref{G3:G(0)<1} are satisfied by both stepsize changes, $\CSA$ and $\CSAd$, as stated in the following lemma.
		\begin{lemma}\label{lem:Gamma-CSA}
			Assume that $c_\sigma\in (0,1]$. Then, the stepsize change functions $\CSA$ and $\CSAd$, defined by \eqref{eq:CSA} and \eqref{eq:CSA-mod} respectively,  satisfy the assumptions \Cref{G1:C1}--\Cref{G3:G(0)<1}.
		\end{lemma}
		\begin{proof}
			\del{Straightforward.}\new{The proof is simple and left to the reader.}
		\end{proof}

		\subsection{Proving the stability of a normalized Markov chain leads to linear convergence}\label{sec:def-normalized-chain}
		Before stating our main results, we define a normalized Markov chain underlying the CMA-ES algorithm.
		\nnew{Let $(\Omega,\cF,\bP)$ be a probability space}. 
		\nnew{An event is an element $\mathsf{W}\in\cF$, and the probability of $\mathsf{W}$ is $\bP[\mathsf{W}]$.}
		\nnew{A random variable $U$ valued in a measured space $(\cU,\mathcal{U})$ is defined as a measurable function $U\colon\Omega\to U$},
		\nnew{and for $\mathsf{A}\in\mathcal{U}$, we identify $\bP[U\in\mathsf{A}]$ to $\bP[\{\omega\in\Omega\mid U(\omega)\in\mathsf{A}\}]$}.
		\del{We recall that a}\nnew{A} \del{Markov}\nnew{transition} kernel on a topological state space $\cX$ equipped with its Borelian $\sigma$-field $\cB(\cX)$ is an application $P\colon \cX\times\cB(\cX)\to\bR_+$ such that, for every $x\in\cX$, $\mathsf A\in\cB(\cX)\mapsto P(x,\mathsf A)$ is a probability measure, 
			and for every $\mathsf A\in\cB(\cX)$, $x\in\cX\mapsto P(x,\mathsf A)$ is a measurable map.
		Then,
		 a (time-homogeneous) Markov chain with transition kernel $P$ \new{on $(\cX,\cB(\cX))$} and initial probability distribution $\nu$ on $\cB(\cX)$ is a sequence of random variables $\Phi=\{\phi_t\}_{t\in\bN}$ \new{valued in $\cX$,} satisfying for every $t\in\bN$
		\begin{multline}
			\bP\left[  (\phi_0,\dots,\phi_t)\in \mathsf A_0\times\dots\times \mathsf A_t  \mid \phi_0\sim \nu \right] \\ 
			= \int_{\cX^{t+1}} \1\left\{ (x_0,\dots,x_{t-1})\in \mathsf A_0\times\dots\times \mathsf A_{t-1} \right\} 
			P(x_{t-1},\mathsf A_t)P(x_{t-2},\mathrm{d}x_{t-1})\dots P(x_0,\mathrm{d}x_1)\nu(\mathrm{d}x_0)  \nonumber 
		\end{multline}%
		\nnew{where for every probability measure $\nu$ on $\cB(\cX)$, we have equipped $(\Omega,\cF)$ with a probability measure $\bP[\cdot|\phi_0\sim\nu]$.}
		\del{where, for two events $\omega$ and $\omega'$, $\bP[\omega]$ denotes the probability of $\omega$ and $\bP[\omega\mid\omega']$ the probability of $\omega$ conditionally to $\omega'$.}%
			\del{When $\nu=\delta_x$ the Dirac distribution in a point $x\in\cX$, w}%
			\nnew{W}e \new{define the $t$-step transition kernel} \nnew{by} $P^t(x,\mathsf A)=\bP[\phi_t\in \mathsf A| \phi_0\sim\delta_x]$ for every $t\geqslant0$ and $\mathsf A\in\cB(\cX)$.

		The sequence $\{(\X_t,p_t^\sigma,p_t^c,\sigma_t,\covmat_t)\}_{t\in\bN}$ introduced in \Cref{sec:algo-cmaes} defines a time-homogeneous Markov chain on the state space $\bR^{3d}\times\bR_{++}\times\Sdpp$. This is immediate from the observation that the definition of $(\X_{t+1},p_{t+1}^\sigma,p_{t+1}^c,\sigma_{t+1},\covmat_{t+1})$ depends only on the previous state $(\X_t,p_t^\sigma,p_t^c,\sigma_t,\covmat_t)$ and the independent random input  $U^1_{t+1},\dots,U^\lambda_{t+1}$. 		
		However, when the mean \new{converges}\del{tends} to the optimum of the function,\del{ we expect} the stepsize $\sigma_t$, the covariance matrix $\covmat_t$ and the path $p_t^c$\del{ to tend} \new{converge} to $0$.\del{ too.}
		\new{Therefore, this Markov chain is not Harris recurrent \nnew{(it does}\del{would} not revisit every \nnew{neighborhood} of any state infinitely many times)}.
		\new{Yet, as illustrated later in \eqref{eq:log-progress-cma} and \Cref{p:linear-convergence}, our methodology to prove linear convergence \cite{auger2016linear} relies \nnew{on}\del{in particular on applying} a Law of Large Numbers which motivates to have a positive Harris recurrent Markov chain (with a stationary probability distribution) and, more generally, \nnew{a chain} stable enough to apply}
		\del{and we }\del{would not be able to}\del{ cannot apply}%
		\nnew{an} ergodic theorem\del{s}~\cite[Theorems 13.0.1]{meyn2012markov} \nnew{and satisfy a Law of Large Numbers~\cite[Theorem 17.0.1]{meyn2012markov}.}
		\del{ to confirm a stationary distribution.}%
			\del{This is \del{why we}\new{a motivation to}}\del{For this reason,}\nnew{Therefore, we} define a normalized process, \new{candidate to have a stationary distribution}, underlying the CMA-ES algorithm.
		Consider $R\colon\Sdpp\to\bR_{++}$  a normalization function
		\new{which is}\del{, for which we assume the following }
		\begin{assumptionR}
			\label{R1:homogeneous}\del{ the function $R$ is} (positively) homogeneous \new{with degree 1}, i.e., for every $\mathbf A\in\Sdpp$ and $\rho >0, R(\rho \mathbf A) = \rho R(\mathbf A)$,
		\end{assumptionR}	
		\begin{assumptionR}
			\label{R2:C1}\del{ the function $R$ is} locally Lipschitz continuous,\del{ and it is} \done{continuously locally Lipschitz sounds strange - rather locally Lipschitz continuous ?}
		\end{assumptionR}	
		\begin{assumptionR}
			\label{R3:differentiable} differentiable on a nonempty open subset of $\Sdpp$.
		\end{assumptionR}
%
		Assumption \Cref{R1:homogeneous} is required to define a normalized Markov chain, see \eqref{eq:normalized-chain-def} below, as proven in \Cref{l:homogeneous-MC}. Assumption \Cref{R2:C1} is used to prove irreducibility and aperiodicity of the normalized chain, 
			\del{based on}\new{notably for the verification of	} condition \Cref{A5:C1-update} introduced in \Cref{sec:deterministic-control-model}.
		Later, \Cref{p:full-rank-control-matrix-cma} \new{uses}\del{relies on} \Cref{R3:differentiable} to prove a maximal rank condition.
		\del{ of the Jacobian of the update function of \new{the}\del{a} normalized chain.}%

		We give examples of normalization functions that satisfy these assumptions.

		\newcommand{\ith}[1][i]{$#1$-$\mathrm{th}$}
		\begin{proposition}\label{p:examples-R}
			The \ith[d] root of the determinant, $\det(\cdot)^{1/d}$,
			and the \ith\ largest eigenvalue, $\lambda_i(\cdot)$, for $i\in\{1,\dots,d\}$ counted with multiplicity,
			are functions defined on $\Sdpp$ that satisfy \Cref{R1:homogeneous}\nnew{--}\del{, \Cref{R2:C1} and }\Cref{R3:differentiable}.
		\end{proposition}
		\begin{proof}
The proof of the property \Cref{R1:homogeneous} is immediate from the linearity of the determinant as a function of the columns of the matrix and the definition of an eigenvalue. 
	For the properties \Cref{R2:C1} \nnew{and \Cref{R3:differentiable}}, we know that the determinant of a matrix is a polynomial function of the coefficients of the matrix~\cite[Section 0.3]{horn2013matrix}, hence it is infinitely differentiable and in particular is locally Lipschitz~\cite[Proposition and Corollary 2.2.1]{clarke1990optimization}.
	For the eigenvalues, $\lambda_i(\cdot)$ is locally Lipschitz on $\Sdpp$, as a consequence of Weyl's theorem~\cite[Corollary 4.3.15]{horn2013matrix}.
	Besides, $\lambda_i(\cdot)$ is infinitely differentiable on a neighborhood of any symmetric matrix with eigenvalues that have simple multiplicity~\cite[Theorem 5.3]{serre2010matrices}.
	\end{proof}
	
		Given \new{the CMA-ES Markov chain $\{(\X_t,p_t^\sigma,p_t^c,\sigma_t,\covmat_t)\}_{t\in\bN}$ defined in \Cref{sec:algo-cmaes} and} a normali\-zation function $R$, we\del{We} define the \textit{normalized \del{Markov }chain} $\Phi=\{(\Z_t,p_t,q_t,\ncovmat_t,r_t)\}_{t\geqslant1}$ as follows.\footnote{%
			\new{The definition of $q_t$ in \eqref{eq:normalized-chain-def} suggests transforming $p_t^c$ in \eqref{eq:pc} like $\covmat_{t}^{1/2}\covmat_{t-1}^{-1/2}p_{t}^c$ to avoid the time index $t-1$ in \eqref{eq:normalized-chain-def}.
			Then, $p^c_{t+1}$ would become equal to $\covmat_t^{1/2} p_{t+1}^\sigma$.} \nnew{We can prove unbiasedness for $p^\sigma$ \cite{hansen2014principled} and affine invariance \del{for $p^c$}\new{with $p^c$ and $c_\sigma=1$} \cite{auger2016these}\done{put reference}.} \niko{Do we have an equation that expresses the bias of $p_c$? Can $p_c$ become unbiased when we normalize the covariance matrix appropriately?}
		}
		For all $t \geq 1$, we set 
		\begin{equation}
			\label{eq:normalized-chain-def}
			\begin{array}{l}
				\Z_t = \frac{\X_t - x^*}{\sigma_t\sqrt{R(\covmat_t)} } ,\,
				 p_t = p_t^\sigma ,\,
				q_t = \frac{p_t^c}{\sqrt{R(\covmat_{t-1})}} ,\,
				{\ncovmat}_t = \frac{\covmat_t}{R(\covmat_{t})} ,\,
				r_t =\frac{R(\covmat_t)}{R(\covmat_{t-1})} \enspace.
			\end{array}
		\end{equation}
		\del{\del{Remark that s}\nnew{S}ince the definition of $q_t$ and $r_t$ requires the covariance matrix at time $t-1$, the time index of the chain $\Phi$ starts at $t=1$.}
		We prove below that when the objective function is scaling-invariant,	the normalized chain defined by \eqref{eq:normalized-chain-def} is a time-homogeneous Markov chain \nnew{that can be defined independently of the original Markov chain $\{(\X_t,p_t^\sigma,p_t^c,\sigma_t,\covmat_t)\}_{t\in\bN}$}. \del{Yet, w}We establish first that on scaling-invariant functions,\del{ \wrt\ zero}
		the permutation sorting the candidate solutions $\X_t+\sigma_t\sqrt{\covmat_t}U_{t+1}^i$ also sorts the \new{vectors} $\Z_t+\sqrt{\ncovmat_t}U_{t+1}^i$, for $i=1,\dots,\lambda$.

		\begin{lemma}
			\label{p:ordering-normalized-offsprings}
			Let $t\geqslant1$ and suppose that the objective function $f$ satisfies \Cref{F2:scaling-invariant}. 
			Let $s_{t+1}\in\mathfrak{S}_\lambda$ be a (random) permutation that sorts 
			the indices $i=1,\dots,\lambda$ \wrt\ the $f$-values of $\X_t+\sigma_t\sqrt{\covmat_t}U_{t+1}^i$.
			Then, $s_{t+1}$ \nnew{also} sorts\del{ increasingly} 
			the indices $i=1,\dots,\lambda$ \wrt\ the \nnew{$f$}\del{$f(x^*+\cdot)$}-values of $\nnew{x^*} + \Z_t+\sqrt{\ncovmat_t}U_{t+1}^i$\del{, where $\nnew{\fx}=f(\cdot+x^*)$}.
			Moreover, we can ensure the uniqueness of $s_{t+1}$ by imposing a tie-break (cf.\ \Cref{sec:algo-cmaes}).
			 \del{If moreover, \new{the function} $f$ satisfies \Cref{F1:negligible-level-sets} and the sampling \new{distribution }$\pUd$ satisfies \Cref{P1:density}, then such a permutation is (almost surely) unique.}%
		\end{lemma}
		\begin{proof}
			Let $i=1,\dots,\lambda$.\del{Then,} \new{By definition of $\Z_t$ and $\ncovmat_t$, }we obtain
			$$
			f\left(x^* +\Z_t+\sqrt{\ncovmat_t}U_{t+1}^{i}\right) = f\left(x^* + R(\covmat_t)^{-1/2}\sigma_t^{-1} \times \left[ \X_t - x^* + \sigma_t\sqrt{\covmat_t} U_{t+1}^i \right] \right).
			$$\niko{Note to understand this: from $\X_t=x^* + \sigma R(C) z_t$ we get
				$f(\X + \sigma \sqrt{C} U)
				= f(x^* + \sigma R(C) z + \sigma \sqrt{C} U)
				= f(x^* + \sigma R(C) (z + \sqrt{\Sigma} U))
				\stackrel{\text{rank}}{\sim} f(x^* + z + \sqrt{\Sigma} U)$}
			We conclude by using the definition of a scaling-invariant function \eqref{eq:SI}. 			
		\end{proof}
		\del{Now,}\new{From the previous lemma,} we deduce \nnew{in \Cref{t:homogeneous-MC} below} that the normalized chain defined in \eqref{eq:normalized-chain-def} is a time-homogeneous Markov chain\del{.} 
		\del{In addition, we prove that }%
		\nnew{that}\del{ it} can be defined independently of the original Markov chain. Indeed, given\del{if the normalization function} $R$ \del{a normalization function }satisfying \Cref{R1:homogeneous}, denote with a slight abuse of notation \nnew{(since we use the same notation as for \eqref{eq:normalized-chain-def} with however a different time index)}
			the \nnew{time-}homogeneous Markov chain 
			$\Phi=\{\phi_t\}_{t\geqslant0}=\{(\Z_t,p_{t},q_{t}, \ncovmat_t,r_t)\}_{t\geqslant 0}$ 
			defined via \new{$\phi_0 \in  \cY=(\bR^d)^3\times R^{-1}(\{1\})\times\bR_{++}$} 
			(where $R^{-1}(\{1\}) = \{ \ncovmat \in \Sdpp  : R(\ncovmat)=1\}$) 
			and the following recursion\niko{I aligned the = symbols in the equation (instead of the leftmost symbol)}
	\begin{equation}
				\begin{aligned}
					\Z_{t+1}& = { \cfrac{\Z_t + c_m \sqrt{\ncovmat_t} \w_m^\top \Utt  }{\sqrt{r_{t+1}}\,\Gamma(p_{t+1})} =} \cfrac{\Fx(\Z_t,\sqrt{\ncovmat_t} \w_m^\top \Utt )}{\sqrt{r_{t+1}}\,\Gamma(p_{t+1})  }  \\
					p_{t+1} &= \nnew{(1-c_\sigma) p_t + \sqrt{c_\sigma(2-c_\sigma)\mueff} \w_m^\top \Utt}  = \Fpsigma(p_t, \w_m^\top \Utt )   \\
					q_{t+1} &= \Fpc(r_t^{-1/2}q_t, \sqrt{\ncovmat_t} \w^\top_m \Utt )  \\
					\ncovmat_{t+1} &= r_{t+1}^{-1}{\Fc\!\!\left(\ncovmat_t,q_{t+1},
						\nnew{\sqrt{\ncovmat_t}}\sum_{i=1}^\mu \wic \left[  U_{t+1}^{s_{t+1}(i)} \right]
						\left[ U_{t+1}^{s_{t+1}(i)} \right]^\top\nnew{\sqrt{\ncovmat_t}}\right)}  \\
					r_{t+1} &= R\circ \Fc \!\!\left(\ncovmat_t,q_{t+1},
					\nnew{\sqrt{\ncovmat_t}}
					\sum_{i=1}^\mu \wic 
					\left[  U_{t+1}^{s_{t+1}(i)} \right]\left[ U_{t+1}^{s_{t+1}(i)} \right] ^\top
					\nnew{\sqrt{\ncovmat_t}}\right)   
			\end{aligned}
			\label{eq:normalized-chain-updates}
		\end{equation}
\del{where}\nnew{with} $\mathbf{U}=\{U_{t+1}\}_{t\in\bN}$ \del{is }an i.i.d.\ process independent of $\phi_0$ 
with $U_1=(U_1^1,\dots,U_1^\lambda)\sim(\pUk{d})^{\otimes\lambda}$,  
\nnew{and} $s_{t+1}$ \del{is }the (almost surely unique) permutation that sorts the $f(\Z_t+\sqrt{\ncovmat_t}U_{t+1}^i)$, $i=1,\dots,\lambda$.
\nnew{Moreover, $U_{t+1}^{s_{t+1}}$ denotes the collection of vectors $(U_{t+1}^{s_{t+1}(1)},\dots,U_{t+1}^{s_{t+1}(\lambda)})$.}
\del{ and the notation $\Utt = (U_{t+1}^{s_{t+1}(1)}, \ldots,U_{t+1}^{s_{t+1}(\lambda)} )  \in (\R^d)^\lambda$ has been used. }%
\new{Remark that in \eqref{eq:normalized-chain-updates}, the update of the covariance matrix $\ncovmat_{t+1}$ writes 
$$
\ncovmat_{t+1} = \frac{\tilde{\ncovmat}_{t+1}}{R(\tilde{\ncovmat}_{t+1})}
$$
where $\tilde{\ncovmat}_{t+1}$ is the covariance matrix to which we apply the rank-one and rank-mu updates, i.e.,
\begin{multline}\label{tilda-matrix}
		\tilde{\ncovmat}_{t+1} \coloneqq 
		\Fc\left(\ncovmat_t,q_{t+1},
		{\sqrt{\ncovmat_t}}\sum_{i=1}^\mu 
		\wic \left[  U_{t+1}^{s_{t+1}(i)} \right]\left[ U_{t+1}^{s_{t+1}(i)} \right]^\top
		{\sqrt{\ncovmat_t}}\right) \\ = (1-c_1-c_\mu) \ncovmat_t + c_1 q_{t+1}(q_{t+1})^\top+ c_\mu {\sqrt{\ncovmat_t}}\sum_{i=1}^\mu 
		\wic \left[  U_{t+1}^{s_{t+1}(i)} \right]\left[ U_{t+1}^{s_{t+1}(i)} \right]^\top
		{\sqrt{\ncovmat_t}} \enspace,
	\end{multline}
	where $\Fc$ is defined via \eqref{eq:FC}. Similarly $r_{t+1}$ can be expressed using $\tilde{\ncovmat}_{t+1}$ as
 $$
					r_{t+1} = R(\tilde{\ncovmat}_{t+1}) \enspace.  
$$
}
The update of $\phi_t$ \del{defined }in \eqref{eq:normalized-chain-updates} defines a function $F_\Phi$ such that 
\begin{equation}\label{eq:Fphi}
\phi_{t+1} = F_\Phi(\phi_t,U_{t+1}^{s_{t+1}}) \enspace.
\end{equation}
\del{We show\del{ in the next proposition} that the normalized chain \eqref{eq:normalized-chain-def} obeys \eqref{eq:Fphi} and \del{can thus be}\new{is thus} defined\done{}\niko{isn't it rather "is defined" or "can be written"?} independently of the \del{original CMA-ES Markov chains}\new{Markov kernel underlying the update of CMA-ES}.\done{I changed the formulation (the plural is because we have a different MC for every possible initialization)}\niko{why plural, we only have one original (parametrized) chain, no?} }
\del{This is summarized by the next proposition.}
We prove in the next proposition that if $f$ is scaling-invariant, the normalized chain defined in \eqref{eq:normalized-chain-def} can be defined independently of the original CMA-ES chain via the recursion \eqref{eq:Fphi} provided it is initialized properly. 
\nnew{While the normalized process \eqref{eq:normalized-chain-def} imposes $t\geq1$, the next proposition defines this process via the recursion \eqref{eq:normalized-chain-updates} and thus allows to start with any time index.}
\del{Remark that the definition of the normalized process \eqref{eq:normalized-chain-def} imposes\del{ us} to start indexing it with $t=1$. However, the next proposition allows a definition of this process via the recursion \eqref{eq:normalized-chain-updates} and thus to start with any time index.}
\done{clarify the $t\geq 1$ and $t \geq 0$}
		\begin{proposition}\armand{I removed the bar on $\phi$}
			\label{t:homogeneous-MC}\label{l:homogeneous-MC}
			Suppose that the objective function $f$ satisfies \del{\Cref{F1:negligible-level-sets}-}\Cref{F2:scaling-invariant} and 
			that the normalization function $R$ satisfies \Cref{R1:homogeneous}.\del{, and that the sampling \new{distribution} $\pUd$ is such that \Cref{P1:density} holds.}
			Let $\{(\X_t,p_t^\sigma,p_t^c,\covmat_t,\sigma_t)\}_{t\in\bN}$ be the chain associated to CMA-ES defined in \Cref{sec:algo-cmaes} 
			and $\new{\del{\Bar}{\Phi}=\{\del{\bar} \phi_t\}}_{t\geqslant1}=\{(\Z_t,p_{t},q_{t}, \ncovmat_t,r_t)\}_{t\geqslant1}$ be 
			the normalized process defined \new{via}\del{  in} \eqref{eq:normalized-chain-def} \nnew{for $t\geqslant1$}. 
				Then
			$\new{\del{\bar} \Phi}$ is a time-homogeneous Markov chain valued in the state space $\cY=(\bR^d)^3\times R^{-1}(\{1\})\times\bR_{++}$ that satisfies 
			$$
				\del{\bar} \phi_1 = 
				\left( \cfrac{m_1}{\sigma_1 \sqrt{R(\covmat_1)}}, p_1^\sigma, \cfrac{p_1^c}{\sqrt{R(\covmat_0)}},  \cfrac{\covmat_1}{R(\covmat_1)}, \cfrac{R(\covmat_1)}{R(\covmat_0)} \right)  
			$$
			 \nnew{and for $t \geqslant 1$ we have}
\new{			
			$$
				\del{\bar}{\phi}_{t+1} = F_\Phi({\del{\bar} \phi_t}, U_{t+1}^{s_{t+1}})
			$$
			where $F_\Phi$ is the function in \eqref{eq:Fphi} defined via the equations \eqref{eq:normalized-chain-updates}, }%
		$s_{t+1}\in\mathfrak{S}_\lambda$ is \del{the (almost surely unique)}a permutation that 
			sorts\footnote{\new{We always sort increasing and,} as explained in \Cref{sec:algo-cmaes}, in case of a tie between the $f$-values of the candidate solutions of indices $i$ and $j$ with $i<j$, we impose $s_{t+1}^{-1}(i)<s_{t+1}^{-1}(j)$ to ensure the uniqueness of the permutation $s_{t+1}$.\niko{Should better move to the notations section?}} 
			\del{increasingly }the $f(x^*+\Z_t+\sqrt{\ncovmat_t}U_{t+1}^i)$, $i=1,\dots,\lambda$,
			and $\mathbf{U}=\{U_{t+1}\}_{t\geqslant1}$ is \del{a}\new{the} i.i.d.\ process 
			\new{used to define $\{(\X_t,p_t^\sigma,p_t^c,\covmat_t,\sigma_t)\}_{t\in\bN}$\del{ minus the first term $U_1$}}%
				\done{there is no $U_0$ for the initial chain as well}\del{(the same as defined in \Cref{sec:algo-cmaes})}, 
			\new{thus} independent 
			of \del{$\phi_0$}\nnew{$\del{\bar} \phi_1$.}%
			\del{with $U_2=(U_2^1,\dots,U^\lambda_2)\sim(\pUk{d})^{\otimes\lambda}$
			\del{ and $\Fx$, $\Fpsigma$, $\Fpc$ and $\Fc$ are the update functions defined in \eqref{eq:Fm}, \eqref{eq:p} and \eqref{eq:FC}}.
		}%
	\end{proposition}
\begin{proof}
By \Cref{p:ordering-normalized-offsprings}, 
	it is sufficient to show that \eqref{eq:normalized-chain-updates} holds 
	for every $t\geqslant1$ in order to prove 
	that $\Phi$ is a time-homogeneous Markov chain. 
Let $t\geqslant1$, and consider the matrix $\tilde{\ncovmat}_{t+1}$ defined in \eqref{tilda-matrix}.
Since $\Fc$ is homogeneous \wrt\ its first variable,
 positively homogeneous of degree 2 \wrt\ the second variable,
 using \eqref{eq:normalized-chain-def} 
 and the definition of $\covmat_{t+1}$ in \eqref{eq:c-update} we find
 \del{note that, by \eqref{eq:normalized-chain-def}, }
 $$
 	\tilde{\ncovmat}_{t+1} 
 	= R(\covmat_t)^{-1}\covmat_{t+1} \enspace.
 $$
By \new{the property }\Cref{R1:homogeneous} \del{we have}\new{applied to the previous equation we obtain}
	$
		R(\tilde{\ncovmat}_{t+1}) = R(\covmat_t)^{-1} R(\covmat_{t+1})=r_{t+1}
	$. 
Furthermore, \del{we have that}\new{the following holds}
\begin{align*}
	\Z_{t+1} = &  R(\covmat_{t+1})^{-1/2} \sigma_{t+1}^{-1} \times (\X_{t+1}\nnew{{}- x^*)}\\
	= &  r_{t+1}^{-1/2} R(\covmat_t)^{-1/2} \sigma_t^{-1} \Gamma(p_{t+1}^\sigma)^{-1} \times \left[ \X_t \nnew{{}- x^*} + c_m \sigma_t\sqrt{\covmat_t}\sum_{i=1}^\mu \wi U_{t+1}^{s_{t+1}(i)}\right] \\
	= &  r_{t+1}^{-1/2}\Gamma(p_{t+1})^{-1} \times \left[ \Z_t + c_m \sqrt{\ncovmat_t} \sum_{i=1}^\mu \wi U_{t+1}^{s_{t+1}(i)} \right] \new{ = \frac{\Fx(\Z_t,\sqrt{\ncovmat_t} \w_m^\top \Utt )}{\sqrt{r_{t+1}}\Gamma(p_{t+1})  }} \enspace,
\end{align*}
\nnew{where $\Fx$ is defined via \eqref{eq:Fm}.}
Moreover, 
\begin{align*}
	\ncovmat_{t+1} = & ~ R(\covmat_{t+1})^{-1} \covmat_{t+1} \\
	= &~R(\covmat_{t+1})^{-1} \left[ (1-c_1-c_\mu) \covmat_t + c_1 (p_{t+1}^c)(p_{t+1}^c)^\top  + c_\mu \sum_{i=1}^\mu \wic \left(\sqrt{\covmat_t}U_{t+1}^{s_{t+1}(i)}\right)\left(\sqrt{\covmat_t}U_{t+1}^{s_{t+1}(i)}\right)^\top \right] \\
	= &~ r_{t+1}^{-1}\times\left[  (1-c_1-c_\mu) \ncovmat_t + c_1 (q_{t+1})(q_{t+1})^\top+ c_\mu \sum_{i=1}^\mu \wic \left(\sqrt{\ncovmat_t}U_{t+1}^{s_{t+1}(i)}\right)\left(\sqrt{\ncovmat_t}U_{t+1}^{s_{t+1}(i)}\right)^\top \right] \\
	& = r_{t+1}^{-1} \Fc\left(\ncovmat_t,q_{t+1},\nnew{\sqrt{\ncovmat_t}}\sum_{i=1}^\mu \wic \left[  U_{t+1}^{s_{t+1}(i)} \right]\left[ U_{t+1}^{s_{t+1}(i)} \right]^\top\nnew{\sqrt{\ncovmat_t}}\right)
\end{align*}
Finally, 
\begin{align*}
	q_{t+1} & = R(\covmat_t)^{-1/2} p_{t+1}^c \\
	& = R(\covmat_t)^{-1/2} (1-c_c) p_t^c + \sqrt{\mueff c_c(2-c_c)} R(\covmat_t)^{-1/2} \covmat_t^{1/2} \sum_{i=1}^\mu \wi U_{t+1}^{s_{t+1}(i)}  \\
	& = r_t^{-1/2} (1-c_c) q_t + \sqrt{\mueff c_c(2-c_c)} \ncovmat_t^{1/2} \sum_{i=1}^\mu \wi U_{t+1}^{s_{t+1}(i)}  \new{= \Fpc(r_t^{-1/2}q_t, \sqrt{\ncovmat_t} \w^\top_m \Utt )\enspace,} 
\end{align*}
\nnew{where $\Fpc$ is defined via \eqref{eq:p}.}
\end{proof}

\nnew{Now that we formally prove that the normalized chain defined in \eqref{eq:normalized-chain-def} is a time-homogeneous Markov chain when the algorithm optimizes a scaling-invariant function, we}
\del{We}\new{recapitulate}\del{explain here} how \del{the}\new{its} stability\del{ of the normalized chain \eqref{eq:normalized-chain-def}} is connected to the linear convergence of CMA-ES
on scaling-invariant functions. For $T\in\bN$, \nnew{using the definition of the normalized Markov chain in \eqref{eq:normalized-chain-def} and the definition of the stepsize change \eqref{eq:stepsize-update} we obtain}\del{ we have}\del{We are interested to prove the (asymptotic) linear convergence of the CMA-ES algorithm towards the minimum of a scaling invarian\nnew{t}\del{ce} function.}
\begin{align}
\frac{1}{T} \log \frac{\|\X_T - {x^*}\|}{\|\X_0 - {x^*}\|} & = \frac{1}{T} \sum_{t=0}^{T-1} [ \log\|\X_{t+1} - {x^*}\| -\log\|\X_t - {x^*}\| ] \nonumber \\
& = \frac{1}{T} \sum_{t=0}^{T-1} \left(  \log \left( \|\Z_{t+1}\| \sqrt{R(\covmat_{t+1})} \sigma_{t+1}\right) -  \log \left( \|\Z_{t}\| \sqrt{R(\covmat_{t})} \sigma_{t} \right) \right) \nonumber \\
& = \frac{1}{T}\sum_{t=0}^{T-1} \left( \log\|\Z_{t+1}\|-\log\|\Z_t\| + \log\frac{\sigma_{t+1}}{\sigma_t} +\frac{1}{2}\log \frac{R(\covmat_{t+1})}{R(\covmat_t)} \right)\nonumber \\
& = \frac{1}{T} \sum_{t=0}^{T-1} \left( \log\|\Z_{t+1}\| -\log\|\Z_t\| + \log \Gamma(p_{t+1}) + \frac{1}{2} \log r_{t+1} \right) \enspace. \label{eq:log-progress-cma}
\end{align}
\del{where we have use\nnew{d} \eqref{eq:normalized-chain-def} and \eqref{eq:stepsize-update}. }
\new{If} the Law of Large Numbers \new{applies} to the RHS \new{of \eqref{eq:log-progress-cma}}, we obtain a limit of the LHS \del{for}\nnew{when} $T$ \nnew{goes} to infinity.
If this limit is \nnew{proven to be} strictly negative, we have \new{shown} linear convergence of \nnew{the underlying optimization algorithm}.\del{CMA-ES \new{when Gaussian distributions are used for sampling}).}
In order to apply \new{limit theorems~\cite[Theorem 17.0.1]{meyn2012markov} and obtain} a Law of Large Numbers, \new{we} require the chain $\Phi$ to be geometrically ergodic. 
Key assumption\new{s} for\del{ the} ergodicity are\del{ the} \del{$\varphi$-}irreducibility \new{and aperiodicity} of the Markov chain \new{whose notions will be formally introduced in Section~\ref{sec:IrrApp}}.
\new{We thus connected the stability of the normalized chain \nnew{to} the convergence of the underlying optimization algorithm. More formally the following proposition holds.}

\del{This is summarized by the following proposition.}%

\begin{proposition}\label{p:linear-convergence}
	Consider the CMA-ES algorithm defined in \Cref{sec:algo-cmaes} optimizing \new{a function $f$ 
		satisfying\del{ \Cref{F1:negligible-level-sets} and} \Cref{F2:scaling-invariant}}. 
	Assume that the process \new{$\Phi$}, \new{obeying \eqref{eq:normalized-chain-updates}} with state space $\cY=(\bR^d)^3\times R^{-1}(\{1\})\times\bR_{++}$ 
		(where $R^{-1}(\{1\}) = \{ \ncovmat \in \Sdpp  : R(\ncovmat)=1\}$),
		is an \del{$\varphi$-}irreducible, aperiodic and positive Harris-recurrent \new{Markov chain} with (unique) invariant probability measure $\pi$.
	\new{Assume moreover that the functions
	\begin{equation}
		(z,p,q,\ncovmat,r)\in\cY \mapsto \log\|z\|,\log\Gamma(p),\log r
	\end{equation}
	are $\pi$-integrable.}
		 Then the CMA-ES algorithm \new{behaves globally} asymptotically linearly \new{almost surely}\del{, that is,}\nnew{:}\done{precise the almost surely - it is done in MT via a.s $[\mathbb{P}_*]$ and then it becomes critical to put both limits below together. Two equations?}
	\begin{equation}\label{eq:linear-convergence}
	\lim_{T \to \infty} \frac{1}{T} \log \frac{\|\X_T - x^{*}\|}{\| \X_0 - x^*\|} = \lim_{t\to\infty} \bE\left[ \log \frac{\|\X_{t+1} - x^{*}\|}{\| \X_t - x^*\|} \right] = \int \left(\log\Gamma(p) + \frac{1}{2}\log r\right) \mathrm{d}\pi \enspace .
	\end{equation}
\end{proposition}

\begin{proof} 
	%
	\del{This}%
	\nnew{The almost sure limit of the LHS in \eqref{eq:linear-convergence}}
	follows directly from \eqref{eq:log-progress-cma}\nnew{, \del{the ergodic theorem and the }LLN for ergodic chains}~\cite[Theorem 17.0.1]{meyn2012markov}. 
	\nnew{The limit of the expectation in \eqref{eq:linear-convergence} follows from the ergodic theorem~\cite[Theorem 14.0.1]{meyn2012markov}, since
		$$
		\log\frac{\|m_{t+1}-x^*\|}{\|m_t-x^*\|}  =  \log\|z_{t+1}\|-\log\|z_t\| + \log\Gamma(p_{t+1}) +\frac{1}{2} \log r_{t+1}
		\enspace.
		$$}
	\done{The limit of $\lim_{t\to\infty} \bE\left[ \log \frac{\|\X_{t+1} - x^{*}\|}{\| \X_t - x^*\|} \right] $ does not follow from \eqref{eq:log-progress-cma} so the "follows directly" is misleading. Rather rewrite the needed equation while for applying the LLN we can indeed use \eqref{eq:log-progress-cma}}
	\del{This generalizes to any value of $x^*$ since CMA-ES is translation invariant~\cite{auger2016these}.}%
\end{proof}
The previous proposition illustrates that proving irreducibility and aperiodicity of 
the chain $\Phi$ is a stepping stone to establish linear convergence of CMA-ES.
Proving these properties will occupy \del{the rest of this paper}\Cref{sec:proofs}.
Later, we intend to prove the geometric ergodicity 
by means of Foster-Lyapunov drift conditions~\cite[Theorem 15.0.1]{meyn2012markov} 
that depend on small sets (as given in \Cref{t:main-1}). 
Characterizing small sets is\del{
will be} facilitated by the topological T-chain property as formalized in the next section.\done{is it really in the next section?}


		
		\section{\nnew{Main Results I:} Irreducibility, aperiodicity, and T-chain property of normalized Markov chains underlying the CMA-ES algorithm}
		\label{sec:IrrApp}

\new{We present in this section one of the two main results of this paper stating the irreducibility, aperiodicity and T-chain property of the normalized chains underlying the CMA-ES algorithm defined in \eqref{eq:normalized-chain-def}.} \new{We start by introducing}\del{Prior to that, \del{Before to state the main result of this section, }we introduce}
		\del{We start this section by giving}\new{the} definitions of
		irreducibility, aperiodicity and T-kernel.
			 \del{property of the normalized chain \eqref{eq:normalized-chain-updates} associated to CMA-ES}%
		Let $P$ be a transition kernel on a state space $(\cX,\cB(\cX))$.
		We say that $P$ is irreducible when there exists a nontrivial nonnegative measure $\varphi$ on $\cB(\cX)$ such that, 
		for every $x\in\cX$ and every $\mathsf A\in\cB(\cX)$ with $\varphi(\mathsf A)>0$, 
		there exists a positive integer $k$ satisfying $P^k(x,\mathsf A)>0$. 
		\del{If $P$ is irreducible \wrt}When a measure $\varphi$ satisfies this definition, we say\del{ then} that $P$ is $\varphi$-irreducible.
		
		When $P$ is irreducible, the period of $P$ is the largest integer $k\geqslant1$ such that there exist disjoint sets $\mathsf D_1,\dots,\mathsf D_k\in\cB(\cX)$ with
		\begin{equation}
			\left\{
			\begin{array}{l}
				\varphi((\mathsf D_1\cup\dots\cup \mathsf D_k)^c)=0 \text{ for every irreducibility measure } \varphi \text{ of } P \\
				P(x_i,\mathsf D_{i+1}) = 1 \text{ for } x_i\in \mathsf D_i \text{ and } i=0,\dots,k-1 ~ (\mathrm{mod}~ k) .
			\end{array}
			\right.
		\end{equation}
		\del{When $P$ is $\varphi$-irreducible, then it}An irreducible transition kernel $P$ always admits a period $k\geqslant1$~\cite[Theorem 5.4.4]{meyn2012markov},
		\nnew{and} when $k=1$, $P$ is said to be aperiodic.

		For any positive integer $m$, a set $\mathsf C\in\cB(\cX)$ is called $m$-small when there exists a nontrivial measure $\nu_m$ on $\cB(\cX)$ such that $P^m(x,\mathsf A)\geqslant \nu_m(\mathsf A)$ for every $x\in \mathsf C$ and every $\mathsf A\in\cB(\cX)$.

		Given a probability distribution $b$ on $\bN$, we define the transition kernel $K_b$ on $(\cX,\cB(\cX))$ as\break $K_b(x,\mathsf A) = \sum_{k\geqslant0} b(k) P^k(x,\mathsf A)$.
				
		A substochastic kernel on $(\cX,\cB(\cX))$ is a function $T\colon\cX\times\cB(\cX)\to\bR$ such that $T(\cdot,\mathsf A)$ is measurable for every $\mathsf A\in\cB(\cX)$
			and $T(x,\cdot)$ is a finite measure on $\cB(\cX)$ with $T(x,\cX)\leqslant1$
			for every $x\in\cX$.
		We say that the substochastic kernel $T$ is a \emph{continuous component} of 
			the transition kernel $K_b$ when $T(\cdot,A)$ is lower semicontinuous 
			on $\cX$, $T(x,\cX)>0$ and $K_b(x,\mathsf A)\geqslant T(x,\mathsf A)$ 
			for every $x \in \cX$ and $\mathsf A \in \cB(\cX)$.
		A transition kernel $P$ on $(\cX,\cB(\cX))$ is called a T-kernel when 
			there exist a probability measure $b$ on $\bN$ and a substochastic kernel $T$ \del{such that}which is a continuous component of the transition kernel $K_b$. Moreover, we say that a Markov chain is irreducible, respectively aperiodic, a T-chain, when its transition kernel is irreducible, respectively aperiodic, a T-kernel.
		We can now state our first main contribution \nnew{presented in the next theorem and its corollary. They constitute a first milestone towards a linear convergence proof of CMA-ES}. 
		The complete proof of the \del{next}following theorem is presented in \Cref{sec:proofs} 
			(cf.\ \Cref{t:main-bis}\del{ and \Cref{rem:homeomorphic}}).

		\begin{theorem}
				\label{t:main-1}
				Suppose that the objective function $f$, the normalization function\del{s} $R\del{(\cdot)}$, the stepsize change $\Gamma$ and the sampling distribution $\pUd$
				satisfy \Cref{F1:negligible-level-sets}-\Cref{F2:scaling-invariant}, \Cref{R1:homogeneous}-\Cref{R3:differentiable}, \Cref{G1:C1}-\Cref{G3:G(0)<1} and \Cref{P1:density}, respectively.
				
				Let $\Phi=\{(\Z_t,p_t,q_t,{\ncovmat}_t,r_t)\}_{t\geqslant1}$ be the normalized Markov chain \del{associated to}\new{underlying} CMA-ES defined via \eqref{eq:normalized-chain-def} {and $P$ its transition kernel}. \nnew{Assume that $0 < c_1 + c_\mu < 1$.}
				Then, 
				\begin{itemize}
					\item[(i)] if $c_c,c_\sigma\in (0,1)$, $c_\mu>0$
					\del{such that}\new{and} $1-c_c\neq(1-c_\sigma)\sqrt{1-c_1-c_\mu}$,
					then $P$ is an \del{$\varphi$-}irreducible aperiodic $T$-kernel, such that compact sets of $\bR^d\times\bR^d\times\bR^d\times R^{-1}(\{1\})\times\bR_{++}$ are small;
					\item[(ii)] if $c_c\in(0,1)$, $c_\sigma=1$ and $c_\mu>0$, then the \del{normalized chain}\new{process} $\{(\Z_t,q_t,{\ncovmat}_t,r_t)\}_{t\geqslant1}$ is \del{a time-homogeneous Markov chain with }an \del{$\varphi$-}irreducible aperiodic $T$-\del{kernel}\nnew{chain}, such that compact sets of $\bR^d\times\bR^d\times R^{-1}(\{1\})\times\bR_{++}$ are small;
					\item[(iii)] if $c_\sigma\in(0,1)$ and $c_c=1$, then the \del{normalized chain}\new{process} $\{(\Z_t,p_t,\ncovmat_t)\}_{t\geqslant1}$ is \del{a time-homogeneous Markov chain with }an \del{$\varphi$-}irreducible aperiodic $T$-\del{kernel}\nnew{chain}, such that compact sets of $\bR^d\times\bR^d\times R^{-1}(\{1\})$ are small; 					
					\item[(iv)] if $c_c=c_\sigma=1$, then the \del{normalized chain}\new{process} $\{(\Z_t,\ncovmat_t)\}_{t\geqslant1}$ is \del{a time-homogeneous Markov chain with }an \del{$\varphi$-}irreducible aperiodic $T$-\del{kernel}\nnew{chain}, such that compact sets of $\bR^d\times R^{-1}(\{1\})$ are small. 
				\end{itemize}			
			\end{theorem}
	\nnew{This result} \nnew{covers the entire range of eligible}\del{includes most} hyperparameter settings for CMA-ES except \nnew{when $c_1 + c_\mu = 1$, or $c_\mu=0$ and $c_c<1$, or $1-c_c = (1-c_\sigma)\sqrt{1-c_1-c_\mu} > 0$.}\del{
		for $c_1 + c_\mu = 1$ and\del{that} \new{when $c_c < 1$,} \del{for $c_\mu=0$ and}\new{we\del{ only} require $c_\mu>0$}\del{ to be stricly positive} and\del{ that} $1-c_c \neq (1-c_\sigma)\sqrt{1-c_1-c_\mu}$\del{ when $c_c<1$}.}
	\nnew{Most importantly,}\del{In other words,} \new{when cumulation is used \new{in}\del{for} the rank-one update\del{ ($c_c < 1$)}, \nnew{we need for our proof} the rank-mu update\del{ ($c_\mu > 0$)}.\del{is needed for the chain to be irreducible and aperiodic.}\del{(We however believe that this might be relaxed with more work).}\del{ However when no cumulation is used in the rank-one update, then we prove that the CMA-ES normalized Markov chain is an irreducible aperiodic Markov chain with or without rank-mu update. In particular} \nnew{Without cumulation however ($c_c=1$),} the rank-one update\del{ of the covariance matrix} is \nnew{already sufficient}\del{enough} to \nnew{prove irreducibility and aperiodicity.}\del{ensure the irreducibility of the normalized Markov chain.}}

	\nnew{We finally formulate a particular case of \Cref{t:main-1}.}
	\del{ in the following corollary.
	\new{\nnew{Besides,} the previous theorem is stated with an abstract normalization function $R(\cdot)$ satisfying the assumptions \Cref{R1:homogeneous}-\Cref{R3:differentiable}.
	However, this theorem might be used later on with a \del{concrete}\nnew{specific} choice of function $R(\cdot)$.\todo{Niko: reformulate}\niko{This seems entirely obvious to me.}
		This might be useful for instance when proving ergodicity via a Foster-Lyapunov drift condition, we might need to specify the function $R(\cdot)$. We thus state the following corollary\nnew{, which is a particular case of \Cref{t:main-1}}.\niko{I am not sure I understand the "idea" of the transition. We can't just say something like we emphasize that this specific algorithm/construction choice in covered by the theorem as stated in the corollary? }}%
		More precisely,}\del{Consequently,}%
	Using Proposition~\ref{p:examples-R},\del{ and} Lemma~\ref{lem:Gamma-CSA}
	\nnew{and \Cref{t:main-1},} we find that Markov chains \nnew{obtained with some standard}\del{using the default} stepsize change of CMA-ES\del{ \eqref{eq:CSA}} and normalized 
		by its minimum eigenvalue \new{(possibly expressed in a different coordinate system which would be more fitted to the objective function $f$)} or $\det(\cdot)^{1/d}$
	are \del{$\varphi$-}irreducible, aperiodic T-chains.\del{ \del{under the conditions of}\nnew{when} the previous theorem \nnew{applies}.}
	\begin{corollary}
		\new{Let $\mathbf{H}\in\Sdpp$.}
		Consider
		the process $\Phi$ defined via \eqref{eq:normalized-chain-def}
		with a normalization function $R=\det(\cdot)^{1/d}$ or $R=\lambda_{\min}(\mathbf H^{1/2}\times\cdot\times\mathbf H^{1/2})/\lambda_{\min}(\mathbf H)$
		and with the CSA stepsize change $\Gamma=\CSA$
		or $\Gamma=\CSAd$, see \eqref{eq:CSA} or \eqref{eq:CSA-mod}, respectively. 
		Assume as in \Cref{t:main-1} that $f$ satisfies \Cref{F1:negligible-level-sets}-\Cref{F2:scaling-invariant} and the sampling distribution  satisfies \Cref{P1:density}, then
			under the same conditions on the hyperparameters as in \Cref{t:main-1}, $\Phi$ 
		is an irreducible aperiodic T-chain.
	\end{corollary}

	\section{\nnew{Main results II:} Extension of the analysis of nonlinear state-space models}\label{sec:NLSS}
	
	\new{We present in this section our\del{ main results concerning} methodological extensions of tools to analyze the irreducibility, aperiodicity and T-chain property of Markov chains. After reminding the basics\del{ of those tools} in \Cref{sec:deterministic-control-model}, we present} \nnew{two extensions.}\del{an\del{the} extension\del{s} \nnew{in}}\del{ with respect to}%
\del{	In this section, we recall and extend the tools to prove \Cref{t:main-1}.
	We \del{recall}\nnew{remind}\niko{\tiny remind, summarize} \new{our} methodolo\-gy~\nnew{\cite{gissler2024irreducibility}} to prove\del{ the} irreducibility, aperiodicity and topological properties of a Markov chain \new{in the next \Cref{sec:deterministic-control-model}}.\del{
	This is based on existing works~\cite{gissler2024irreducibility}, but we remind the main results for the sake of completeness.}
	In order to \del{carry out the proof of}\nnew{prove \Cref{t:main-1}}\del{analysis}, we extend \new{this methodology}\del{those tools} in two directions.}

	First\del{ of all}, \nnew{some of the}\del{as we}\del{because we are interested to}\del{ analyze different} learning rate settings \new{from}\del{in} \Cref{t:main-1}(i), (ii), (iii) and (iv) 
		\del{we end up in }\del{some \new{of these} settings}\new{give rise to a}\del{end up with\niko{\tiny give rise to, [turn out to] bring about} a}\del{ to have} so-called \textit{redundant Markov chain}\del{s}, where one state variable can be dropped to define another Markov chain. 
	We thus introduce\del{ formally in \Cref{sec:subchain}} redundant and projected Markov chains \new{in \Cref{sec:subchain}}
		and explain how\del{ the} irreducibility, aperiodicity \new{and} T-chain property of the projected chain\del{s} can be deduced from \del{the irreducibility, aperiodicity, T-chain property from}\del{the}\nnew{an} analysis of the redundant chain.
	\nnew{The main result of this section is \Cref{c:verifiable-conditions-subchain}.}
		
The second \nnew{methodological} extension\del{the setting where the}\del{done is to be able to deal with}
is motivated\del{ by Markov chains that \new{are not defined}}\del{do not live on a smooth manifold}\del{This is achieved}\del{ the analysis of}
by the Markov chain \eqref{eq:normalized-chain-updates} which is valued in a possibly nonsmooth manifold since the normalization $R(\cdot)$ may be not continuously differentiable, for instance when $R(\cdot)=
\lambda_{\min}(\cdot)$. 
\nnew{To analyze such a\del{ Markov} chain, we} \new{apply}\del{ define via} a homeomorphic transforma\-tion, \nnew{thereby defining} a Markov chain\del{ that live} valued in a smooth manifold, and explain how\del{ the} irreducibility, aperiodicity and the T-chain property \nnew{of the original Markov chain} can be deduced from \nnew{an analysis of} the transformed Markov chain.

\new{These results \nnew{are}\del{will be} applied in \Cref{sec:proofs} for the proof of \Cref{t:main-1}.}
\del{In \Cref{sec:homeomorphism}, we introduce the idea to renormalize the covariance $\ncovmat$ in the Markov chain \eqref{eq:normalized-chain-updates} by a smooth map $\rho$. We hence }

	\subsection{Deterministic control model and sufficient conditions for irreducibility and aperiodicity}\label{sec:deterministic-control-model}
		
		We introduce in this section \new{different definitions and theorems our analysis is based on}
		\del{our\del{a} methodology\todo{Anne: revise - old} to analyze
		\del{the generic deterministic control models \del{associated}\nnew{for} \del{to }}%
		the Markov chains of CMA-ES and
		\del{, introduce different definitions and assumptions that we will verify in order to prove the irreducibility of the normalized CMA-ES Markov chain \cite{}.}\del{The main theorem we will use later and stating that a Markov kernel following the different assumptions introduced is an irreducible, aperiodic, T-kernel  is reminded at the end of the section.}%
		}\del{and refer to }the original article~\cite{gissler2024irreducibility} \nnew{to which we refer} for more details.\del{ing \del{the model}\nnew{it}\niko{Do you mean methodology?}\armand{yes, I changed to 'it' to avoid repetition}\niko{"it" is generally "risky"} in\del{for} more precisions\niko{I am not sure what the intended meaning of the sentence was}.}
			Let $\cX$ and $\cV$ be two smooth connected manifolds,\footnote{In the rest of the paper, manifolds will be considered as connected.} equipped with their Borel $\sigma$-fields, denoted $\cB(\cX)$ and $\cB(\cV)$, respectively. We later denote the dimension of $\cX$ by $n$. 
			\del{We adopt here the same notations \nnew{as in}\del{that} the article that presents the theory we remind here and refer to its supplementary material for \nnew{the}\del{ a reminder of a} formal definition of manifolds \cite{gissler2024irreducibility}.}%
			\del{We refer to the supplementary material of \nnew{the original article presenting these results}~\todo{cite} for a \nnew{reminder of a} formal definition of manifolds. We also adopt the same notations, in particular for the tangent spaces of a manifold, e.g.,}%
			\new{T}\del{ t}he tangent space of $\cX$ at a point $x\in\cX$ is denoted $\mathrm T_x\cX$, and we denote $\mathrm{dist}_{\cX}$ and $\mathrm{dist}_{\cV}$ the distance functions on $\cX$ and $\cV$, respectively, 
			which induce their respective topology. 
	Consider a transition kernel $P$ on $(\cX,\cB(\cX))$
	associated to \nnew{the Markov chain following} the\del{ following} update equation
	\del{, such that \nnew{a} Markov chain\del{s} $\Phi=\{\phi_t\}_{t\in\bN}$ with kernel $P$ obey\nnew{s} the\del{ following} control model}%
	\begin{equation}
		\label{eq:deterministic-control-model}
		\phi_{t+1}=F(\phi_t,\alpha(\phi_t,U_{t+1}))
	\end{equation}
	where $F\colon\cX\times\cV\to\cX$ and $\alpha\colon\cX\times\cU\to\cV$ are measurable functions, 
	and $\{U_{t+1}\}_{t\in\bN}$ is a\nnew{n} i.i.d.\ process independent of $\phi_0$ 
	and \del{with values}\nnew{valued} in \del{the}\new{a} measurable space $(\cU,\del{\cB(\cU)}\nnew{\mathcal{U}})$\new{, where $\mathcal{U}$ is a $\sigma$-field of $\cU$}.\footnote{Since we do not assume $\cU$ to be a topological space, we consider a general $\sigma$-field $\mathcal{U}$ instead of its Borel $\sigma$-field.}
	We consider additionally the following assumptions on the model.
	
	\begin{assumptionH}
		\label{A4:lsc-distribution} For any $x\in\cX$, the distribution $\mu_x$ of the random variable $\alpha(x,U_1)$ admits a density $p_x$ \wrt\ a $\sigma$-finite measure $\zeta_{\cV}$ on $\cV$, such that
		\begin{itemize}
			\item[(i)] the function $(x,v)\mapsto p_x(v)$ is lower semicontinuous;
			\item[(ii)] for $\mathsf A\in\cB(\cV)$, $\zeta_{\cV}(\mathsf A)=0$ if and only if $\mathsf A$ is negligible, i.e., $\mathrm{Leb}(\varphi(\mathsf A\cap V))=0$ for every local chart $(\varphi,V)$ of $\cV$.
		\end{itemize}
	\end{assumptionH}

	\begin{assumptionH}
		\label{A5:C1-update} The function $F:\cX\times\cV\to\cX$ is locally Lipschitz (\wrt\ the metrics $\mathrm{dist}_{\cX}\oplus \mathrm{dist}_{\cV}$ and $\mathrm{dist}_\cX$).
	\end{assumptionH}
	
%
Below, \Cref{p:CMA-follows-CM} 
	provides the\del{a} Markov chain\del{, defined in} \eqref{eq:smooth-normalized-chain-def}
	\del{states }that
	\del{ the transition kernel associated to \nnew{a normalized process underlying} CMA-ES, defined in \eqref{eq:smooth-normalized-chain-def}, }%
	follows the control model \eqref{eq:deterministic-control-model} and
	satisfies \Cref{A4:lsc-distribution} and \Cref{A5:C1-update} under mild assumptions on the objective function $f$ and the stepsize change $\Gamma$.
		
	\del{Before stating Assumption~\Cref{H3:controllability-condition},}%
\del{	let us introduce notations and definitions required to understand it.}%
  We define inductively the \textit{extended transition map} $S_{x}^k\colon\cV^k\to\cX$ 
  associated to \del{the control model }\eqref{eq:deterministic-control-model} 
  for any $k\in\bN$, $x\in\cX$ and $v_{1:k}=(v_1,\dots,v_k)\in\cV^k$ as follows
	\begin{equation}
		\label{eq:extended-transition-map}
		\left\{
		\begin{array}{l}
			S_{x}^0 \coloneqq x \\
			S_{x}^k (v_{1:k}) \coloneqq F\left( S_{x}^{k-1}(v_{1:k-1}),v_k \right) \quad  \text{for } k\geqslant 1 .
		\end{array}
		\right.
	\end{equation}
	From this definition, we obtain that if $F$ is \del{$C^m$}locally Lipschitz (respectively differentiable), then $(x,v_{1:k}) \mapsto S_{x}^k (v_{1:k})$ is \del{$C^m$}locally Lipschitz (respectively differentiable).\armand{I replaced $C^m$ by locally Lipschitz here}
	Likewise, we define \del{inductively }the \textit{extended probability density}\del{ as the function} $p_{x}^k\colon\cV^k\to\bR_{+}$ by
	\begin{equation}
		\left\{
		\begin{array}{l}
			p_{x}^1(v_1) \coloneqq p_{x}(v_1) \\
			p_{x}^k(v_{1:k}) \coloneqq p_{x}^{k-1}(v_{1:k-1})\times p_{S_{x}^{k-1}(v_{1:k-1})}(v_k) \quad  \text{for } k\geqslant2 .
		\end{array}
		\right.
	\end{equation}
	Given the Markov chain $\{\phi_t\}_{t\in\bN}$ defined via \eqref{eq:deterministic-control-model}, the function $p_{x}^k$ is a density associated to the random variable
	$(\alpha(\phi_0,U_1),\dots,\alpha(\phi_{k-1},U_k))$, when $\phi_0=x$.
	For $x\in\cX$ and $k\in\bN^*$, we define the \textit{control sets} of \eqref{eq:deterministic-control-model} by
	\begin{equation}
		\label{eq:control-sets}
		\cO_{x}^k\coloneqq \left\{ v_{1:k}\in\cV^k \mid p_{x}^k(v_{1:k})>0 \right\} \enspace.
	\end{equation}
	\del{Remark that the }Assumption \Cref{A4:lsc-distribution}(i) implies that these sets are open subsets of $\cV^k$.
	We define \del{also}moreover
	\begin{equation}
		\cO_{x}^\infty \coloneqq \left\{ v_{1:\infty} \in\cV^\bN \mid \forall k\geqslant1, v_{1:k}\in\cO_{x}^k  \right\} \enspace.
	\end{equation}
	We say that $x^*\in\cX$ is a \textit{steadily attracting state}, 
	when for every $x\in\cX$ and every neighborhood $U$ of $x^*$, 
	there exists $T>0$ such that for every $k\geqslant T$, 
	there exists $v_{1:k}\in\del{\overline{\cO_{x}^k}}\cO_{x}^k$ such that $S_x^k(v_{1:k})\in U$.
	When $F$ is continuous, $v_{1:k}$ can be taken in $\overline{\cO_{x}^k}$ \cite[Corollary~4.5]{gissler2024irreducibility}, 
	i.e., $x^*\in\cX$ is steadily attracting if and only if for every  neighborhood $U$ of $x^*$,
	 there exists $T>0$ such that for every $k\geqslant T$, 
	 there exists $v_{1:k}\in\overline{\cO_{x}^k}$ such that $S_x^k(v_{1:k})\in U$.
	In particular, if for every $x\in\cX$, 
	there exists $v_{1:\infty}\in\overline{\cO_x^\infty}$ such that $S_x^k (v_{1:k})$ tends to $x^*$, 
	then $x^*$ is a steadily attracting state~\cite[Corollary~4.5]{gissler2024irreducibility}.
	
		\new{We formulate now the following controllability condition of a steadily attracting state.}
		
		\begin{assumptionH}
			\label{H3:controllability-condition} There exist a steadily attracting state $x^*\in\cX$, an integer $k>0$, and a path $v^*_{1:k}\in\overline{ \cO^k_{x^*}}$, such that $\partial S_{x^*}^k(v_{1:k}^*)$ is of maximal rank.
		\end{assumptionH}
		
\del{	We say that \Cref{H3:controllability-condition} is a \textit{controllability condition} of a steadily attracting state $x^*$.} %
	For a locally Lipschitz function $G\colon\cV^k\to\cX$, $\partial G(v)$ is the Clarke's derivative of $G$ at 
	a point $v\in\cV^k$, 
	which is a set of linear applications between $\mathrm T_v\cV^k$ and $\mathrm T_{G(v)}\cX$~\new{\cite[Appendix~B]{gissler2024irreducibility}.}
	\del{We point out here that i}If $G$ is differentiable at $v$, 
	then $\partial G(v)=\{\mathcal{D}G(v)\}$, 
	where $\mathcal{D}G(v)$ denotes the usual differential application of $G$ in $v$. %
	\new{We then say}\del{Moreover, we say then} that $\partial G(v)$ is of maximal rank when \new{its}\del{every} element\new{s are}\del{ in it is} of maximal rank,
	\nnew{that is,\del{ is} of rank $n$ (the dimension of $\cX$)}. %
	\del{In particular, w}\nnew{W}hen $G$ is differentiable at $v$ and $\mathcal{D}G(v)$ is of maximal rank, 
	then $\partial G(v)=\{\mathcal{D}G(v)\}$ is of maximal rank. \nnew{We base our analysis on the following statement.}
	\begin{theorem}[Sufficient conditions for irreducibility and aperiodicity\new{~\cite[Theorem~2.3]{gissler2024irreducibility}}]
		\label{t:verifiable-conditions}
		Consider the Markov kernel $P$ defined via \eqref{eq:deterministic-control-model} such that \Cref{A4:lsc-distribution}-\Cref{H3:controllability-condition} are satisfied. Then $P$ is an irreducible, aperiodic T-kernel, and every compact set of $\cX$ is small.
	\end{theorem}
\new{This theorem \nnew{summarises}\del{frames} the methodology we follow to analyze \nnew{a normalized} Markov chain \nnew{under\-lying} CMA-ES:
	 we prove that \nnew{the chain}\del{it} satisfies \eqref{eq:deterministic-control-model} as well as 
	 conditions \Cref{A4:lsc-distribution}-\Cref{H3:controllability-condition}\footnote{\nnew{Condition 
	 		\Cref{H4:controllability-subchain} introduced in \Cref{sec:subchain} is 
	 		required instead of \Cref{H3:controllability-condition} when $c_c=1$ \del{and/}or $c_\sigma=1$.}} 
	 under 
	 appropriate conditions on the learning rates\nnew{, as well as on the 
	 	functions $f$, $\Gamma$ and $R$, and on the sampling distribution $\pUd$}.
 	\done{check if we need to says something more than conditions  on the learning rate}}

	\newcommand{\laconic}{projected}
	\newcommand{\laconicMC}{\laconic\ Markov chain}
	\newcommand{\laconicC}{\laconic\ chain}
	\newcommand{\redundant}{redundant}
	
	\subsection{Irreducibility and aperiodicity of a \laconicMC\del{~of a \redundant\ chain}\del{ without cumulation}} \label{sec:subchain}
	\done{correct everywhere an projected}
	The CMA-ES algorithm \new{maintains}\del{involves} two paths $p_t^c$ and $p_t^\sigma$ 
		that do not parametrize the \del{Gaussian }probability distribution \new{for}\del{used to} sampl\new{ing}\del{e} candidate solutions 
		but are used for \new{(accelerating) the update}\del{(speeding-up) the adaptation} of the covariance matrix and \new{the} stepsize\new{, respectively}. 
	Yet, when no cumulation for the stepsize path is used, 
		i.e., $c_\sigma=1$, or no cumulation for the rank-one update path is used, i.e., $c_c=1$, 
		the CMA-ES algorithm typically still works properly while it is \new{sometimes} slower~\nnew{\cite{gissler2023evaluation}}. 
	In \nnew{these}\del{this} case\nnew{s}, the \nnew{normalized} Markov chain \del{representing}\nnew{underlying} CMA-ES\del{ is simpler 
		in the sense that it} can be described with \nnew{fewer}\del{less} variables: 
		$p_t^c$ and $p_t^\sigma$\del{ are void as they} boil down to random vectors that depend on 
		the previous step only through the \new{ranking} permutation\del{ encoding the ranking} of candidate solutions. 
	\del{Yet, i}In order to analyze those algorithm variants \new{without}\del{while not} repeating\del{ entirely the}\del{all}\del{ entire} proofs \nnew{with small variations},
		we \new{introduce here a method}\del{present here }\del{a set of }\del{new results} that allows 
		to derive \del{theoretical }propert\del{y}\nnew{ies} for a \nnew{\laconic}\del{ simpler} \nnew{Markov} chain 
		from a \del{complete}\nnew{\redundant} \nnew{Markov} chain with a specific parameter setting.
	\nnew{We define a \emph{\redundant} Markov chain as a Markov chain $\{(\phi_t,\xi_t)\}_{t\in\bN}$ valued in a topological product space $\cX\times\cY$ such that the process $\{\phi_t\}_{t\in\bN}$ also is a Markov chain, valued in $\cX$.
		In that case, we say that $\{\phi_t\}_{t\in\bN}$ is a \emph{\laconic} Markov chain of $\{(\phi_t,\xi_t)\}_{t\in\bN}$.}
	
	Prior to that, we formalize the simplification of the \new{normalized} Markov chain \new{of CMA-ES} when at least one cumulation parameter is set to $1$. The proof is a direct consequence of \Cref{t:homogeneous-MC} \new{and thus} \new{omitted.}

	\begin{corollary}\label{cor:normalizedMCwithoutcumul}
		\del{is scaling-invariant satisfying {is positively homogeneous} satisf{ying} }%
		Suppose that the objective function $f$ 
			and that the normalization function $R$ \new{satisfy} \Cref{F2:scaling-invariant} \new{and} \Cref{R1:homogeneous}, respectively. 
		Let $\{(\X_t,p_t^\sigma,p_t^c,\covmat_t,\sigma_t)\}_{t\in\bN}$ be the \nnew{Markov} chain associated to CMA-ES defined in \Cref{sec:algo-cmaes} 
			and $\Phi=\{\phi_t\}_{t\geqslant1}=\{(\Z_t,p_{t},q_{t}, \ncovmat_t,r_t)\}_{t\geqslant1}$ be the 
			normalized process defined in \eqref{eq:normalized-chain-def}. 
		\begin{enumerate}
			\item[(i)] If $c_\sigma=1$, then the \del{sequence}\nnew{process} $\{(\Z_t,q_t,\ncovmat_t,r_t)\}_{t\geqslant1}$ defines a (time-homogeneous) Markov chain.
			\item[(ii)] If $c_c=1$, then the \del{sequence}\nnew{process} $\{(\Z_t,p_t,\ncovmat_t)\}_{t\in\bN}$ defines a (time-homogeneous) Markov chain.
			\item[(iii)] If $c_\sigma=c_c=1$, then the \del{sequence}\nnew{process} $\{(\Z_t,\ncovmat_t)\}_{t\in\bN}$ defines a (time-homogeneous) Markov chain.
		\end{enumerate}
		\end{corollary}
	\del{This}\nnew{\Cref{cor:normalizedMCwithoutcumul}} motivates \del{to introduce}\nnew{the introduction of} the notion 
		of \nnew{a} \textit{\laconicC}
		of a Markov chain $\tilde\Phi$, 
		and to provide conditions for irreducibility and aperiodicity as in \Cref{t:verifiable-conditions}.
	
	Define $\tilde\Phi = \{(\phi_t,\chi_t)\}_{t\in\bN}$ a \new{so-called }\textit{redundant} Markov chain 
		on $(\cX\times\cY,\cB(\cX\times\cY))$, with \del{Markov}\nnew{transition} kernel $\tilde P$, 
		\del{following the deterministic control model defined earlier in \Cref{sec:deterministic-control-model}\del{, i.e.,}\nnew{as} }
		\nnew{such that}
	\begin{equation}
		\label{eq:control-model-overchain}
		(\phi_{t+1},\chi_{t+1}) = \tilde F (\phi_t,\chi_t,\tilde\alpha(\phi_t,\chi_t,U_{t+1}))
	\end{equation}
	where $\tilde F\colon\cX\times\cY\times\cV\to\cX\times\cY$ and $\tilde\alpha\colon\cX\times\cY\times\cU\to\cV$ are measurable maps, 
		$\cX,\cY,\cV$ are (smooth, connected) manifolds\del{ of dimensions $n,d,p$ respectively}, 
		$(\cU,\del{\cB(\cU)}\nnew{\mathcal{U}})$ is a measurable space 
		and $\{U_{t+1}\}_{t\in\bN}$ is an i.i.d.\ process valued in $\cU$, independent of $(\phi_0,\chi_0)$.
	Assume \Cref{A4:lsc-distribution}-\Cref{A5:C1-update}, and denote $\tilde S^k_{(x,y)}$, $\tilde p_{(x,y)}^k$ and $\tilde\cO_{(x,y)}^k$ the extended transition map, the extended probability density and the control sets associated to the control model \eqref{eq:control-model-overchain}, for every $(x,y)\in\cX\times\cY$ and $k\in\bN$, respectively. 
	Besides, we suppose \new{redundancy of the chain by assuming} that the function $\tilde\alpha$ does not depend on the variable $\chi$, i.e., \new{there exists a function $\alpha$ such that}\del{ we can write}
	\begin{equation}\label{eq:redundancy}
		\tilde\alpha(\phi,\chi,u) = \alpha(\phi,u) \quad \text{for every } \phi\in\cX, \chi\in\cY , u\in\cU .
	\end{equation}
	Furthermore, we suppose that $\Phi=\{\phi_t\}_{t\in\bN}$ is a Markov chain on $\cX$ with \del{Markov}\nnew{transition} kernel denoted $P$, following the next deterministic control model
	\begin{equation}
		\label{eq:control-model-subchain}
		\phi_{t+1} = F(\phi_t,\alpha(\phi_t,U_{t+1})),
	\end{equation}
	with $\{U_{t+1}\}_{t\in\bN}$ being the i.i.d.\ process introduced to define the \redundant\ chain via \eqref{eq:control-model-overchain}.
	Then, we say that $\Phi$ is a \laconicC\ of $\tilde\Phi$.
	\del{Again,}\new{As above,} we denote $S^k_x$, $p_x^k$ and $\cO_x^k$ the extended transition map, the extended probability density and the control sets associated to the control model \eqref{eq:control-model-subchain}, for every $x\in\cX$ and $k\in\bN$, respectively.
	The next proposition connects the assumptions required for the two deterministic control models \eqref{eq:control-model-overchain} and \eqref{eq:control-model-subchain} that are useful to show that $\tilde \Phi$ and $\Phi$ are irreducible aperiodic T-chains.
	\begin{proposition}
		\label{p:DCM-subchain}
		Consider the control models \new{associated to the \redundant~chain} \eqref{eq:control-model-overchain} and \new{its associated \laconicC} \eqref{eq:control-model-subchain}. Then,
		\begin{itemize}
			\item[(i)] \new{if \Cref{A4:lsc-distribution} (resp.\ \Cref{A5:C1-update}) is satisfied for the \redundant~chain \eqref{eq:control-model-overchain}, 
				\nnew{then} it is satisfied for its \laconic\ chain \eqref{eq:control-model-subchain}};
			\item[(ii)] the closures of the control sets $\cO^k_{\phi}$ of \new{the \laconicC} \eqref{eq:control-model-subchain} equal the closures of the control sets $\tilde\cO^k_{(\phi,\chi)}$ of \new{the \redundant~chain} \eqref{eq:control-model-overchain}, that is,
			\begin{equation}
				\overline{\cO^{k}_{\phi}} = \overline{\tilde\cO^k_{(\phi,\chi)}} \quad \text{for every } \phi\in\cX, \chi\in\cY, k\geqslant1;
			\end{equation}
			\item[(iii)] the extended transition maps $S_\phi^k$ and $\tilde S_{(\phi,\chi)}^k$, 
				defined in \eqref{eq:extended-transition-map}, of the control models \new{of the \laconicC} \eqref{eq:control-model-subchain}, 
				and \new{the \redundant\ chain} \eqref{eq:control-model-overchain}, respectively, satisfy
			\begin{equation}
				S_{\phi}^k = \Pi_{\cX} \circ \tilde S_{(\phi,\chi)}^k \quad \text{for every } \phi\in\cX, \chi\in\cY, k\geqslant1, 
			\end{equation}
			where $\Pi_{\cX}\colon\cX\times\cY\to\cX$ is the canonical projection of $\cX\times\cY$ on $\cX$.
		\end{itemize}
	\end{proposition}
	\begin{proof}
		First, we prove (i).
		Suppose that \new{the redundant Markov chain following} \eqref{eq:control-model-overchain} satisfies \Cref{A4:lsc-distribution}, 
			i.e., for all $(\phi,\chi)\in\cX\times\cY$, the random variable $\tilde\alpha(\phi,\chi,U_1)$ admits 
			a density $\tilde p_{(\phi,\chi)}$ \wrt\ a $\sigma$-finite measure $\zeta_{\cV}$ satisfying \Cref{A4:lsc-distribution}, 
			such that $(\phi,\chi,v)\mapsto \tilde p_{(\phi,\chi)}(v)$ is \lsc. 
		\del{However, }%
		\new{Let $\phi\in\cX$. 
		By \eqref{eq:redundancy} \nnew{we have} $\alpha(\phi,U_1)=\tilde\alpha(\phi_,\chi,U_1)$ for every $\chi\in\cY$. 
		\del{Fixing an arbitrary}%
		\nnew{Let} $\chi_0\in\cY$. 
		\del{we know that}%
		\nnew{Then the random variable} $\tilde\alpha(\phi_,\chi_0,U_1)$ admits a density $\tilde p_{(\phi,\chi_0)}$ \wrt\ \nnew{a measure} $\zeta_{\cV}$ on $\cV$ \nnew{satisfying \Cref{A4:lsc-distribution}(ii), and $(\phi,v)\mapsto\tilde p_{(\phi,\chi_0)}(v)$}\del{which} is \lsc. Hence,  $\alpha(\phi,U_1)$ also admits a density \wrt\ $\zeta_{\cV}$ \del{that we }denote\nnew{d} $p_\phi$, \nnew{such that}
		$p_\phi(v) = \tilde p_{(\phi,\chi_0)}(v)$ for every  $v\in\cV$. 
		\del{Repeating the reasoning for another $\chi$ b}%
		By uniqueness of the density up to a null set, we obtain that, for $\phi\in\cX$ and $\chi\in\cY$
		\begin{equation}\label{eq:equaldens}
		p_\phi(v) = \tilde p_{(\phi,\chi)}(v) \quad \text{for } \zeta_{\cV}\mbox{-almost every } v\in\cV.
		\end{equation}
		Since $(\phi\del{,\chi_0},v) \mapsto  p_{(\phi,\chi_0)}(v)$ is \lsc, we obtain that $(\phi,v) \mapsto p_\phi(v) = p_{(\phi,\chi_0)}(v)$ is \lsc and thus the \laconicC~satisfies \Cref{A4:lsc-distribution}.}%
		\del{by definition, we have that $\alpha(\phi,U_1)=\tilde\alpha(\phi_,\chi,U_1)$ admits then the same density \wrt\ $\zeta_{\cV}$, and thus \eqref{eq:control-model-subchain} satisfies \Cref{A4:lsc-distribution}.} For assumption \Cref{A5:C1-update}, if $\tilde F$ is locally Lipschitz, then, since by definition we have $F(\phi,v) = \Pi_{\cX}\circ \tilde F(\phi,\chi,v)$ for $\phi\in\cX$, $\chi\in\cY$ and $v\in\cV$, by composition $F$ is locally Lipschitz ($\cY$ being nonempty).
		
		For (ii), \del{, observe that, since,}from \eqref{eq:equaldens} 
			and by definition of the control sets in \eqref{eq:control-sets}, 
			\del{  as shown above, $\tilde p_{(\phi,\chi)}=p_\phi$ for every $(\phi,\chi)\in\cX\times\cY$, then, by \eqref{eq:control-sets},}%
			for every $(\phi,\chi)\in\cX\times\cY$, $\cO^{k}_{\phi}$ and $\tilde\cO^k_{(\phi,\chi)}$ only differ by 
			a $\zeta_{\cV}$-negligible set. 
		Besides, both are open sets and $\zeta_{\cV}$ is ,by\del{ assumption} \Cref{A4:lsc-distribution}(ii), 
			a Borel measure, so $\overline{\cO^{k}_{\phi}} = \overline{\tilde\cO^k_{(\phi,\chi)}}$ 
			(as a direct consequence of Carath\'eodory's criterion of Borel measures~\cite[Theorem 1.9]{gariepy2015measure}). 
		
		In order to prove (iii), we proceed by induction. 
		Indeed, $S_\phi^0 = \phi = \Pi_{\cX}(\phi,\chi)=\Pi_{\cX}\circ \tilde S_{(\phi,\chi)}^0$. 
		\del{For}\nnew{Let} $k\geqslant0$ \nnew{and} assume $S_{\phi}^k=\Pi_\cX\circ \tilde S_{(\phi,\chi)}^k$. 
		Let $v_{1:k+1}\in\cV^{k+1}$, we find 
			$S_{\phi}^{k+1} (v_{1:k+1}) = F(S_{\phi}^k(v_{1:k}),v_{k+1} ) =  \Pi_{\cX} \circ \tilde F(S_{\phi}^k(v_{1:k}),\chi_k,v_{k+1})$ 
			where  $\chi_k=\Pi_{\cY} \circ \tilde S_{(\phi,\chi)}^k(v_{1:k}) \in\cY$. 
		By induction hypothesis we find that 
			$\Pi_{\cX} \circ \tilde F(S_{\phi}^k(v_{1:k}),\chi_k,v_{k+1}) 
			= \Pi_{\cX} \circ \tilde F(\tilde S_{(\phi,\chi)}^k(v_{1:k}),v_{k+1}) 
			=  \Pi_{\cX} \circ \tilde S_{(\phi,\chi)}^{k+1} (v_{1:k+1})$ 
			and thus  
			$S_{\phi}^{k+1} (v_{1:k+1}) = \Pi_{\cX} \circ \tilde S_{(\phi,\chi)}^{k+1} (v_{1:k+1})$. 
		Hence $S_{\phi}^{k+1}=\Pi_\cX\circ \tilde S_{(\phi,\chi)}^{k+1}$.
	\end{proof}
	We deduce \del{then }the \del{next}\nnew{following} result, 
		which characterizes the controllability condition for the \laconicMC.
		\begin{proposition}
		\label{c:SAS-for-subchain}
		\new{Assume that $F$ is continuous.\footnote{
				Alternatively, we can assume without loss of generality that for every $\phi\in\cX$, 
				the density functions $\tilde p_{(\phi,\chi)}$ are \del{all the same}\nnew{identical} for \del{every }$\chi\in\cY$. }}
		\del{Suppose that}\nnew{If} $x^*=(\phi^*,\chi^*)\in\cX\times\cY$ is 
			a steadily attracting state of \new{the redundant chain }\eqref{eq:control-model-overchain}\del{. T}\nnew{, t}hen 
			$\phi^*$ is a steadily attracting state of \new{the \laconicC~}\eqref{eq:control-model-subchain}. 
		
		Suppose moreover that there exists $k\geqslant 1$ and $v^*_{1:k}\in\overline{ \tilde\cO_{x^*}^k }$ such that $\tilde S_{x^*}^k$ is differentiable at $v^*_{1:k}$, and for every $h^\phi\in \mathrm T_{\phi_k}\cX$, there exists $h^\chi\in \mathrm T_{\chi_k}\cY$, where $(\phi_k,\chi_k)=\tilde{S}^k_{x^*}(v_{1:k}^*)$ and with $(h^\phi,h^\chi)\in\mathrm{rge}~\mathcal{D} \tilde S^k_{x^*}(v^*_{1:k})$.
		Then $v^*_{1:k}\in\overline{\cO_{\phi^*}^k}$ and $\mathcal{D} S^k_{\phi^*}(v^*_{1:k})$ exists and is of maximal rank.
	\end{proposition}
	\begin{proof}
	\new{Since $F$ is continuous, we can use the definition of a steadily attracting set via taking the elements for the $k$-step paths within the closure of the control sets (see Section~\ref{sec:deterministic-control-model}) instead of the control sets. According to the previous proposition 	$\overline{\cO^{k}_{\phi}} = \overline{\tilde\cO^k_{(\phi,\chi)}} \quad \text{for every } \phi\in\cX, \chi\in\cY, k\geqslant1~ $ and we can thus easily  prove that}\del{		By definition of steadily attracting states,
		it is straightforward that} $\phi^*$ is steadily attracting for \eqref{eq:control-model-subchain} when $(\phi^*,\chi^*)$ is steadily attracting for \eqref{eq:control-model-overchain}.
		Moreover, by \Cref{p:DCM-subchain}(iii), we have that $S_{\phi^*}^k = \Pi_{\cX} \circ \tilde S_{(\phi^*,\chi^*)}^k$.
		Then, by the chain rule~\cite[Corollary 2.6.6]{clarke1990optimization}, we have
		\begin{align*}
			\mathcal{D} S^{k}_{\phi^*} (v_{1:k}^*) & =  \mathcal{D}\Pi_{\cX}(\tilde S_{(\phi^*,\chi^*)}^k(v_{1:k}^*)) \circ \mathcal{D} \tilde S_{(\phi^*,\chi^*)}^k(v^*_{1:k}) \\
			& = \Pi_{\mathrm T_{ S_{\phi^*}^k(v_{1:k}^*)}\cX} \circ \mathcal{D} \tilde S_{(\phi^*,\chi^*)}^k(v_{1:k}^*).
		\end{align*}
		\done{typos here with mixing up $v^*$ and $v_{1:k}^*$ plus $\Pi_{\mathrm T_{\tilde S_{\phi^*}^k(v^*)}\cX}$ reads $\Pi_{\mathrm T_{S_{\phi^*}^k(v^*)}\cX}$ in previous equation?}
		The\del{n,}\nnew{refore} every $h^\phi\in \mathrm T_{S_{\phi^*}^k(v^*_{1:k})}\cX$ belongs to the range of $\mathcal{D}S_{\phi^*}^k(v_{1:k}^*)$\done{$\mathcal{D}S_{\phi^*}^k(v_{1:k}^*)$} \nnew{(we use here that by assumption there exists $h^\chi \in \mathrm T _ {\Pi_{\cY}S_{(\phi^*,\xi^*)}(v^*_{1:k})} $)}, 
			\del{thus }making it a surjective linear map, hence of maximal rank.%
	\end{proof}
	As a consequence, we derive sufficient conditions for irreducibility and aperiodicity of the kernel $P$ of the \laconicMC~$\Phi$. We replace the controllability condition \Cref{H3:controllability-condition} by the following.
	\begin{assumptionH}
		\label{H4:controllability-subchain} There exist a steadily attracting $x^*$, an integer $k>0$ and a path $v^*_{1:k}\in\overline{ \tilde\cO_{x^*}^k }$ such that $S_{x^*}^k$ is differentiable at $v^*_{1:k}$, and for every $h^\phi\in \mathrm T_\phi\cX$, there exists $h^\chi\in \mathrm T_\chi\cY$ with $(h^\phi,h^\chi)\in\mathrm{rge}~\mathcal{D} \tilde S^k_{x^*}(v^*_{1:k})$, where $(\phi,\chi)=\tilde{S}_{x^*}^k(v_{1:k}^*)$.
	\end{assumptionH}
	\begin{theorem}[Sufficient conditions for irreducibility and aperiodicity of a \laconicMC]
		\label{c:verifiable-conditions-subchain}
		Consider the control model \new{of the \redundant~chain} \eqref{eq:control-model-overchain} \del{for which}\new{and assume it satisfies} conditions \Cref{A4:lsc-distribution}-\Cref{A5:C1-update} and \Cref{H4:controllability-subchain}\del{ are satisfied}. Then the kernel $P$ \new{of the \laconicC~}defined via \eqref{eq:control-model-subchain} is an irreducible aperiodic T-kernel, and every compact set of $\cX$ is small.
	\end{theorem}
	
	\begin{proof}
		Denote $x^*= (\phi^*,\chi^*)$. Since $x^*$ is steadily attracting for \eqref{eq:control-model-overchain} \del{, then}\nnew{and} 
			\new{since $F$ is continuous by \Cref{A5:C1-update},} 
			\nnew{then} by \Cref{c:SAS-for-subchain}
			$\phi^*$ is steadily attracting for \eqref{eq:control-model-subchain}. 
		Besides, by \Cref{p:DCM-subchain}(ii), we have $v^*_{1:k}\in\overline{\cO_{x^*}^k}=\overline{\tilde\cO_{\phi^*}^k}$. 
		\del{Furthermore, using}\nnew{By} \Cref{c:SAS-for-subchain} again, we find that $\mathcal{D} S_{\phi^*}^k(v^*_{1:k})$ exists and is of maximal rank. 
		We complete the proof by \del{using}\nnew{applying} \Cref{t:verifiable-conditions}.
	\end{proof}
	
	\new{\Cref{c:verifiable-conditions-subchain} is used later to analyze the normalized chain defined in \eqref{eq:normalized-chain-def} when $c_c=0$ or $c_\sigma=0$. Even though \Cref{t:verifiable-conditions} could be applied directly to the \laconicC, this generalization allows to find a steadily attracting state and prove a controllability condition on the same chain for all \nnew{settings}\del{studied}\del{different cases,} without repeating \del{several times a similar}\nnew{the same} proof.}

	\subsection{Homeomorphic transformation of an irreducible aperiodic T-chain}\label{sec:homeomorphism}
	
		The state space of the chain $\Phi$ defined via \eqref{eq:normalized-chain-def} 
			is not a smooth manifold 
			if the normalization function $R$ is not continuously differentiable on $\Sdpp$. 
		\new{
			\del{Yet we want to prove the irreducibility of some normalized CMA Markov chains for $R(\cdot)$ which is not $\cC^1$, }%
			\nnew{In order to include nonsmooth functions $R$ in our analysis}
			(for instance if $R(\cdot)$ is the minimal eigenvalue of a positive definite matrix)},
		\del{Therefore }%
		we apply \Cref{t:verifiable-conditions} (or \Cref{c:verifiable-conditions-subchain} when we do not have cumulation) 
			to a Markov chain $\Theta$, 
			defined as a homeomorphic transformation of $\Phi$, 
			such that the state space of $\Theta$ is a smooth manifold. 
		This is \del{done}\nnew{achieved} in \Cref{sec-DCM-for-CMA,sec:steadily-attracting-state,sec:rank-condition-cma,sec:main-result}. 
		\del{Here}\del{In this section}\nnew{Now},
			we \del{observe that}\new{explain why} it is sufficient to prove that the transformed chain $\Theta$ is 
			an irreducible, aperiodic T-chain\del{ such that compact sets are small in order} 
			to have the same properties on $\Phi$.

		\begin{theorem}\label{p:irreducibility-via-homeomorphism}
			Let $\xi\colon\cY\to\cX$ be a homeomorphism between the topological spaces $\cY$ and $\cX$, equipped with their respective Borel $\sigma$-fields. Let $\Phi=\{\phi_t\}_{t\in\bN}$ be a (time-homogeneous) Markov chain with state space $\cY$, and define $\Theta=\{\xi(\phi_t)\}_{t\in\bN}$. Then,
			\begin{enumerate}
				\item[(i)] $\Theta$ is a (time-homogeneous) Markov chain with state space $\cX$;
				\item[(ii)] if $\Theta$ is irreducible (resp.\ aperiodic, a T-chain\del{, such that compact sets of $\cX$ are small}), 
					then $\Phi$ is irreducible (resp.\ aperiodic, a T-chain\del{, such that compact sets of $\cY$ are small}).
			\end{enumerate}
		\end{theorem}
	
		\begin{proof}
			First, we prove (i). Denote $P$ the Markov kernel of $\Phi$. 
			\del{Then, given that }%
			\nnew{If} the distribution of $\xi(\phi_0)$ is $\delta_x$ for $x\in\cX$, 
			\del{i.e.,}\nnew{then} $\phi_0$ is distributed under $\delta_{\xi^{-1}(x)}$. 
			\del{Hence, f}\nnew{F}or $\mathsf A\in\cB(\cX)$ and $x\in\cX$, we have \nnew{then}
			$$
				\bP\left[ \xi(\phi_1)\in \mathsf A \mid \xi(\phi_0)=x \right] 
				= \bP \left[ \phi_1 \in\xi^{-1}(\mathsf A) \mid \phi_0 = \xi^{-1}(x) \right] 
				= P(\xi^{-1}(x),\xi^{-1}(\mathsf A))
			$$
			\del{Then}\nnew{Moreover}, $P(\xi^{-1}(\cdot),\xi^{-1}(\cdot))$ defines 
				a Markov kernel for $\Theta$. 
			Indeed, since $\xi$ is a homeomorphism, 
				$\xi^{-1}$ is \new{continuous and thus} measurable. 
			\new{In particular the $k$-step transition kernel of $\Theta$ 
				equals $P^k(\xi^{-1}(\cdot),\xi^{-1}(\cdot))$ 
				(where $P^k$ is the $k$-step transition kernel of $\Phi$). 
			Thus $\Theta$ is a time-homogeneous Markov chain.}
			
			Now we prove (ii). 
			Suppose that $\Theta$ is irreducible, i.e., 
				the kernel $P(\xi^{-1}(\cdot),\xi^{-1}(\cdot))$ admits a nontrivial nonnegative measure $\vartheta$ on $\cB(\cX)$ such that
				for $x\in\cX$ and $\mathsf A\in\cB(\cX)$ with $\vartheta(\mathsf A)>0$, 
				there exists $k>0$ with $P^k(\xi^{-1}(x),\xi^{-1}(\mathsf A))>0$. 
			Then, for every $\mathsf B\in\cB(\cY)$ such that $\vartheta(\xi(\mathsf B))>0$ and for every $y\in\cY$, 
				there exists $k>0$ with $P^k(x,\mathsf B)>0$, i.e., $\Phi$ is $\vartheta\circ\xi$-irreducible. 
			Likewise, 
			for every irreducibility measure $\varphi$ of $\Phi$, 
				then $\varphi\circ\xi^{-1}$ is a irreducibility measure of $\Theta$. 
			In particular, every irreducibility measure $\vartheta$ of $\Theta$ can be defined as $\varphi\circ\xi^{-1}$ 
				for some irreducibility measure $\varphi$ of $\Phi$.
			Moreover, denote $k\geqslant1$ the period of $\Theta$, \del{hence}\nnew{i.e.,} $k$ is the largest integer such that 
				there exists disjoint sets $\mathsf D_1,\dots,\mathsf D_k\in\cB(\cY)$ with
			$$
			\left\{
			\begin{array}{l}
				\varphi((\mathsf D_1\cup\dots\cup \mathsf D_k)^c)=0  \quad \text{for any irreducibility measure }\varphi \text{ of } \Phi \\
				P(\xi^{-1}(y_i),\xi^{-1}(\mathsf D_{i+1}))=1 \text{ for } y_i\in \mathsf D_i \text{ and } i =0,\dots,k-1 \mod k .
			\end{array}
			\right.
			$$
			Therefore $k$ is the largest integer such that there exists disjoint sets $\mathsf C_1,\dots,\mathsf C_k\in\cB(\cX)$ with
			$$
			\left\{
			\begin{array}{l}
				\varphi(\xi^{-1}(\mathsf C_1\cup\dots\cup \mathsf C_k)^c)=0 \quad \text{for any irreducibility measure }\varphi \text{ of } \Phi  \\
				P(x_i,\mathsf C_{i+1})=1 \text{ for } x_i\in \mathsf C_i \text{ and } i =0,\dots,k-1 \mod k .
			\end{array}
			\right.
			$$
			Hence, the period of $\Phi$ equals $k$ the period of $\Theta$. In particular, if $\Theta$ is aperiodic, then $\Phi$ is aperiodic.
			
			Suppose now that $\Theta$ is a T-chain and let $T\colon\cX\times\cB(\cX)\to\bR_+$ 
				be a substochastic kernel such that $K_b(\xi^{-1}(\cdot),\xi^{-1}(\cdot))\geqslant T$ 
				for some probability distribution $b$ on $\bN$, $T(\cdot,\cX)>0$ 
				and $x\mapsto T(x,\mathsf A)$ is \lsc for $\mathsf A\in\cB(\cX)$.
			Then, if we define $T'(y,\mathsf B)=T(\xi(y),\xi(\mathsf B))$ 
				for $y\in\cY$ and $\mathsf B\in\cB(\cY)$, 
				\del{then}\nnew{we obtain that} $T'$ is a substochastic kernel such that $K_b\geqslant T'$ 
				for some probability distribution $b$ on $\bN$, 
				$T'(\cdot,\cY)>0$ and \nnew{that} $y\mapsto T(y,\mathsf B)$ is \lsc 
				for \nnew{every} $\mathsf B\in\cB(\cY)$. 
			Therefore, $\Phi$ is a T-chain.
			\del{
			Finally, observe that if $\mathsf C\in\cB(\cX)$ is a $m$-small set for the chain $\Theta$, 
			then $\xi(\mathsf C)$ is a $m$-small set for the chain $\Phi$. 
			Indeed, there exists by definition a nontrivial measure $\nu_m$ on $\cB(\cX)$ such that
			 $P^m(\xi^{-1}(x),\xi^{-1}(\mathsf A))\geqslant \nu_m (\mathsf A)$ for $x\in \mathsf C$ and
			  $\mathsf A\in\cB(\cX)$ and then $P^m(y,\mathsf B)\geqslant \nu_m(\xi(\mathsf B))$ 
			  for $y\in\xi(\mathsf C)$ and $\mathsf B\in\cB(\cY)$. 
			In particular, since $\xi$ is a homeomorphism, every compact set $\mathsf D$ of $\cY$ can be written as $\mathsf D=\xi^{-1}(\mathsf C)$ for
			 some compact set $\mathsf C$ of $\cX$, hence if compact sets of $\cX$ are small for $\Theta$,
			  then compact sets of $\cY$ are small for $\Phi$.
			}%
		\end{proof}

\niko{~\\book mark}
\niko{\\book mark: check comments from here}

	\section{Proof\del{s of main results} \nnew{of \Cref{t:main-1}}}\label{sec:proofs}
	
	The objective of this section is to prove \Cref{t:main-1}. 
	To do so, we investigate
	nonlinear state-space models associated to the recursion
	\del{underlying the Markov kernel of a normalized \new{Markov} chain obeying }%
	\eqref{eq:normalized-chain-updates} using \nnew{the} theoretical tools presented in \del{\cite{gissler2024irreducibility}}\nnew{\Cref{sec:NLSS}}.
	
	\nnew{Since the normalization function $R$ is not assumed to be smooth, 
		we consider a transformed Markov chain---for which we can apply \Cref{p:irreducibility-via-homeomorphism}---valued in a smooth manifold and which can be transformed via a homeomorphism into the normalized Markov chain \eqref{eq:normalized-chain-updates}.}
	In \Cref{sec-DCM-for-CMA}, we introduce the control model associated 
	to this transformed process and verify the conditions \Cref{A4:lsc-distribution}, \Cref{A5:C1-update} reminded in \Cref{sec:deterministic-control-model}.
	\del{to a normalized chain of CMA-ES}%
	%
	%
	%
	
	Then, a last condition\del{ needed}, \Cref{H3:controllability-condition} \nnew{or \Cref{H4:controllability-subchain}},\footnote{\Cref{H3:controllability-condition} for case (i) and \Cref{H4:controllability-subchain}\del{ instead of} for\del{ the} cases (ii), (iii), (iv) of \Cref{t:main-1}.} is
	proven in two steps: 
	\del{in }\Cref{sec:steadily-attracting-state}\del{ we} prove\new{s} the existence of a steadily attracting state (defined in \Cref{sec:deterministic-control-model}) 
	and\del{ in} \Cref{sec:rank-condition-cma}\del{ we} show\new{s} that a required controllability condition is satisfied.
	We conclude the proof in \Cref{sec:main-result}\del{.}\nnew{, where}
		\del{T}\nnew{t}he \new{Markov chains associated to the different}\del{redundant chains appearing for those specific} learning rate settings 
		are \del{then studied in \Cref{sec:main-result}}\nnew{analyzed, based on \Cref{c:verifiable-conditions-subchain}}.
		\del{ at the same time that learning rate settings where we cannot simplify the chain\niko{I don't understand what this sentence means, specifically "that" and "where" don't make sense to me.} to a projected one.}%

	\subsection{Definition of normalized chains underlying CMA-ES %
		following \eqref{eq:deterministic-control-model} and satisfying \Cref{A4:lsc-distribution}-\Cref{A5:C1-update}}\label{sec-DCM-for-CMA}
	In order to apply \Cref{t:verifiable-conditions} to the\del{ chain} \new{normalized CMA-ES Markov chain }defined via \eqref{eq:normalized-chain-def}, 
		\del{it is required to have}\del{we need that}%
		\nnew{we require} the state space $\cY=\bR^{3d}\times R^{-1}(\{1\})\times\bR_{++}$ \del{is}\nnew{to be} a smooth connected manifold.
	\del{However}\nnew{As mentioned in the previous section}, 
		this is not necessarily true unless we assume that the normalization $R$ is continuously differentiable. 
		\del{Yet, we are interested to analyze normalized chains when $R(C) = \lambda_{\min}(C)$ in which case \Cref{R2:C1} is satisfied but $R$ is not continuously differentiable.}%
	Hence, we introduce a homeomorphic transformation of the normalized chain which lives on a smooth manifold. \new{Consider} $\rho\colon\Sdpp\to\bR_{++}$\del{ is} a map satisfying
	\begin{assumptionRho}
		\label{rho1:homogeneous} the function $\rho$ is (positively) homogeneous and $\rho(\Id)=1$,
	\end{assumptionRho}%
	\begin{assumptionRho}
		\label{rho2:C1} the function $\rho$ is \del{continuously differentiable}\new{smooth ($\cC^\infty$)} on $\Sdpp$.
	\end{assumptionRho}%
\nnew{We keep this smooth normalization function abstract for the moment and will take it equal to $\rho(\cdot)=\det(\cdot)^{1/d}$ for proving \Cref{t:main-1}.}\anne{I added this to clarify that at some point we indeed specify $\rho$ - since keeping it abstract suggests otherwise that we develop some methodology rather than try to prove the theorem.}
	Define \new{now}
	\begin{equation}\label{eq:def-homeo}
		\begin{array}{rl}
			\xi\colon & \cY \to \cX \\
			& (z,p,q,\ncovmat,r) \mapsto \left(z,p,q,\rho(\ncovmat)^{-1}\ncovmat,r\right)
		\end{array}
	\end{equation}
\new{where $\cX = \bR^{d}\times\bR^d\times\bR^d\times\rho^{-1}(\{1\})\times\bR_{++}$. 
	Then, as stated in the next proposition, $\cX$ defines}
		\del{	These assumptions are sufficient to define}%
		a smooth connected manifold\del{as stated in the next proposition}.\done{to be discussed: where do we need that the manifold is connected? not sure we consistently say it.}
	\begin{proposition}\label{prop:rho-manifold}
		Suppose that the map $\rho\colon\Sdpp\to\bR_{++}$ satisfies \Cref{rho1:homogeneous}-\Cref{rho2:C1}.
		Then, the set $\cX=\bR^{d}\times\bR^d\times\bR^d\times\rho^{-1}(\{1\})\times\bR_{++}$ is a smooth connected manifold of dimension $3d+d(d+1)/2$.
	\end{proposition}
	\begin{proof}
		First note that the set $\mathsf M\coloneqq \bR^{3d}\times\Sdpp\times\bR_{++}$ is an open subset of the Euclidean space $\bR^{3d}\times\Sd\times\bR$, hence is a smooth submanifold of dimension $3d+d(d+1)/2+1$. Moreover, $\mathsf N\coloneqq\bR$ is a smooth manifold of dimension $1$. Define then the map \done{I introduce $\bar \varrho$ to not confuse - check for possible typos between $\bar \varrho$ and $\varrho$.}
		$$
			\begin{array}{rl}
				\bar\varrho\colon & \mathsf M\to \mathsf N \\
				& (z,p,q,\ncovmat,r) \mapsto \rho(\ncovmat) .
			\end{array}
		$$
		Then, by \Cref{rho2:C1}, $\bar \varrho$ is \del{a smooth map}\nnew{smooth}. 
		Moreover, it is a submersion at every point of $\mathsf M$. 
		Indeed, let $(z,p,q,\ncovmat,r)\in \mathsf M$, and let $\varepsilon \in (-1,1)$. 
		Then,
		$$
			 \bar \varrho \left( (z,p,q,\ncovmat,r) + (0,0,0,\varepsilon \ncovmat, 0)  \right) =  \rho ((1+\varepsilon) \ncovmat) = \rho(\ncovmat) +  \varepsilon \rho(\ncovmat)
		$$
		by \Cref{rho1:homogeneous}. Therefore, by Taylor expansion \new{and since the derivative $\mathcal{D}\bar \varrho(z,p,q,\ncovmat,r)$ is linear}, we have 
		$$\mathcal{D}\bar \varrho(z,p,q,\ncovmat,r)(0,0,0,\kappa\ncovmat, 0) = \kappa\rho (\ncovmat)\del{ \neq 0},$$
		\new{for every $\kappa\in\bR$, with $\rho (\ncovmat)>0$}.
		Hence, $\mathcal{D}\bar \varrho(z,p,q,\ncovmat,r)\colon \bR^{3d}\times\Sd\times\bR\to\bR$ is surjective
			\del{so}\nnew{and thus} $\bar \varrho$ is a submersion. 
		Therefore, by the submersion level set theorem \cite[Corollary 5.14]{lee2012introduction}, 
			$\cX=\bar \varrho^{-1}(\{1\})$ is a smooth manifold of dimension $3d+d(d+1)/2+1-1= 3d+d(d+1)/2$.
		
		Let us prove now that $\cX$ is connected.
		Since $\bR^{3d}\times\bR_{++}$ is connected, 
			it is sufficient to prove that the manifold $\rho^{-1}(\{1\})$ is connected,
			\del{This is then}\nnew{and thus} sufficient to prove that $\rho^{-1}(\{1\})$ 
			is path-connected~\cite[Proposition 1.11]{lee2012introduction}. 
		Let $\ncovmat_0,\ncovmat_1\in\rho^{-1}(\{1\})$. Since $\Sdpp$ is connected, there exists a continuous path $\gamma\colon[0,1]\to\Sdpp$ with $\gamma(0)=\ncovmat_0$ and $\gamma(1)=\ncovmat_1$. Define then the path $\hat\gamma\colon[0,1]\to\rho^{-1}(\{1\})$ by $\hat\gamma(t)=\gamma(t)/\rho(\gamma(t))$ for $t\in[0,1]$. Since $\gamma$ and $\rho$ are continuous, then $\hat{\gamma}$ is continuous. Besides, $\hat{\gamma}(0)=\ncovmat_0/\rho(\ncovmat_0)=\ncovmat_0$ and $\hat\gamma(1)=\ncovmat_1/\rho(\ncovmat_1)=\ncovmat_1$, ending the proof.
	\end{proof}	
	Moreover, the map $\xi$ defined \new{in \eqref{eq:def-homeo}}\del{above} is a homeomorphism \new{as stated below}.
	\begin{proposition}\label{p:xi-is-homeo}
		Suppose that $R$ and $\rho$ are both continuous and satisfy \Cref{R1:homogeneous} and \Cref{rho1:homogeneous}.
		Then, the map $\xi$ defined in \eqref{eq:def-homeo} between the sets $\cY$ and $\cX$ is a homeomorphism and
		 \begin{equation}\label{eq:xi-1}
			\begin{array}{rl}
				\xi^{-1}\colon & \cX \to \cY \\
				& (z,p,q,\rncovmat,r) \mapsto \left(z,p,q,R(\rncovmat)^{-1}\rncovmat,r\right) .
			\end{array} 
			\end{equation}
	\end{proposition}
	\begin{proof}
	\new{We can easily verify that}\del{This is straightforward if we write} the expression of the reciprocal function of $\xi$ is \eqref{eq:xi-1}.
		Then $\xi^{-1}$ (resp.\ $\xi$) is continuous since $R$ (resp.\ $\rho$) is continuous and takes value in $\R_{++}$. 
			\end{proof}
\new{We formalize in the next lemma the update equations for $\{\theta_t\}_{t \in \mathbb{N}} = \{ \xi(\phi_t) \}_{t \in \mathbb{N}} $. }	
	\begin{lemma}
		\label{l:theta-homogeneous-MC}
	\del{Let $R$ be a positively homogeneous}\nnew{Suppose that the} normalization function $R$ \new{satisfies \Cref{R1:homogeneous}} 
	and \nnew{let} $\Phi = \{ \phi_t \}_{t \in \mathbb{N}} $ be the\del{ CMA-ES normalized} Markov chain defined via \eqref{eq:normalized-chain-updates}. 
	Let $\rho$ be a \del{smooth }normalization function satisfying \Cref{rho1:homogeneous} 
	and let $\xi$ be the homeomorphism defined in \eqref{eq:def-homeo}. 
	Then the \nnew{Markov} chain $\Theta = \{\theta_t\}_{t \in \mathbb{N}} = \{ \xi(\phi_t) \}_{t \in \mathbb{N}} $ satisfies 	
	\begin{equation}
		\label{eq:smooth-normalized-chain-def}
		\begin{array}{l}
			\Z_{t+1} = \frac{\Fx(\Z_t,\sqrt{R(\rncovmat_t)^{-1}\rncovmat_t} \w_m^\top \Utt )}{\sqrt{r_{t+1}}\Gamma(p_{t+1})  }   \\
			p_{t+1} =  \Fpsigma(p_t, \w_m^\top \Utt ) \\
			q_{t+1} =  \Fpc(r_t^{-1/2}q_t, \sqrt{R(\rncovmat_t)^{-1}\rncovmat_t} \w^\top_m \Utt ) \\
			\rncovmat_{t+1} = \frac{\Fc\left(R(\rncovmat_t)^{-1}\rncovmat_t,q_{t+1}, R(\hat\ncovmat_t)^{-1} \sqrt{\hat\ncovmat_t}
				\sum_{i=1}^\mu \wic \left[  U_{t+1}^{s_{t+1}(i)} \right]
				\left[ U_{t+1}^{s_{t+1}(i)} \right]^\top
				\nnew{\sqrt{\hat\ncovmat_t}}\right)}
			{\rho\circ \Fc\left(R(\rncovmat_t)^{-1}\rncovmat_t,q_{t+1}, R(\hat\ncovmat_t)^{-1} \sqrt{\hat\ncovmat_t}
				\sum_{i=1}^\mu \wic \left[  U_{t+1}^{s_{t+1}(i)} \right]\left[ U_{t+1}^{s_{t+1}(i)} \right]^\top
				\nnew{\sqrt{\hat\ncovmat_t}}\right)}  \\
			r_{t+1} = R\circ \Fc \left(R(\rncovmat_t)^{-1}\rncovmat_t,q_{t+1},
			\nnew{R(\hat\ncovmat_t)^{-1} \sqrt{\hat\ncovmat_t}}
			\sum_{i=1}^\mu \wic \left[  U_{t+1}^{s_{t+1}(i)} \right]\left[ U_{t+1}^{s_{t+1}(i)} \right]^\top
			\nnew{ \sqrt{\hat\ncovmat_t}}\right) 
		\end{array}
	\end{equation}
where $\{U_{t+1}\}_{t\in\bN}$ is an i.i.d.\ process independent of $\theta_0=(\Z_0,p_0,q_0,\rncovmat_0,r_0) = \xi(\phi_0)$ \del{with values}\nnew{distributed} in the measured space $\cU=(\bR^d)^\lambda$ with $U_1\sim\pUd$, and $s_{t+1}$ is a permutation of $\mathfrak{S}_\lambda$ that sorts \del{increasingly }the $f$-values\done{Niko - consistency with earlier change of Niko} of $x^*+z_t+\sqrt{R(\rncovmat_t)^{-1}\rncovmat_t}U_{t+1}^i$, for $i=1,\dots,\lambda$.%
\del{defined via }%
\done{}%
	\end{lemma}
	\begin{proof}
	Since for $t\in\bN$, \nnew{according to \eqref{eq:xi-1}, }we have that $\ncovmat_t= R(\rncovmat_t)^{-1}\rncovmat_t$, the update equations 
	for $p_{t+1}$, $q_{t+1}$ ,$r_{t+1}$ and $z_{t+1}$ in \eqref{eq:smooth-normalized-chain-def} are deduced directly from \Cref{l:homogeneous-MC} \nnew{where we replace $\ncovmat_t$ by $R(\rncovmat_t)^{-1}\rncovmat_t$}. Moreover, we have, by \Cref{l:homogeneous-MC} \nnew{and using the definition of $\tilde{\ncovmat}_{t+1}$ in \eqref{tilda-matrix}}
	\begin{align*}
			\rncovmat_{t+1} & = \frac{\ncovmat_{t+1}}{\rho(\ncovmat_{t+1})} \nnew{= \frac{\tilde{\ncovmat}_{t+1}}{R(\tilde{\ncovmat}_{t+1})}  \times \frac{R( \tilde{\ncovmat}_{t+1})}{\rho( \tilde{\ncovmat}_{t+1})}} \\
			& =\frac{\Fc\left(\ncovmat_t,q_{t+1}, 
				 \sqrt{\ncovmat_t}
				\sum_{i=1}^\mu \wic \left[  U_{t+1}^{s_{t+1}(i)} \right]
				\left[ U_{t+1}^{s_{t+1}(i)} \right]^\top
				\sqrt{\ncovmat_t}\right)}{\rho\circ\Fc\left(\ncovmat_t,q_{t+1}, 
				\sqrt{\ncovmat_t}
				\sum_{i=1}^\mu \wic \left[  U_{t+1}^{s_{t+1}(i)} \right]
				\left[ U_{t+1}^{s_{t+1}(i)} \right]^\top
				\sqrt{\ncovmat_t}\right) } \enspace.
	\end{align*}
	The proof ends by replacing $\ncovmat_t$ by $R(\rncovmat_t)^{-1}\rncovmat_t$.
	\end{proof}
	

	Using \Cref{p:irreducibility-via-homeomorphism}, 
		we can transfer the irreducibility, aperiodicity and the T-chain property from the \nnew{Markov} chain $\Theta$ 
		to the original normalized chain $\Phi$ we are interested \del{to analyze}\nnew{in}. 
	Our objective from now on is \del{thus to prove the irreducibility and aperiodicity of the chain $\Theta$}%
		to prove that the Markov chain $\Theta$ is an irreducible aperiodic T-chain.
	Our strategy \del{to prove the irreducibility and aperiodicity of the chain $\Theta$}\new{for that} 
		is to \del{verify that the verifiable conditions given in }%
		\nnew{apply} \Cref{t:verifiable-conditions}
		\nnew{and verify that the required assumptions are satisfied}. 
		\del{are satisfied when studying the stability of \del{the}\new{its} corresponding deterministic control model \new{of the form \eqref{eq:deterministic-control-model}}.}%
	We first prove that $\Theta$ follows a deterministic control model 
		of the form \eqref{eq:deterministic-control-model}, 
		described in \Cref{sec:deterministic-control-model}. 
	
	\new{Consider t}\del{T}he \nnew{smooth} manifold $\cX  = \bR^d\times\bR^d \times \bR^d \times \rho^{-1}(\{1\}) \times\bR_{++}$ \nnew{(see \Cref{prop:rho-manifold})}
		\new{that} defines the state space of the \nnew{Markov} chain $\Theta$ 
		and let $\cV \coloneqq \bR^{d\mu}$.
		We define 
		\begin{equation}
		\label{eq:def-alpha}
		\begin{array}{rlll}
			\alpha_\Theta \colon & \cX\times\cU & \to & \cV \\
			& ((z,p,q,\nnew{\rncovmat},r), (u^1,\dots,u^\lambda)) & \mapsto & \left[ \sqrt{\frac{\rncovmat}{R(\rncovmat)}}  u^{s_{f(x^*+z+   \sqrt{R(\nnew{\rncovmat})^{-1}\nnew{\rncovmat}}\cdot)}^u (i)} \right]_{i=1,\dots,\mu} 
		\end{array}
	\end{equation} 
	where given $g\colon \mathsf A\to\bR$ a function and $v\in \mathsf A^\lambda$ we have used the notation $s_{g}^v$ for a permutation that sorts increasingly the $g(v^i)$, $i=1,\dots,\lambda$. Consider the $(z^+,p^+,q^+,\rncovmat^+,r^+)$ the update of $\theta=(z,p,q,\rncovmat,r)$ given the random input equals $v =\alpha_\Theta(\theta,u)$ that is
	\done{Write the correct equations}
	\begin{align}
	z^+ & = \frac{z+c_m\vwm^\top v}{\sqrt{r^+}\Gamma(p^+)} \label{eq:update-z} \\ 
	p^+ & = (1-c_\sigma)p +\sqrt{c_\sigma(2-c_\sigma)\mueff }R(\rncovmat)^{1/2}\rncovmat^{-1/2}\vwm^\top v \label{eq:update-p} \\
	q^+ & =  r^{-1/2}(1-c_c)q +\sqrt{c_c(2-\del{c_v}\nnew{c_c})\mueff}  \vwm^\top v \label{eq:update-q} \\ 
	\rncovmat^+ & =  \cfrac{ (1-c_1-c_\mu) R(\rncovmat)^{-1} \rncovmat + c_1 q^+(q^+)^\top+ c_\mu \sum_{i=1}^\mu\wic v_iv_i^\top}
		{\rho\left((1-c_1-c_\mu) R(\rncovmat)^{-1}\rncovmat + c_1 q^+(q^+)^\top+ c_\mu \sum_{i=1}^\mu\wic v_iv_i^\top\right)} \label{eq:update-Sigma} \\
	r^+ & = R\left((1-c_1-c_\mu) R(\rncovmat)^{-1}\rncovmat + c_1 q^+(q^+)^\top+ c_\mu \sum_{i=1}^\mu\wic v_iv_i^\top\right). \label{eq:update-r} 
	\end{align}
This update defines a function			
\del{We define}  $F_\Theta \colon\cX\times\cV\to\cX$ such that
$$
(z^+,p^+,q^+,\rncovmat^+,r^+) = F_\Theta((z,p,q,\rncovmat,r),\alpha_\Theta(\theta,u))
$$	
that can be expressed as
\begin{multline}
		\label{eq:def-F}
			F_\Theta ((\varnorm), (v^1,\dots,v^\mu)) =  \left( 
			 	 \Fz(z,p,q,R(\nnew{\rncovmat})^{-1}\nnew{\rncovmat},r;v)   , 
				 \Fsig(p,R(\nnew{\rncovmat})^{-1}\nnew{\rncovmat};v)  ,
				 \Fq(q,r;v)  ,  \right. \\
				\left.  \FK(q,R(\nnew{\rncovmat})^{-1}\nnew{\rncovmat},r;v)  ,
				  \Fr(q,R(\nnew{\rncovmat})^{-1}\nnew{\rncovmat},r;v) \right)^\top	 
\end{multline}
	\nnew{for $(z,p,q,\rncovmat,r)\in\cX$ and $(v^1,\dots,v^\mu) \in\cV$, and}
	where $\Fz$, $\Fsig$, $\Fq$, $\FK$ and $\Fr$ are defined as follows
	\begin{align}\label{eq:update-rnormalized-chain}
		&	 \Fz(z,p,q,\ncovmat,r;v)  = \Fr(q,\ncovmat,r;v)^{-1/2}\Gamma\circ\Fsig(p,\ncovmat;v)^{-1} \Fx(z, \w_m^\top v )\\
		&	\Fsig(p,\ncovmat;v)  = \Fpsigma(p,\ncovmat^{-1/2}\w_m^\top v) \\
		&	\Fq(q,r;v)  = \Fpc(r^{-1/2}q,\w_m^\top v) 
			\label{eq:Fq} \\
		&	 \FK(q,\ncovmat,r;v) = \cfrac{\Fc\left( \ncovmat,\Fq(q,r;v),\sum_{i=1}^\mu\wic v_iv_i^\top \right)}{\rho\circ\Fc\left( \ncovmat,\Fq(q,r;v),\sum_{i=1}^\mu\wic v_iv_i^\top \right)}
		\label{eq:FSigma}	\\
		&	 \Fr(q,\ncovmat,r;v)  = R\circ \Fc \left( \ncovmat,\Fq(q,\ncovmat,r;v),\sum_{i=1}^\mu\wic v_iv_i^\top \right).   \label{eq:update-Fr}
	\end{align}

\del{	and where we have
	\begin{equation}
		\cV \coloneqq \bR^{d\mu} .
	\end{equation}}
\del{	\new{Given} $g\colon \mathsf A\to\bR$ a function and $v\in \mathsf A^\lambda$, we use the notation $s_{g}^v$ for a permutation that sorts increasingly the $g(v^i)$, $i=1,\dots,\lambda$. 
	\new{\del{Using}\nnew{With} this notation, we define}	
\del{	Then, we choose}the function $\alpha_\Theta$ as follows.
	\begin{equation}
		\label{eq:def-alpha}
		\begin{array}{rlll}
			\alpha_\Theta \colon & \cX\times\cU & \to & \cV \\
			& ((z,p,q,\nnew{\rncovmat},r), (u^1,\dots,u^\lambda)) & \mapsto & \left[ \sqrt{\frac{\rncovmat}{R(\rncovmat)}}  u^{s_{f(x^*+z+   \sqrt{R(\nnew{\rncovmat})^{-1}\nnew{\rncovmat}}\cdot)}^u (i)} \right]_{i=1,\dots,\mu} .
		\end{array}
	\end{equation} 
	}
	Then, as stated in the next proposition, 
		$\Theta$ follows \del{a deterministic control}\nnew{the} model \nnew{described in \Cref{sec:deterministic-control-model}} 
		with the functions $F_\Theta$ and $\alpha_\Theta$ defined above.	
	\begin{proposition}
		\label{p:CMA-follows-CM}
		Suppose that the normalization functions $R$ satisfies \Cref{R1:homogeneous} and $\rho$ satisfies \Cref{rho1:homogeneous}.
		Then, the Markov chain $\Theta = \{(\Z_t,p_t,q_t, \rncovmat_t,r_t)\}_{t\in\bN}$ defined by \eqref{eq:smooth-normalized-chain-def} satisfies
		\del{ the following control model}%
		\begin{equation}\label{eq:control-modelTheta}
		\theta_{t+1} = F_{\Theta}\left(\theta_t,\alpha_{\Theta}(\theta_t,U_{t+1}) \right)
		\end{equation}
		\del{ with $F$ being defined via \eqref{eq:def-F} and $\alpha$ via \eqref{eq:def-alpha}.}%
		\new{where $F_\Theta$ is defined in \eqref{eq:def-F} and $\alpha_\Theta$ in \eqref{eq:def-alpha}.}%
	\end{proposition}
	\begin{proof}
	Straightforward by \Cref{l:theta-homogeneous-MC}. 
	\end{proof}
	
	Before to prove that the control model \eqref{eq:control-modelTheta} associated to the Markov chain $\Theta$ satisfies \del{that we prove that }the assumptions \Cref{A4:lsc-distribution} and \Cref{A5:C1-update}\del{ hold}, we characterize in the next lemma\del{ characterizes} the density of the random variable $\alpha\new{_\Theta}(\theta,\del{V}\new{U})$ \del{(more precisely we derive the density of $\left[\hat{\ncovmat}/R(\hat{\ncovmat})\right]^{-1/2}\alpha\new{_\Theta}(\theta,\del{V}\new{U})$)} for $\theta\in\cX$ and $U\sim(\pUk{d})^{\otimes\lambda}$ \new{assuming that the objective function is the composite of a strictly increasing function with a function with negligible level sets and the distribution $\pUd$ is admits a density positive everywhere \wrt\ the Lebesgue measure}.
	\nnew{The latter assumption could be relaxed with more work, but for the purposes of this paper we only consider positive densities.}
	\done{discuss the assumptions - I guess we could relax with more work the density positive everywhere but it's more convenient, right?}
	
	\newcommand{\fx}{\ensuremath{f_{*}}}
	\begin{lemma}\label{l:density-alpha}
	\done{lsc of the density: state and small proof?}
		Suppose that the objective function $f$ satisfies \Cref{F1:negligible-level-sets} and that the probability distribution $\pUd$ satisfies \Cref{P1:density}. 
		Define, for any $\theta = (z,p,q,\hat\ncovmat,r)\in\bR^d\times\bR^d\times\bR^d\times\Sdpp\times\bR_{++}$ and $v=(v_1,\dots,v_\mu)\in\bR^{d\mu}$,
		\begin{equation}\label{eq:density}
			p_{\del{\theta}\nnew{z,\ncovmat}}(v) = 
			\frac{\lambda!}{(\lambda-\mu)!}
			\1\{ \nnew{\fx}(z+\sqrt{\ncovmat}v_1) < \dots < \nnew{\fx}(z+\sqrt{\ncovmat}v_\mu) \} 
			(1-Q_{\del{\theta}\nnew{z,\ncovmat}}^{\nnew{\fx}}(v_\mu))^{\lambda-\mu} \densityd(v_1)\dots\densityd(v_\mu)
		\end{equation}
		with $\ncovmat=\hat{\ncovmat}/R(\hat{\ncovmat})$ where $R$ is the normalization function used in \eqref{eq:def-alpha},
		 $Q_{\del{\theta}\nnew{z,\ncovmat}}^{\nnew{\fx}}(u)=\int \1\{ \nnew{\fx}(z+\sqrt{\ncovmat}\xi)< \nnew{\fx}(z+\sqrt{\ncovmat}u) \} \pUd(\mathrm{d}\xi)$ for $u\in\bR^d$, and $\fx=f(\cdot+x^*)$. Then, $p_{z,\ncovmat}$ defines a density (\wrt\ Lebesgue in $\bR^{d\mu}$) of the random variable $\del{R(\ncovmat)^{1/2}}\ncovmat^{-1/2}\alpha_\Theta(\theta,U)$, where $U\sim(\pUk{d})^{\otimes\lambda}$ such that the density of $\alpha_\Theta(\theta,U)$ equals
		 \begin{equation}\label{eq:density-alpha}
  		 	v \mapsto \det\ncovmat^{-1/2} p_{z,\ncovmat}(\ncovmat^{-1/2} v) \nnew{= \frac{1}{\sqrt{\det \ncovmat}} p_{z,\ncovmat}(\ncovmat^{-1/2} v) } \enspace.
		 \end{equation}
		 \new{Besides, when $R$ is continuous, the function $((z,p,q,\hat\ncovmat,r),v)\in\bR^{3d}\times\Sdpp\times\bR_{++}\times\bR^{d\mu}\mapsto p_{z,\rncovmat/R(\rncovmat)}(v)$ is lower semicontinuous} \nnew{and thus the density function \eqref{eq:density-alpha} is lower semicontinuous as well}.
	\end{lemma}
	
	\begin{proof}
		Let $U^1,\dots,U^\lambda\del{\sim\pUd}$ \new{be} independent random vectors identically distributed \del{following}\nnew{under} the probability distribution $\pUd$, 
			and denote $U=(U^1,\dots,U^\lambda)$. 
		Let $\theta = (z,p,q,\hat\ncovmat,r)\in\bR^d\times\bR^d\times\bR^d\times\Sdpp\times\bR_{++}$. 
		Since the objective function $f$ satisfies \Cref{F1:negligible-level-sets}, 
		then the random vector $V=\del{R(\ncovmat)^{1/2}}\ncovmat^{-1/2}\alpha_\Theta(\theta,U)$ satisfies almost surely
		$$
		V = \sum_{\sigma\in\mathfrak{S}_\lambda} \1\left\{ \fx\left(z+\sqrt{\ncovmat}U^{\sigma(1)}\right) < \dots < \fx\left(z+\sqrt{\ncovmat}U^{\sigma(\lambda)}\right) \right\} \times \left(U^{\sigma(1)},\dots,U^{\sigma(\mu)} \right)  .
		$$
		where $\mathfrak{S}_\lambda$ is the set of permutations of $\{1,\dots,\lambda\}$.	
		Hence, by symmetry,
		\begin{multline*}				
			V = \frac{1}{(\lambda-\mu)!}\sum_{\sigma\in\mathfrak{S}_\lambda} \1\left\{ \fx\left(z+\sqrt{\ncovmat}U^{\sigma(1)}\right) < \dots < \fx\left(z+\sqrt{\ncovmat}U^{\sigma(\mu)}\right) \right\} \\ \times \prod_{k=\mu+1}^\lambda \1\left\{ \fx\left(z+\sqrt{\ncovmat}U^{\sigma(\mu)}\right) < \fx\left(z+\sqrt{\ncovmat}U^{\sigma(k)}\right) \right\} \times \left(U^{\sigma(1)},\dots,U^{\sigma(\mu)} \right)  .
		\end{multline*}
		Let $\eta\colon\bR^{d\mu}\to\bR_+$ be a smooth map with compact support. 
		We \del{obtain}\nnew{have}
		\begin{multline*}
			\bE\left[ \eta(V) \right] =\frac{1}{(\lambda-\mu)!}\sum_{\sigma\in\mathfrak{S}_\lambda}  \int   \1 \left\{ \fx\left(z+\sqrt{\ncovmat}u_{\sigma(1)}\right) < \dots < \fx\left(z+\sqrt{\ncovmat}u_{\sigma(\mu)}\right) \right\}  \\
			\times \prod_{k=\mu+1}^\lambda \1\left\{ \fx\left(z+\sqrt{\ncovmat}u_{\sigma(\mu)}\right) < \fx\left(z+\sqrt{\ncovmat}u_{\sigma(k)}\right) \right\} 
			\\ 
			 \times \eta\left(u_{\sigma(1)},\dots,u_{\sigma(\mu)} \right) \densityd(u_1)\dots\densityd(u_\lambda) \mathrm{d}u_1\dots\mathrm{d}u_\lambda .
		\end{multline*}
		However, observe that, for each $k=\mu+1,\dots,\lambda$, we have
		$$
		\int \1\left\{ \fx\left(z+\sqrt{\ncovmat}u_{\sigma(\mu)}\right) < \fx\left(z+\sqrt{\ncovmat}u_{\sigma(k)}\right) \right\} \densityd(u_{\sigma(k)}) \mathrm{d}u_{\sigma(k)} = 1- Q_{\del{\theta}\nnew{z,\ncovmat}}^{\fx}\left(u_{\sigma(\mu)}\right) .
		$$
		We deduce \del{then }the desired result. 
		\new{Since the composition of lower semicontinuous functions is lower semicontinuous and since $f$ is continuous, when $R$ is continuous, the function
			$((z,p,q,\hat\ncovmat,r),v)\mapsto p_{z,\rncovmat/R(\rncovmat)}(v)$ is lower semicontinuous.}
	\end{proof}
	
	Furthermore, under \del{some }assumptions \del{that are }detailed in \Cref{sec:assumptions}, 
		we verify that \Cref{A4:lsc-distribution} and \Cref{A5:C1-update} hold. 
	
	\newcommand{\F}{\nnew{F_\Theta}}	
	\newcommand{\aalpha}{\nnew{\alpha_\Theta}}
	\begin{proposition}
		\label{p:CMA-satisfies-assumptions}
		Suppose that the objective function $f$ satisfies \Cref{F1:negligible-level-sets}-\Cref{F2:scaling-invariant}, 
			that the normalization function $R$ satisfies \Cref{R1:homogeneous}-\Cref{R2:C1},
			and that the stepsize change $\Gamma$ is such that \Cref{G1:C1} hold.
		Suppose moreover that $\rho$ satisfies \Cref{rho1:homogeneous}-\Cref{rho2:C1}.
		Consider the Markov chain $\Theta = \{(\Z_t,p_t,q_t, \rncovmat_t,r_t)\}_{t\in\bN}$ defined by \eqref{eq:smooth-normalized-chain-def}. Define the functions $\F$\done{$F_\Theta$} and $\aalpha$\done{$\alpha_\Theta$} via \eqref{eq:def-F} and \eqref{eq:def-alpha} respectively. Then, $\Theta$ follows \eqref{eq:control-modelTheta}, and \Cref{A4:lsc-distribution}-\Cref{A5:C1-update} hold.
	\end{proposition}
	
	\begin{proof}
		By \Cref{l:density-alpha}, we find that \Cref{A4:lsc-distribution} holds \done{\Cref{l:density-alpha} says nothing on the lsc of the density - state it and make a small proof ?}, with $\zeta_{\cV}$ being the Lebesgue measure on $\cV=\bR^{d\mu}$. Furthermore, using \Cref{R2:C1}, \Cref{rho2:C1} and \Cref{G1:C1}, we deduce, by composition, that \Cref{A5:C1-update} is satisfied.
	\end{proof}
	
	\subsection{Finding steadily attracting states}\label{sec:steadily-attracting-state}
	
	In this section and in \Cref{sec:rank-condition-cma}, 
		we prove that the control model \eqref{eq:control-modelTheta} satisfies condition \Cref{H3:controllability-condition}.\footnote{
			Or condition \Cref{H4:controllability-subchain} if we assume no cumulation on the stepsize or the covariance matrix.} 
	This is required to apply \Cref{t:verifiable-conditions}\footnote{Or \Cref{c:verifiable-conditions-subchain}.} 
		and find that the Markov chain $\Theta$ obeying to \eqref{eq:smooth-normalized-chain-def} is an irreducible aperiodic T-chain. 
	In this section, we focus on\del{ finding} the existence of steadily attracting states.
	This is formalized in the next proposition.
	
	\begin{proposition}
		\label{p:steadily-attracting-state-for-CMA}
		Suppose that the objective function $f$ satisfies \Cref{F1:negligible-level-sets}-\Cref{F2:scaling-invariant}, 
		that the stepsize change satisfies \nnew{\Cref{G1:C1}-}\Cref{G2:unbounded}, 
		that the normalization functions $R$ and $\rho$ satisfy \Cref{R1:homogeneous}-\Cref{R2:C1} and \Cref{rho1:homogeneous}-\Cref{rho2:C1} respectively, 
		and that the sampling distribution $\pUd$ is such that \Cref{P1:density} holds. 
		
		Then
		$(0,0,0,\Id,1-c_1-c_\mu)$ is a steadily attracting state for the control model \eqref{eq:control-modelTheta} 
		with the functions $F_\Theta$ and $\alpha_\Theta$ given by \eqref{eq:def-F} and \eqref{eq:def-alpha}, respectively. 
	\end{proposition}
	
	\begin{proof}
		Let $\theta_0\del{=(\Z_0,p_0,q_0, \rncovmat_0,r_0)}\in\cX$. \nnew{By \Cref{l:path-step-1}, we find $v_1$ such that $S_{\theta_0}^1(v_1) = (0,p_1,q_1,\rncovmat_1,r_1)$. If $q_1 \neq 0$, by  \Cref{l:path-step-2}, we set $v_2, v_3$  such that $S_{\theta_0}^3(v_{1:3}) = (0,p_3,0,\rncovmat_3,r_3)$. Using \Cref{l:path-step-3} we reach via a $4(d-1)$ steps a state $\theta=(0,\cdot,0,\Id,\cdot)$. Using \Cref{l:path-step-4}, we complete the path $v_1$ in case $q_1=0$ or $v_1, v_2, v_3$ otherwise into}
\del{		By \Cref{l:path-step-1,l:path-step-2,l:path-step-3,l:path-step-4}, }%
\del{there exists }%
$v_{1:\infty}=(v_1,v_2,\dots)\in\overline{\cO_{\theta_0}^\infty}$ such that $\lim_{k\to\infty} S_{\theta_0}^k (v_{1:k})=(0,0,0,\Id,1-c_1-c_\mu)$. 
		This implies~\cite[Corollary 4.5]{gissler2024irreducibility} that $(0,0,0,\Id,1-c_1-c_\mu)$ is a steadily attractive state.
	\end{proof}
	
	The proof of \Cref{p:steadily-attracting-state-for-CMA} relies on \Cref{l:path-step-1,l:path-step-2,l:path-step-3,l:path-step-4} below. 
	First, we state the next proposition, 
	which is useful to provide candidates for the paths between an initial state 
	and the steadily attracting state given by \Cref{p:steadily-attracting-state-for-CMA}.

	\begin{proposition}
		\label{p:characterization-control-sets-cma}
		In the context of \Cref{p:steadily-attracting-state-for-CMA}, let $\theta_0=(\Z_0,p_0,q_0,{\rncovmat}_0,r_0)\in\cX$, \nnew{let $k \geq 1$} and $v_{1:k}=(v_1,\dots,v_k)\in\cV^k$ be such that for $i=1,\dots,k$, we have $v_i=[v_i^1,\dots,v_i^\mu]\in\bR^{d\mu}$ with $v_i^1=\dots=v_i^\mu\in\bR^d$. Then, $v_{1:k}\in\overline{\cO_{\theta_0}^k}$.
	\end{proposition}
	\begin{proof}
		We prove here that $v_1 =\nnew{[\bar v_1,\ldots, \bar v_1]} \in\overline{ \cO^1_{\theta_{0}}}$\done{Armand: validate}.
		By \Cref{l:density-alpha}, it is sufficient to prove that there exists a sequence
			$\{ w_n = [w_n^1,\dots,w_n^\mu] \in\bR^{d\mu}\}_{n\in\bN}$ which converges to $v_1$ such that  $p_{z_0,\ncovmat_0}(\ncovmat_0^{-1/2}w_n)>0$ for all $n\in\bN$,
			where $\ncovmat_0=R(\rncovmat_0)^{-1}\rncovmat_0$
			 and $p_{z_0,\ncovmat_0}$ is the density defined via \eqref{eq:density}.
		Moreover, by \Cref{P1:density} and by definition of $p_{z_0,\ncovmat_0}$, it is sufficient to prove that for every $n\in\bN$,
		$
		f(z_0+w_n^1) < \dots < f(z_0+w_n^\mu) .
		$
		Furthermore, by \Cref{F1:negligible-level-sets}, for every $n\in\bN$, there exists $z_n^1,\dots,z_n^\mu \in \mathsf B(z_0+\nnew{\bar v_1},\del{1/k}1/n)$\done{Armand: validate red} such that
		$
		f(z_n^1)<\dots<f(z_n^\mu).
		$
		\del{The proof ends by taking}\new{We take} $w_n^i= z_n^i-z_0$ for $i=1,\dots,\mu$ and $n\in\bN$. \nnew{Then $w_n$ converges to $v_1$ and belongs to $\cO^1_{\theta_{0}}$, so that}  $v_1\in\overline{ \cO^1_{\theta_{0}}}$. \new{Similarly, $v_2\in\overline{ \cO^1_{S_{\theta_0}^1(x_1)}}$ for all $x_1$ and using the continuity of $v \mapsto S_{\theta_0}^1(v)$ in $v_1$, we deduce that $v_{1:2} \in \overline{ \cO^2_{\theta_{0}}} = \overline{\{(x_1,x_2) | p_{\theta_0}^1(x_1) \times p_{S_{\theta_0}^1(x_1)}(x_2) > 0 \}}$. } \new{Similarly we obtain that $v_{1:k}\in\overline{ \cO^k_{\theta_{0}}}$.}
		\del{Then, by concatenation \anne{I don't see a proof which is simply by concatenation - we should construct a sequence that converges to $v_{1:k}$ that belong to $\cO^k_{\theta_{0}}$ -  I  might miss something so maybe detail what you mean}\armand{if $v_1\in\overline{ \cO^1_{\theta_{0}}}$ and $v_2\in\overline{ \cO^1_{S_{\theta_0}^1(v_1)}}$, then \anne{the "then" hides the argument of continuity you need since $v_1$ is not in the control set, you will not place yourself in $S_{\theta_0}^1(v_1)$ and take an element from $ \cO^1_{S_{\theta_0}^1(v_1)}$ arbitrarily close} $v_{1:2}\in\overline{ \cO^2_{\theta_0}}$}, we would obtain $v_{1:k}\in\overline{ \cO^k_{\theta_{0}}}$.}
	\end{proof}

	The following lemma is the first step to build a path between an arbitrary initial state $\theta_0\in\cX$ 
		to the steadily attracting state $\theta^*=(0,0,0,\Id,1-c_1-c_\mu)$ given in \Cref{p:steadily-attracting-state-for-CMA}. 
	More precisely, it shows that from $\theta_0\in\cX$, we can reach via a one-step path a state $\theta_1$ such that $\Z_1=0$.
	
	\begin{lemma}
		\label{l:path-step-1}
		In the context of \Cref{p:steadily-attracting-state-for-CMA}, let $\theta_0=(\Z_0,p_0,q_0, \rncovmat_0,r_0)\in\cX$. 
		Then there exists $\theta_1=(0,p_1,q_1, \rncovmat_1,r_1)\in\cX$ and $v_1\in\overline{\cO^1_{\theta_0}}$ such that $S_{\theta_0}^1(v_1)=\theta_1$.
		Moreover, we can choose $v_1$ \new{as a function of $z_0$} such that $v_1$ goes to $0$ when $z_0$ tends to $0$.
	\end{lemma}
	\begin{proof}
		\del{Set}Let $v_1=-c_m^{-1}\times [\Z_0,\dots,\Z_0]\in(\bR^d)^\mu$\del{, which}\new{. It} belongs to $\overline{ \cO^1_{\theta_0}}$ by \Cref{p:characterization-control-sets-cma}.
		Set $\theta_1 = (z_1,p_1,q_1,\rncovmat_1,r_1) =S_{\theta_0}^1(v_1)$. 
			Then, $ z_1  =  \Fz (z_0,p_0,q_0,\rncovmat_0/R(\rncovmat_0),r_0;v_1) = r_1^{-1/2}\Gamma(p_1)^{-1}\times (z_0 - c_m \times c_m^{-1} \vwm^\top [z_0,\dots,z_0]) = 0$, see \eqref{eq:update-z}. We have used in particular that $\sum \wi=1$.
	\end{proof}

	We make the following observation when the mean $z_0$ is in $0$: 
		by performing one step via $v_1=(u_1,\dots,u_1)$ \nnew{for any $u_1$ in $\mathbb{R}^d$}, we can find a zero mean again \nnew{in two steps} by choosing a \del{step}\nnew{path $v_{1:2}$ appropriately.}\del{ accordingly.}
	\begin{lemma}\label{l:step-back-to-mean-0}
		In the context of \Cref{p:steadily-attracting-state-for-CMA}, let $\theta_0=(0,p_0,q_0, \rncovmat_0,r_0)\in\cX$.
		Then, 
			given $v_1= (u_1,\dots,u_1)\in\overline{ \cO^1_{\theta_{0}}}$ for some $u_1\in\bR^d$,
			and by defining $\theta_1 = (z_1,p_1,q_1,\rncovmat_1,r_1)=S^1_{\theta_0}(v_1)$ and
			$v_2 = -r_1^{-1/2}\Gamma(p_1)^{-1}v_1$, we have that
			$v_{1:2}=[v_1,v_2]\in\overline{ \cO^2_{\theta_0}}$ and 
			$\theta_2= (z_2,p_2,q_2,\rncovmat_2,r_2)= S^2_{\theta_0}(v_{1:2})$ satisfies
			$z_2=0$.
	\end{lemma}
	\begin{proof}
		By \Cref{p:characterization-control-sets-cma}, we have $v_1\in\overline{ \cO^1_{\theta_{0}}}$ and $v_{1:2}=[v_1,v_2]\in\overline{ \cO^2_{\theta_0}}$.
		Moreover, we have
		$$
			z_2 = r_2^{-1/2}\Gamma(p_2)^{-1} \times \left(  r_1^{-1/2}\Gamma(p_1)^{-1} \times (0+c_m u_1)  - c_m    r_1^{-1/2}\Gamma(p_1)^{-1} u_1      \right) =0
		$$
		ending the proof.
	\end{proof}

	Next, from any initial state $\theta_0\in\cX$ with $\Z_0=0$, we reach via a two-steps path a state $\theta_2\in\cX$ with $\Z_{2}=0$ and $q_{2}=0$.
	
	\begin{lemma}
		\label{l:path-step-2}
		In the context of \Cref{p:steadily-attracting-state-for-CMA}, let $\theta_0=(0,p_0,q_0, \rncovmat_0,r_0)\in\cX$ such that $q_0\neq0$. Then, there exist $\theta_2=(0,p_2,0, \rncovmat_2,r_2)\in\cX$ and $v_{1:2}\in\overline{ \cO^2_{\theta_0}}$ such that $S_{\theta_0}^2(v_{1:2})=\theta_2$. Moreover, we can choose $v_{1:2}$ such that $v_{1:2}\to0$ when $q_0$ tends to $0$.
	\end{lemma}
	\begin{proof}
		
		Let $u_1\in\bR^d$ and set $v_1=(u_1,\dots,u_1)$. It belongs to $\overline{ \cO^1_{\theta_0}}$ by \Cref{p:characterization-control-sets-cma}. Then, define  $$\theta_1=(\Z_1,p_1,q_1, \rncovmat_1,r_1)= \F\left(\theta_0,\aalpha\left(\theta_0,v_1 \right) \right) = S_{\theta_0}^1(v_1) \enspace .$$
		Then, define $v_2=-r_1^{-1/2}\Gamma(p_1)^{-1} v_1\in\overline{ \cO^1_{\theta_1}}$, and
		$$\theta_2=(\Z_2,p_2,q_2, \rncovmat_2,r_2)= \F\left(\theta_1,\aalpha\left(\theta_1,v_2 \right) \right) = S^2_{\theta_0}(v_{1:2}) \enspace .$$
		Then, by \Cref{l:step-back-to-mean-0},
		 $v_{1:2}=(v_1,v_2)\in\overline{ \cO^2_{\theta_0}}$ and $z_2=0$.
		\del{$$
		\Z_2 = c_m r_2^{-1/2}\Gamma(p_2)^{-1} \times \left[r_1^{-1/2}\Gamma(p_1)^{-1}  u_1 - r_1^{-1/2}\Gamma(p_1)^{-1}u_1 \right] = 0 
		$$
		and}
		Moreover,
		\begin{align*}
			q_2 & = (1-c_c)^2 r_0^{-1/2}r_1^{-1/2} q_0 + (1-c_c) r_1^{-1/2} \sqrt{c_c(2-c_c)  \mueff} u_1 \\ & ~~~~~~~~~~ - r_1^{-1/2} \Gamma(p_1)^{-1} \sqrt{c_c(2-c_c)\mueff} u_1 \\
			& = r_1^{-1/2} \times \left[ (1-c_c)^2(r_0^{-1/2} q_0 + \left( 1-c_c - \Gamma(p_1)^{-1}\right) \times \sqrt{c_c(2-c_c)\mueff} \times u_1 \right] \enspace .
		\end{align*}
		Let $\kappa\in\bR$, and choose $u_1= \kappa q_0$. Since $v\mapsto S_{\theta_0}^2(v)$ is continuous, then both $r_1$ and $q_2$ depend continuously on $\kappa$. Moreover, we have
		$$
		q_2 = r_1^{-1/2} \times \left[ (1-c_c)^2r_0^{-1/2} + \left( 1-c_c - \Gamma(p_1)^{-1}\right) \times \sqrt{c_c(2-c_c)\mueff} \kappa \right] \times q_0 \enspace.
		$$
		But, as $r_1>0$, and as 
		$
		\Gamma(p_1)^{-1} = \Gamma\left(  (1-c_\sigma)p_0+\sqrt{c_\sigma(2-c_\sigma)\mueff R(\rncovmat_0) } \kappa \rncovmat_0^{-1/2}q_0\right)^{-1}
		$
is less that $1-c_c$ when $\kappa\to\pm\infty$ by \Cref{G2:unbounded}, then by the intermediate value theorem (since $\Gamma$ is continuous by \Cref{G1:C1}), there exists $\kappa\in\bR$ such that $q_2=0$.\del{This ends the proof.} \nnew{With the above choice of $u_1= \kappa q_0$, $v_1 = (u_1,\ldots,u_1)$ and $v_2=-r_1^{-1/2}\Gamma(p_1)^{-1} v_1\in\overline{ \cO^1_{\theta_1}}$, we see that $v_{1:2}\to0$ when $q_0$ tends to $0$.}
	\end{proof}
	
	\del{Then, f}\nnew{F}rom an initial state $\theta_0$ with $\Z_0=q_0=0$, 
		we reach via a $4(d-1)$-steps path a state $\theta_{4(d-1)}$ with $\Z_{4(d-1)}=q_{4(d-1)}=0$, 
		and $\rncovmat_{4(d-1)} =\Id$. 
	This is achieved by applying $(d-1)$ times the \del{next}\nnew{following} lemma 
		successively to the $k$-$\mathrm{th}$ largest (counted with multiplicity) eigenvalue of $\rncovmat_0$, for $k=2,\dots,d$.
	For the sake of conciseness, the proof of \Cref{l:path-step-3} is delayed to \Cref{app:proofs-steadily}.
	
	\begin{lemma}
		\label{l:path-step-3}
		In the context of \Cref{p:steadily-attracting-state-for-CMA}, let $\theta_0=(0,p_0,0, \rncovmat_0,r_0)$. 
		Consider an orthonormal basis $\mathcal{B}$ of $\bR^d$ composed of eigenvectors of $\rncovmat_0$ 
			such that the matrix $\rncovmat_0$ writes in the basis $\mathcal{B}$ as
		$$
		[\rncovmat_0]_{\mathcal{B}} = \diag\left( \lambda_1,\dots,\lambda_d \right) ,
		$$
		with $\lambda_1=\lambda_2=\dots=\lambda_{k-1} \geqslant \lambda_{k} \geqslant \dots \geqslant \lambda_d$ for some $2\leqslant k\leqslant d$. Then, there exists $\gamma>0$, such that the matrix $\rncovmat_4$ defined by 
		\begin{equation}
		\label{eq:equal-eig}
		[ \rncovmat_4]_{\mathcal{B}} =\gamma\times\diag\left( \lambda_1,\dots,\lambda_{k-1},\lambda_{k-1},\lambda_{k+1},\dots,\lambda_d \right) ,
		\end{equation}
		is such that for some $p_4\in\bR^d$ and $r_4>0$, and $v_{1:4}\in\overline{\cO^4_{\theta_0}}$, we have $S_{\theta_0}^4(v_{1:4})=\theta_4=(0,p_4,0,{\rncovmat}_4,r_4)$.
	\end{lemma}

	Finally, as stated in the next lemma, from an initial state $\theta_0\in\cX$ such that $\Z_0=q_0=0$ and $\rncovmat_0=\Id$, we can reach any neighborhood of the state $\theta^*=(0,0,0,\Id,1-c_1-c_\mu)$.
	
	\begin{lemma}
		\label{l:path-step-4}
		In the context of \Cref{p:steadily-attracting-state-for-CMA}, let $\theta_0=(0,p_0,0,\Id,r_0)\in\cX$. Then there exists $v_{1:\infty}\in\overline{\cO_{\theta_0}^\infty}$ such that $\lim S^t_{\theta_0}(v_{1:t})=(0,0,0,\Id,1-c_1-c_\mu)$ when $t\to\infty$.
	\end{lemma}
	\begin{proof}
		Define $v_{1:\infty}$, by $v_t=0\in\bR^{d\times\mu}$ for all $t\geqslant1$. By \Cref{p:characterization-control-sets-cma}, we have $v_{1:\infty}\in\overline{ \cO^\infty_{\theta_{0}}}$. Denote
		$$
		\theta_{t+1}=(\Z_{t+1},p_{t+1},q_{t+1},{\rncovmat}_{t+1},r_{t+1}) = \F(\theta_t,\aalpha(\theta_t,v_{t+1})) .
		$$
		\new{Since, $\theta_0=(0,p_0,0,\Id,r_0)$ and $v_t=0$,}\del{so that,} by induction, we have
		$
		{\rncovmat}_{t+1} =  \Id$, 
		$\Z_{t+1} = 0$,
		$q_{t+1} = 0$
		 $r_{t+1} = R((1-c_1-c_\mu)\Id)=1-c_1-c_\mu $ and
		$ p_{t+1} = (1-c_\sigma) p_t.
		$
		Since $0\leqslant 1-c_\sigma <1$, then $(\Z_t,p_t,q_t,{\rncovmat}_t,r_t)$ tends to $(0,0,0,\Id,1-c_1-c_\mu)$ when $t\to\infty$, ending the proof.
	\end{proof}

		Lastly, \del{we state}\new{as} a consequence of \Cref{p:steadily-attracting-state-for-CMA}, \nnew{we prove that given any normalized covariance matrix $\rncovmat^*\in\Sdpp$ such that $\rho(\rncovmat^*)=1$, we can find a value for the path $p^*\in\bR^d$, for the variable $r^*>0$, such that the state $\theta^*=(0,p^*,0,\rncovmat^*,r^*)\in\cX$ with normalized mean and normalized path for the rank-one update equal to zero is steadily attracting.
		}\del{	which \del{deduces that other states are}\nnew{find other} steadily attracting \del{too}\nnew{states}.}
		In \Cref{sec:rank-condition-cma}, we use these steadily attracting states
		to prove the controllability condition stated in \Cref{p:full-rank-control-matrix-cma}.
		
	\begin{corollary}\label{c:steadily-attracting-state-for-CMA}
		Consider the context of \Cref{p:steadily-attracting-state-for-CMA}.
		Let $\rncovmat^*\in\Sdpp$ be such that $\rho(\rncovmat^*)=1$. Then, there exist $p^*\in\bR^d$ and $r^*>0$ such that $\theta^*=(0,p^*,0,\rncovmat^*,r^*)\in\cX$ is a steadily attracting state.
	\end{corollary}

	\begin{proof}
		By \Cref{p:steadily-attracting-state-for-CMA}, we know that $\theta_0=(0,0,0,\Id,1-c_1-c_\mu)$ is a steadily attracting state. Hence, in order to prove that a state $\theta^*\in\cX$ is steadily attracting, it is sufficient, as explained below, to prove 
		\begin{itemize}
		\item[(i)] that there exist $k\in\bN$ and $v_{1:k}\in\overline{ \cO^k_{\theta_{0}}}$ such that $S_{\theta_0}^k(v_{1:k})=\theta^*$. 
		\end{itemize}
		
		Indeed, \nnew{assume we have proven (i) and}  let $V$ be a neighborhood of $\theta^*$ and let $\theta\in\cX$. Then, \nnew{by continuity of $w_{1:k} \mapsto S_{\theta_0}^k(w_{1:k})$ around $v_{1:k}$,} there exists $v_{1:k}^*\in\cO^k_{\theta_0}$ such that $S^k_{\theta_0}\nnew{(v_{1:k}^*)}\in V$. 
			Since $x\mapsto p_x^k(v_{1:k}^*)$ is \lsc and $x\mapsto S^k_x(v^*_{1:k})$ is continuous, 
			then there exists a neighborhood $U$ of $\theta_0$ such that for every $x \in U$, $p_x^k(v^*_{1:k})>0$, i.e., $v^*_{1:k}\in\cO^k_x$, 
			and $S_x^k(v_{1:k}^*)\in V$. 
			Moreover, since $\theta_0$ is steadily attracting, then there exists $T>0$ such that for every $t\geqslant T$, 
			there exists $w_{1:t}\in\cO^k_{\theta}$ such that $S^t_{\theta}(w_{1:t})\in U$, 
			hence $[w_{1:t},v_{1:k}^*]\in\cO^{t+k}_{\theta}$ and $S^{t+k}_{\theta}([w_{1:t},v_{1:k}^*])\in V$ \nnew{and hence $\theta^*$ is a steadily attracting state}.
		
		\nnew{Let $\rncovmat^*\in\Sdpp$ be such that $\rho(\rncovmat^*)=1$.} We proceed now as in \Cref{l:path-step-3} \nnew{to prove (i) for a state $\theta^*$ that is equal to $(0,p^*,0,\rncovmat^*,r^*)$ for $p^*$ and $r^*$ constructed below}. For $i=1,\dots,d$, let $\lambda_i$ be the $i$-$\mathrm{th}$ largest eigenvalue of $\rncovmat^*$ (counted with multiplicity), and $(e_1,\dots,e_d)$ an orthonormal basis of eigenvectors of $\rncovmat^*$ such that $\rncovmat^*e_i=\lambda_ie_i$. Then, let $\kappa$ and $\kappa'$ be real numbers, and by \Cref{p:characterization-control-sets-cma}, define $v_{1:4}\in\overline{\cO^4_{\theta_0}}$ by
		$$
		\begin{array}{l}
			v_1 = \kappa [e_1,\dots,e_1] \in\bR^{d\mu},\, 
			v_2 = - r_1^{-1/2} \Gamma(p_1)^{-1} v_1,\,
			v_3 = \kappa' [e_1,\dots,e_1],\,
			v_4 = -r_3^{-1/2} \Gamma(p_3)^{-1} v_3 ,
 		\end{array}
		$$
		where $\theta_t = (z_t,p_t,q_t,\rncovmat_t,r_t)=S^t_{\theta_0}(v_{1:t})$ for $t=1,2,3,4$. Then, as in the proof of \Cref{l:path-step-3}, there exist values of $\kappa$ and $\kappa'$ in $\bR$ such that $z_4=q_4=0$ and such that there exists $\rho_4>0$ with $\rncovmat_4$ satisfying $\rncovmat_4 e_1 = \rho_4 (1-c_1-c_\mu)^{-4(d-1)} \lambda_1 e_1$ and $\rncovmat_4 e_k = \rho_4 (1-c_1-c_\mu)^4 e_k$ for $k=2,\dots,d$.
		
		Then, by repeating these steps with $e_2,\dots,e_d$ instead of $e_1$ and $\lambda_2,\dots,\lambda_d$ instead of $\lambda_1$, then there exist $v_{1:4d}\in\overline{\cO^{4d}_{\theta_0}}$ and $\rho_{4d}>0$ such that $\theta_{4d}=(z_{4d},p_{4d},q_{4d},\rncovmat_{4d},r_{4d})=S_{\theta_0}^{4d}(v_{1:4d})$ satisfies $z_{4d}=q_{4d}=0$ and $\rncovmat_{4d}e_k=\rho_{4d}\lambda_ke_k$ for $k=1,\dots,d$. \nnew{Hence $\rncovmat_{4d}=\rho_{4d} \rncovmat^*$.} But since $\rho(\rncovmat^*)=1$ \nnew{and $\rho(\rncovmat_{4d})=1$}, then \nnew{$\rho_{4d}=1$,} i.e.\ $\rncovmat_{4d}=\rncovmat^*$\del{, ending the proof}\nnew{ such that we have proven (i) for $\theta^*=(0,p_{4d},0,\rncovmat^*,r_{4d})$ and in turn that $\theta^*=(0,p_{4d},0,\rncovmat^*,r_{4d})$ is a steadily attracting state}.
	\end{proof}

	\subsection{Controllability condition}\label{sec:rank-condition-cma}
	\newcommand{\h}{\varepsilon}
	
	In the previous section, 
		we prove that the control model \eqref{eq:control-modelTheta} \del{possesses}\new{admits} steadily attracting states. 
	In the current section, 
		we prove that a controllability condition, 
		as required to satisfy the assumptions \Cref{H3:controllability-condition} or \Cref{H4:controllability-subchain}, 
		is satisfied at a steadily attracting state.
	By combining \Cref{c:steadily-attracting-state-for-CMA} 
		and the following \Cref{p:full-rank-control-matrix-cma}, we prove \Cref{H3:controllability-condition} 
		or \Cref{H4:controllability-subchain}. 
	\del{In this section, f}\new{F}or a finite-dimensional vectorial space $\mathsf E$ equipped with a norm $\|\cdot\|$ and an element $h\in\mathsf E$, 
			we use the notation $o(h)$, respectively $O(h)$, to be understood as $o(\|h\|)$, respectively $O(\|h\|)$.
		Besides, it does not depend on the chosen norm, since all norms on a finite-dimensional space induce the same topology.

	\del{In this section, we prove the controllability condition required in \Cref{t:verifiable-conditions} or in \Cref{cor:normalizedMCwithoutcumul} for the control model \eqref{eq:DCM-normalized-chain}. Depending on whether we have $c_c=1$ and/or $c_\sigma=1$, the controllability condition is different, as formalized in the next proposition.}%
	
	\begin{proposition}
		\label{p:full-rank-control-matrix-cma}
		Suppose that the objective function $f\colon\bR^d\to\bR$, 
			the normalization functions $R$ and $\rho$, 
			the stepsize change $\Gamma\colon\bR^d\to\bR_{++}$ 
			and the sampling distribution $\pUd$ satisfy 
			\Cref{F1:negligible-level-sets}-\Cref{F2:scaling-invariant}, \Cref{R1:homogeneous}-\Cref{R3:differentiable}, \Cref{rho1:homogeneous}-\Cref{rho2:C1}, \Cref{G1:C1}-\Cref{G3:G(0)<1} and \Cref{P1:density}, respectively.
		
		Consider the control model \eqref{eq:control-modelTheta} with the functions $F_\Theta$ and $\alpha_\Theta$ defined by \eqref{eq:def-F} and \eqref{eq:def-alpha} respectively.

		Then, there exist \nnew{a steadily attracting state} $\theta_0\in\cX$\del{ a steadily attracting state}, $T>0$ and $v_{1:T}\in\overline{\cO_{\theta_0}^{T}}$ such that $S_{\theta_0}^T$ is differentiable at $v_{1:T}$, and, by denoting $(z_T,p_T,q_T,\rncovmat_T,r_T)=S_{\theta_0}^T(v_{1:T})$, we have
		\begin{enumerate}
			\item[(a)] if $c_c\neq 1$, $c_\sigma\neq1$, $1-c_c\neq (1-c_\sigma)\sqrt{1-c_1-c_\mu}$ and $c_\mu>0$, then  $\mathcal{D} S^T_{\theta_0}(v_{1:T})$ is of maximal rank;
			\item[(b)] if $c_c=1$ and $c_\sigma\neq 1$, then, for every $(z,p,\ncovmat)\in\bR^d\times\bR^d\times \mathrm T_{\rncovmat_T} \rho^{-1}(\{1\})$, there exist $q\in\bR^d$ and $r\del{>0}\in\bR$ such that $(z,p,q,\ncovmat,r)\in\mathrm{rge}~\mathcal{D} S_{\theta_0}^T(v_{1:T})$;
			\item[(c)] if $c_c\neq1$, $c_\sigma=1$ and $c_\mu>0$, then, for every $(z,q,\ncovmat,r)\in\bR^d\times\bR^d\times \mathrm T_{\rncovmat_T} \rho^{-1}(\{1\})\times\bR$, there exists $p\in\bR^d$ such that $(z,p,q,\ncovmat,r)\in\mathrm{rge}~\mathcal{D} S_{\theta_0}^T(v_{1:T})$;
			\item[(d)] if $c_c=c_\sigma=1$, then, for every $(z,\ncovmat)\in\bR^d\times \mathrm T_{\rncovmat_T} \rho^{-1}(\{1\})$, there exist $(p,q)\in\bR^d\times\bR^d$ and $r\del{>0}\in\bR$, such that $(z,p,q,\ncovmat,r)\in\mathrm{rge}~\mathcal{D} S_{\theta_0}^T(v_{1:T})$.
		\end{enumerate}
	\end{proposition}
	
	Before proving \Cref{p:full-rank-control-matrix-cma}, we state the two following lemmas, which characterize the derivatives of the normalization function $\rho$ and of the transition map $S_{\theta_0}^1$, respectively.
	
	\begin{lemma}
		\label{l:diff-pos-hom}
		\del{Suppose that the normalization function $R$ satisfies \Cref{R1:homogeneous}}\new{Consider a positively homogeneous function $R\colon\Sdpp\to\bR_{++}$}\done{we do not need that $\rho(I_d) = 1$, do we?}. 
		Let $\mathbf A\in\Sdpp$ and $\gamma>0$ \new{and suppose that $R$ is differentiable at $\mathbf A$}. 
		Then, \new{$R$ is differentiable at $\gamma\mathbf A$ and}
		$\mathcal{D} R(\gamma \mathbf A) = \mathcal{D}R(\mathbf A)$.
	\end{lemma}
	\begin{proof}
		\del{By \Cref{rho2:C1}, the function $\rho$ is differentiable on $\Sdpp$.}
		By Taylor expansion, we have, when $\mathbf H\in\Sd$ tends to $0$, that
		\begin{align*}
			R(\gamma \mathbf A+ \mathbf H) & = \gamma \times R(\mathbf A + \gamma^{-1}\mathbf H) = \gamma R(\mathbf A) + \gamma \mathcal{D}R(\mathbf A) \gamma^{-1}\mathbf H + o(\mathbf H) \\
			& = R(\gamma \mathbf A) +  \mathcal{D}R(\mathbf A) \mathbf H + o(\mathbf H) 
		\end{align*}
		and thus,
		 \new{by Taylor expansion, $R$ is differentiable at $\gamma\mathbf A$ and}
		 $\mathcal{D} R(\gamma \mathbf A) = \mathcal{D}R(\mathbf A)$.
	\end{proof}

	\begin{lemma}\label{l:S-differentiable-in}
		Suppose \Cref{G1:C1}, \Cref{R1:homogeneous}, \Cref{R2:C1} and \Cref{rho2:C1}.
		Let $\theta_0 = (z_0,p_0,q_0,\rncovmat_0,r_0)\in\cX$ and $v_1\in\overline{ \cO^1_{\theta_{0}}}$. 
		Then, if 
			(a) $z_0=q_0=0$ and $v_1=0$, or if 
			(b) $p_1 =\Fsig(p_0,R(\rncovmat_0)^{-1}\rncovmat_0;v_1) \neq0$,
			 and if moreover 
			 $R$ is differentiable in $\rncovmat_1= \FK(q_0,R(\rncovmat_0)^{-1}\rncovmat_0,r_0;v_1)$ 
			(see \eqref{eq:FSigma} for the definition of $\FK$),
			then $\nnew{v} \in \cV \mapsto S^1_{\theta_0}\nnew{(v)}$ is differentiable at $v_1$.
	\end{lemma}

	\begin{proof}
		Suppose (a). Then, for $h_1=(h_1^1,\dots,h_1^\mu)\in\cV$, \new{using the update equations \eqref{eq:update-z}, \eqref{eq:update-p}, \eqref{eq:update-q}, \eqref{eq:update-Sigma}, \eqref{eq:update-r}} we have, when $h_1\to0$,
		\begin{align*}
			S_{\theta_0}^1 (v_1+h_1) & = 
				\begin{pmatrix}
				\frac{\new{R((1-c_1-c_\mu) R(\rncovmat_0)^{-1} \rncovmat_0 + o(\|h_1\|))^{-1/2}}(0 + c_m  \vwm^\top h_1)}{
					\Gamma((1-c_\sigma)p_0+\sqrt{c_\sigma(2-c_\sigma)\mueff}R(\rncovmat_0)^{1/2}\rncovmat_0^{-1/2}\vwm^\top h_1)}  \\ 
				(1-c_\sigma)p_0+\sqrt{c_\sigma(2-c_\sigma)\mueff }R(\rncovmat_0)^{1/2}\rncovmat_0^{-1/2}\vwm^\top h_1 \\
				0+\sqrt{c_c(2-c_v)\mueff}   \vwm^\top h_1 \\
				\cfrac{(1-c_1-c_\mu) R(\rncovmat_0)^{-1} \rncovmat_0 + o(\|h_1\|)}{\rho((1-c_1-c_\mu) R(\rncovmat_0)^{-1} \rncovmat_0 + o(\|h_1\|))}   \\
				R((1-c_1-c_\mu) R(\rncovmat_0)^{-1} \rncovmat_0 + o(\|h_1\|))
				\end{pmatrix}
			\enspace .
		\end{align*}
		However, by \Cref{G1:C1}, the stepsize change $\Gamma$ is locally Lispchitz, hence
		$$
			\Gamma((1-c_\sigma)p_0+\sqrt{c_\sigma(2-c_\sigma)\mueff}R(\rncovmat_0)^{1/2}\rncovmat_0^{-1/2}\vwm^\top h_1) = \Gamma((1-c_\sigma)p_0) + O(\|h_1\|) = 	\Gamma((1-c_\sigma)p_0) + o(1)  \enspace.
		$$
		Moreover, by \Cref{rho2:C1}, $\rho$ is differentiable at $(1-c_1-c_\mu) R(\rncovmat_0)^{-1} \rncovmat_0$. Hence by Taylor expansion
		$$
			\rho((1-c_1-c_\mu) R(\rncovmat_0)^{-1} \rncovmat_0 + o(\|h_1\|)) = \rho((1-c_1-c_\mu) R(\rncovmat_0)^{-1} \rncovmat_0 )+ o(\|h_1\|) \enspace.
		$$
		Likewise, by assumption, $R$ is differentiable at $(1-c_1-c_\mu) R(\rncovmat_0)^{-1} \rncovmat_0$. Thus,
		$$
		R((1-c_1-c_\mu) R(\rncovmat_0)^{-1} \rncovmat_0 + o(\|h_1\|)) = R((1-c_1-c_\mu) R(\rncovmat_0)^{-1} \rncovmat_0 )+ o(\|h_1\|) \enspace.
		$$
		Therefore,
		\begin{align*}
			S_{\theta_0}^1 (v_1+h_1) & = S_{\theta_0}^1 (v_1) + 
				\begin{pmatrix}
				\new{R((1-c_1-c_\mu) R(\rncovmat_0)^{-1} \rncovmat_0)^{-1/2}}\Gamma((1-c_\sigma)p_0)^{\new{-1}} c_m \vwm^\top h_1  \\
				\sqrt{c_\sigma(2-c_\sigma)\mueff}\new{R(\rncovmat_0)^{1/2} \rncovmat_0^{-1/2}}\vwm^\top h_1 \\
				\sqrt{c_c(2-c_v)\mueff}   \vwm^\top h_1 \\
				0  \\
				0
				\end{pmatrix}
			+ o(\|h_1\|) \enspace,
		\end{align*}
		which proves by Taylor expansion that $S^1_{\theta_0}$ is differentiable at $v_1=0$.
		
		Now, suppose (b). 
		By \Cref{G1:C1}, $\Gamma$ is differentiable at $p_1$, 
			by \Cref{rho2:C1}, $\rho$ is differentiable on $\Sdpp$, 
			and by assumption, $R$ is differentiable at $\rncovmat_1 \nnew{= \mathbf A_1 / \rho(\mathbf A_1)}$ \nnew{where $\mathbf A_1 =(1-c_1-c_\mu)R(\rncovmat_0)^{-1} \rncovmat_0 + c_1 q_1 (q_1)^\top + c_\mu \sum_{i=1}^\mu \wic v_1^i (v_1^i)^\top$}. 
			\new{Since $R$ is positively homogeneous, it is also differentiable in any multiple by a scalar of \del{$\rncovmat_t$}\nnew{$\rncovmat_1$, so in $\mathbf A_1$}.}
		Thus, by composition $S^1_{\theta_0}$ is differentiable at $v_1$.
	\end{proof}

We prove now \Cref{p:full-rank-control-matrix-cma}. 
The first step of the proof consists in the following lemma which applies to all cases (a)-(d) in \Cref{p:full-rank-control-matrix-cma}. 
It provides a path $v_{1:T}\in\overline{ \cO^T_{\theta_0}}$ 
	(where $\theta_0$ is the steadily attracting state found in \Cref{sec:steadily-attracting-state}) 
	such that the range of $\mathcal{D}S_{\theta_0}^T$ covers all elements in 
	the tangent space relative to the covariance matrix variable.
	The proof of \Cref{l:contrabillity-condition-step0} is delayed to \Cref{app:proofs-controllability}.

\begin{lemma}
		\label{l:contrabillity-condition-step0}
	Suppose that the objective function $f\colon\bR^d\to\bR$, 
	the normalization functions $R$ and $\rho$, 
	the stepsize change $\Gamma\colon\bR^d\to\bR_{++}$ 
	and the sampling distribution $\pUd$ satisfy \Cref{F1:negligible-level-sets}-\Cref{F2:scaling-invariant}, \Cref{R1:homogeneous}-\Cref{R3:differentiable}, \Cref{rho1:homogeneous}-\Cref{rho2:C1}, \Cref{G1:C1}-\Cref{G3:G(0)<1} and \Cref{P1:density}, respectively.
	 Consider the control model \eqref{eq:control-modelTheta} with the functions $F_\Theta$ and $\alpha_\Theta$ defined by \eqref{eq:def-F} and \eqref{eq:def-alpha} respectively.
	
	Then, there exist a steadily attracting state $\theta_0\in\cX$, $T\in\bN$, $v_{1:T}\in\overline{ \cO^T_{\theta_{0}}}$, and $\mathsf W$ a subspace of $\cV^{T}$, such that:
	\begin{enumerate}
		\item[(i)] $S_{\theta_0}^T$ is differentiable at $v_{1:T}$;
		\item[(ii)] for every $h_{\ncovmat}\in \mathrm T_{\rncovmat_T}\rho^{-1}(\{1\})$, there exists $h_z, h_p\in\bR^d$, $h_r\in\bR$, and $h_{1:T}\in\mathsf W$ such that $\mathcal{D}S_{\theta_0}^T(v_{1:T})h_{1:T} = [h_z, h_p, 0, h_{\ncovmat}, h_r]$;
		\item[(iii)] $z_T=q_T=0$ and $p_T\neq 0$;
	\end{enumerate}
	where $\theta_t = (z_t,p_t,q_t,\rncovmat_{t},r_t)=S_{\theta_0}^t(v_{1:t})$ for $t=1,\dots,T$.
\end{lemma}

The next lemma is the second step of the \del{current}\del{present}proof of \Cref{p:full-rank-control-matrix-cma}. 
It deduces from \Cref{l:contrabillity-condition-step0} a path in 
which the transition map is differentiable and 
is of interest to apply \Cref{t:verifiable-conditions} or \Cref{c:verifiable-conditions-subchain}. 
It applies to all cases (a)-(d).
We delay once more the proof of \Cref{l:contrabillity-condition-step1} to \Cref{app:proofs-controllability}.

	\begin{lemma}
		\label{l:contrabillity-condition-step1}
		Suppose that the objective function $f\colon\bR^d\to\bR$, the normalization functions $R$ and $\rho$, the stepsize change $\Gamma\colon\bR^d\to\bR_{++}$ and the sampling distribution $\pUd$ satisfy \Cref{F1:negligible-level-sets}-\Cref{F2:scaling-invariant}, \Cref{R1:homogeneous}-\Cref{R3:differentiable}, \Cref{rho1:homogeneous}-\Cref{rho2:C1}, \Cref{G1:C1}-\Cref{G3:G(0)<1} and \Cref{P1:density}, respectively.
		%
		Consider the control model \eqref{eq:control-modelTheta} with the functions $F_\Theta$ and $\alpha_\Theta$ defined by \eqref{eq:def-F} and \eqref{eq:def-alpha} respectively.
		
		Then, there exist $\theta_0\in\cX$ a steadily attracting state, and $p\in\bR_{\neq 0}^d$, such that, for every $j\in\bN$,		
		there exist $T\in\bN$ and $v_{1:T}\in\overline{\cO_{\theta_0}^T}$, 
		with $S_{\theta_0}^T$ being differentiable at $v_{1:T}$, and
		\begin{equation}\label{eq:TE}
			S_{\theta_0}^T (v_{1:T}+h_{1:T}) = S_{\theta_0}^T(v_{1:T}) + \mathbf C_j \times L(h_{1:T}) + o(h_{1:T})
		\end{equation}
		for every $h_{1:T}\in\mathsf{W}_L$, where $\mathsf{W}_L$ is a well-chosen subspace of $\cV^T$\todo{footnote}\todo{to be discussed: isn't it the Tangent space ? We do not have latitute to "choose" it - do we? - Clarify why if we have it for the well-chosen subpace we have it fo the tangent space}\armand{the tangent space of $V^T$ is itself ($V^T$ is Euclidean). We choose a subspace $W_L$ because we want the matrix $C_j$ to come up}, $L\colon \mathsf{W}_L\to \bR^{s-1}\times\bR^d\times\bR^d\times\bR^d$ is a surjective linear map, $s=d(d+1)/2$, and $\mathbf C_j$ is a matrix of the form:
		\begin{equation}\label{eq:matrix-Cj}
			\mathbf C_j = \left[
			\begin{array}{ccccc}
				* & \cdots & * & 0 \hspace{1.5em} 0 & L^z \\
				\begin{array}{c}
					* \\ *
				\end{array} &
				\begin{array}{c}
					\cdots \\ \cdots
				\end{array} & 
				\begin{array}{c}
					* \\ *
				\end{array}  & 
				\mathbf L^{p,q}_j &
				\begin{array}{c}
					* \\ *
				\end{array} \\
				\mathbf L_1^{\ncovmat} & \dots & \mathbf L_{s-1}^{\ncovmat} & 0 \hspace{1.5em} 0 &  0  \\
				0 & \cdots & 0 & 0 \hspace{1.5em} 0 & 0
			\end{array}	
			\right]
		\end{equation}
		with $(\mathbf L_1^{\ncovmat},\dots,\mathbf L_{s-1}^{\ncovmat})$ being a basis of $\ker\mathcal{D}\rho(\rncovmat_T)$ (with $\theta_t=(\Z_t,p_t,q_t,\rncovmat_t,r_t)=S_{\theta_0}^t(v_{1:t})$ for $t=0,\dots,T$), $L_z\in\bR_{\neq 0}$, and 
		\begin{equation} \label{eq:Lpqj}
			\mathbf L^{p,q}_j = 
				\begin{bmatrix}
				(1-c_\sigma)^{3} c_{j+1}^p R(\rncovmat_{T})^{1/2}\rncovmat_{T}^{-1/2} & (1-c_\sigma)^{1} c_{j+3}^p R(\rncovmat_{T})^{1/2}\rncovmat_{T}^{-1/2} \\
				(1-c_c)^{3} (1-c_1-c_\mu)^{-3/2} d_{j+1}^p \Id & (1-c_c)^{1} (1-c_1-c_\mu)^{-1/2} d_{j+3}^p \Id 
				\end{bmatrix}
			\enspace ,
		\end{equation}
		 where $c_k^p \coloneqq (1-c_\sigma -(1-c_1-c_\mu)^{-1/2}\Gamma((1-c_\sigma)^k p))^{-1}\sqrt{c_\sigma(2-c_\sigma)\mueff}$ and $d_{k}^p\coloneqq (1-c_1-c_\mu)^{-1/2}\left[1-c_c- \Gamma((1-c_\sigma)^kp)^{-1}\right] $ $\sqrt{c_c(2-c_c)\mueff}$.
		 The symbol $*$ in \eqref{eq:matrix-Cj} represents the elements of the matrix $\mathbf C_j$ that we do not give explicitly
		 (their values do not change the rank of $\mathbf C_j$).
	\end{lemma}

	Next, in order to deduce the case (a) in \Cref{p:full-rank-control-matrix-cma} from \Cref{l:contrabillity-condition-step1}, 
	we first show in the next lemma that the matrix $\mathbf L^{p,q}_j$ defined via \eqref{eq:Lpqj} is invertible when the integer $j$ is sufficiently large.
	
	\begin{lemma}
		\label{l:controllability-condition-step2}
		In the context of \Cref{l:contrabillity-condition-step1}, there exists $j\in\bN$ such that, if $c_c\neq 1$, $c_\sigma\neq1$, $1-c_c\neq (1-c_\sigma)\sqrt{1-c_1-c_\mu}$, then the matrix $\mathbf L^{p,q}_j$ defined via \eqref{eq:Lpqj} is invertible.
	\end{lemma}
	
	\begin{proof}
We have, since \nnew{$c_\sigma \neq 1$, $c_c \neq 1$}:
			\begin{multline*}
				\begin{bmatrix}
					(1-c_\sigma)^{-1} R(\rncovmat_T)^{-1/2}\rncovmat_T^{1/2} & 0 \\
					0 & (1-c_c)^{-1} (1-c_1-c_\mu)^{1/2} \Id
				\end{bmatrix}\times \mathbf{L}_j^{p,q}
			\\ = \begin{bmatrix}
											(1-c_\sigma)^2 c_{j+1}^p \Id & c_{j+3}^p \Id \\
											(1-c_c)^2(1-c_1-c_\mu)^{-1} d_{j+1}^p \Id & d_{j+3}^p \Id
				\end{bmatrix} \enspace.
			\end{multline*}
			\done{in RHS matrix $c_{j+1}$ instead - $c_{j+3}$ - $d_{j+1}$ and $d_{j+3}$ instead}%
			Therefore, it is sufficient to find some $j\in\bN$ such that the RHS in the above equation is invertible, i.e., such that the matrix
		$$
			\mathbf A_j =	\begin{bmatrix}
				(1-c_\sigma)^2\left[ 1-c_\sigma-(1-c_1-c_\mu)^{-1/2} \Gamma_{j+1}^{-1} \right]  &	 1-c_\sigma-(1-c_1-c_\mu)^{-1/2} \Gamma_{j+3}^{-1}     \\
				(1-c_c)^2(1-c_1-c_\mu)^{-1} \left[ 1-c_c- \Gamma_{j+1}^{-1} \right]& 1-c_c- \Gamma_{j+3}^{-1} 
			\end{bmatrix}
			\enspace ,
			$$
			where $\Gamma_k = \Gamma((1-c_\sigma)^k p)$ for $k=j+2,j+4$,
		is full rank.
		\del{
		
		Since $\ncovmat_{T}^{-1/2}$ is invertible , we have
		\begin{align*}
			\det \mathbf L_j^{p,q}\neq 0 \done{L^{p,q}_j ?}& \Leftrightarrow \det 
				\begin{bmatrix}
				(1-c_\sigma)^2c_{j+2}^p  & c_{j+4}^p  \\ (1-c_c)^2(1-c_1-c_\mu)^{-1} d_{j+2}^p 
				& d_{j+4}^p 
				\end{bmatrix}
			\neq 0 \enspace .
		\end{align*}

		Denoting $T_1=T-5$\todo{where do you use $T_1$?}, we get
		$$
			\det \mathbf L^{p,q}_j\neq 0  \Leftrightarrow \det \mathbf A_j\neq 0
		$$
		where
		
		  and }%
		Moreover, when $j\to\infty$, by continuity of $\Gamma$ (by \Cref{G1:C1}), we have that $\Gamma_{j+1}$ and $\Gamma_{j+3}$ tend to $\Gamma(0)$. Hence,
		\begin{align*}
			\lim_{j\to\infty} &  \det \mathbf A_j
			 = \begin{vmatrix}
					(1-c_\sigma)^2(1-c_\sigma-(1-c_1-c_\mu)^{-1/2}\Gamma(0)^{-1}) & 1-c_\sigma - (1-c_1-c_\mu)^{-1/2} \Gamma(0)^{-1} \\
					(1-c_c)^2 (1-c_1-c_\mu)^{-1} (1-c_c-\Gamma(0)^{-1}) & 1-c_c-\Gamma(0)^{-1}
				\end{vmatrix} \\
			& = (1-c_\sigma-(1-c_1-c_\mu)^{-1/2}\Gamma(0)^{-1}) \times (1-c_c-\Gamma(0)^{-1}) \times
			\begin{vmatrix}
				(1-c_\sigma)^2 & 1 \\
				(1-c_c)^2 (1-c_1-c_\mu)^{-1}  & 1
			\end{vmatrix}  \del{\\
			& =
			\left( 1-c_\sigma-(1-c_1-c_\mu)^{-1/2}\Gamma(0)^{-1} \right)
				\times\left( 1-c_c-\Gamma(0)^{-1} \right) 
				\times \left((1-c_\sigma)^2-(1-c_c)^2(1-c_1-c_\mu)^{-1}\right) }\enspace ,
		\end{align*}
		where $\begin{vmatrix}
				a & b \\ c & d
			\end{vmatrix}$ denotes the determinant of the matrix $\begin{pmatrix}
				a & b \\ c & d
			\end{pmatrix}$. 
		However, $\Gamma(0)^{-1}>1$ (by \Cref{G3:G(0)<1}) and $(1-c_1-c_\mu)^{-1}>1$. Hence, there exists $j\in\bN$, such that, if $1-c_c\neq (1-c_\sigma)\sqrt{1-c_1-c_\mu}$, then $\det \mathbf L^{p,q}_j\neq 0$.
	\end{proof}

	We can now end the proof of \Cref{p:full-rank-control-matrix-cma}. Depending on the case (a)-(d), the end of the proof goes differently.
	We present here the proofs of cases (b) and (d), and we delay those of (a) and (d) to \Cref{sec:proof-p-full-rank}.
	
	\begin{proof}[Proof of \Cref{p:full-rank-control-matrix-cma}(d)]
		Suppose that $c_c=c_\sigma=1$. Apply \Cref{l:contrabillity-condition-step1}, we have then that the matrix $\mathbf L^{p,q}_j$ defined via \eqref{eq:Lpqj} is the zero matrix. 
		Then, \new{there exist a steadily attracting state $\theta_0\in\cX$, $T>0$ and $v_{1:T}\in\overline{ \cO^T_{\theta_{0}}}$ such that}
		 we have that $\mathrm{rge}~\mathcal{D}S_{\theta_0}^{T}(v_{1:T})\supset \bR^d\times\{0\}\times\{0\}\times \mathrm T_{\rncovmat_{T}}\rho^{-1}(\{1\}) \times\{0\}$\del{, ending the proof.} \new{and thus by taking $p=q=0$ and $r=0$, we have, for every $(z,\rncovmat)\in\bR^d\times \mathrm T_{\rncovmat_{T}}\rho^{-1}(\{1\})$, 
		 $(z,p,q,\rncovmat,r)\in\mathrm{rge}~\mathcal{D}S_{\theta_0}^T(v_{1:T})$.}
	\end{proof}
	
	\begin{proof}[Proof of \Cref{p:full-rank-control-matrix-cma}(b)]
		Suppose that $c_c=1$ and $c_\sigma\neq 1$. 
		By \Cref{l:contrabillity-condition-step1},
			\new{there exist a steadily attracting state $\theta_0\in\cX$, $T>0$ and $v_{1:T}\in\overline{ \cO^T_{\theta_{0}}}$ such that}
		 the matrix $\mathbf L^{p,q}_j$ defined via \eqref{eq:Lpqj} satisfies:
		\begin{equation}
			\mathbf L^{p,q}_j = \begin{bmatrix}
				(1-c_\sigma)^3 c_{j}^p \ncovmat_{T}^{-1/2} & (1-c_\sigma) c_{j+2}^p \ncovmat_{T}^{-1/2} \\
				0 & 0
			\end{bmatrix} \quad ,
		\end{equation}
		with $\rank \ncovmat_{T}^{-1/2}=d$, and $c_{j}^p\neq0$, $c_{j+2}^p\neq 0$. Thus, $\mathrm{rge}~\mathbf L_j^{p,q}=\bR^d\times\{0\}$ and $\mathrm{rge}~\mathcal{D}S_{\theta_0}^{T}(v_{1:T})\supset \bR^d\times\bR^d\times\{0\}\times \mathrm T_{\rncovmat_{T}}\rho^{-1}(\{1\}) \times\{0\}$, \new{and thus by taking $q=0$ and $r=0$, we have, for every $(z,p,\rncovmat)\in\bR^d\times\bR^d\times \mathrm T_{\rncovmat_{T}}\rho^{-1}(\{1\})$, 
			$(z,p,q,\rncovmat,r)\in\mathrm{rge}~\mathcal{D}S_{\theta_0}^T(v_{1:T})$.}
	\end{proof}

	\subsection{Proof of \Cref{t:main-1}}\label{sec:main-result}
	
	In \Cref{sec-DCM-for-CMA,sec:steadily-attracting-state,sec:rank-condition-cma}, 
	we have proven all required conditions to apply \Cref{t:verifiable-conditions} or \Cref{c:verifiable-conditions-subchain} to
	the Markov chain $\Theta$ defined in \eqref{eq:smooth-normalized-chain-def}. 
	The conclusion is summarized in the next theorem.

	\begin{theorem}
	\label{t:main-bis}
	Suppose\del{ that} the objective function $f\colon\bR^d\to\bR$, the normalization functions $R$ and $\rho$, the stepsize change $\Gamma\colon\bR^d\to\bR_{++}$ and the sampling distribution $\pUd$ satisfy \Cref{F1:negligible-level-sets}-\Cref{F2:scaling-invariant}, \Cref{R1:homogeneous}-\Cref{R3:differentiable}, \Cref{rho1:homogeneous}-\Cref{rho2:C1}, \Cref{G1:C1}-\Cref{G3:G(0)<1} and \Cref{P1:density}, respectively.
	
	Let $\Theta=\{(\Z_t,p_t,q_t,{\rncovmat}_t,r_t)\}_{t\geqslant1}$ be the normalized Markov chain associated to CMA-ES defined via \eqref{eq:smooth-normalized-chain-def} {and $P$ its transition kernel}.
	Then,
	\begin{itemize}
		\item[(i)] if $c_c,c_\sigma\in (0,1)$ are such that $1-c_c\neq(1-c_\sigma)\sqrt{1-c_1-c_\mu}$, and if $c_\mu>0$, then $P$ is an irreducible aperiodic $T$-kernel, such that compact sets of $\cX=\bR^d\times\bR^d\times\bR^d\times \rho^{-1}(\{1\})\times\bR_{++}$ are small;
		\item[(ii)] if $c_c\in(0,1)$, $c_\sigma=1$ and $c_\mu>0$, then the normalized chain $\{(\Z_t,q_t,{\rncovmat}_t,r_t)\}_{t\geqslant1}$ is a time-homogeneous Markov chain with an irreducible aperiodic $T$-kernel, such that compact sets of $\cX_2=\bR^d\times\bR^d\times \rho^{-1}(\{1\})\times\bR_{++}$ are small;
		\item[(iii)] if $c_\sigma\in(0,1)$ and $c_c=1$, then the normalized chain $\{(\Z_t,p_t,\rncovmat_t)\}_{t\geqslant1}$ is a time-homogeneous Markov chain with an irreducible aperiodic $T$-kernel, such that compact sets of $\cX_3=\bR^d\times\bR^d\times \rho^{-1}(\{1\})$ are small;
		\item[(iv)] if $c_c=c_\sigma=1$, then the normalized chain $\{(\Z_t,\rncovmat_t)\}_{t\geqslant1}$ is a time-homogeneous Markov chain with an irreducible aperiodic $T$-kernel, such that compact sets of $\cX_4=\bR^d\times \rho^{-1}(\{1\})$ are small.
	\end{itemize}
\end{theorem}

	\begin{proof}
		By \Cref{p:CMA-follows-CM}, 
			the Markov chain $\Theta$ follows the control model \eqref{eq:control-modelTheta}, 
			and by \Cref{p:CMA-satisfies-assumptions}, 
			\Cref{A4:lsc-distribution} and \Cref{A5:C1-update} hold.
		
		Suppose first that $c_c,c_\sigma\neq 1$ and $1-c_c\neq(1-c_\sigma)\sqrt{1-c_1-c_\mu}$. 
		By \Cref{p:full-rank-control-matrix-cma}, 
			there exist a steadily attracting state 
			$\theta^*\in\cX$, $T\geqslant1$ $v_{1:T}\in\overline{\cO^{T}_{\theta^*}}$ such 
			that $\mathcal{D} S_{\theta^*}^T(v_{1:T})$ exists and 
			is of maximal rank. 
		Hence \Cref{H3:controllability-condition} holds, 
			and we deduce then (i) by applying \Cref{t:verifiable-conditions}.
		
		Now suppose that $c_c\neq 1$ and $c_\sigma=1$, 
			resp.\ $c_\sigma\neq 1$ and $c_c=1$, and $c_c=c_\sigma=1$. 
		Then, by \Cref{cor:normalizedMCwithoutcumul},
			$\Theta^q\coloneqq\{(\Z_t,q_t, \rncovmat_t,r_t)\}_{t\geqslant1}$, 
			resp.\ $\Theta^p\coloneqq\{(\Z_t,p_t, \rncovmat_t)\}_{t\geqslant1}$, 
			and $\Theta^r\coloneqq\{(\Z_t, \rncovmat_t)\}_{t\geqslant1}$, 
			defines a time-homogeneous Markov chain. 
		Moreover, since $\theta^*$ is a steadily attracting state for $\Theta$, 
			then, by \Cref{p:full-rank-control-matrix-cma}, 
			resp.\ $\Theta^p$, $\Theta^q$ and $\Theta^r$,
			follows a control model which satisfies \Cref{H4:controllability-subchain}.
		Thus, by \Cref{c:verifiable-conditions-subchain}, we obtain (ii), (iii) and (iv).
	\end{proof}

		\label{rem:homeomorphic}
		Our main result \Cref{t:main-1}, stated in \Cref{sec:main}, is a consequence of \Cref{t:main-bis} and of \Cref{p:irreducibility-via-homeomorphism}. 
		Indeed, \del{let}\new{consider} $\rho=\det(\cdot)^{1/d}$\new{. It is \del{be}}a normalization function that satisfies \Cref{rho1:homogeneous}-\Cref{rho2:C1}. 
		\del{Then, b}\new{B}y \Cref{p:xi-is-homeo}, the \new{associated} Markov chain $\Theta$ following \eqref{eq:smooth-normalized-chain-def} \new{with this normalization function}\del{ $\rho=\det(\cdot)^{1/d}$}\done{}\niko{do we really need to repeat this which was just said two lines above?} is 
		a transformation of the chain $\Phi$ defined via \eqref{eq:normalized-chain-def} 
		by the homeomorphism $\xi$ defined in \eqref{eq:def-homeo}.
		
		By \Cref{t:main-bis}, $\Theta$ is an irreducible aperiodic T-chain. 
		By \Cref{p:irreducibility-via-homeomorphism}, so is $\Phi$.
		Therefore~\cite[Theorem~6.2.5]{meyn2012markov}, compact sets are small sets.

	\section{Conclusion and perspectives}

	This paper \nnew{expands\del{ and applies} a methodology}\del{provides tools} to analyze \nnew{irreducibility and other} \new{stability} properties\del{, such as irreducibility,} of complex Markov chains \nnew{when they are} expressed as nonsmooth state-space models.
	\nnew{We apply the methodology} in the context of optimization\del{,
	we} \nnew{to the CMA-ES \cite{hansen2014principled}. We} prove\del{
		it\niko{what does "it" refer too?} allows to conclude to the} 
	irreducibility, aperiodicity and topological properties of a stochastic process obtained by normalizing the Markov chain \nnew{that represents}\del{representing}\del{composed of} the state of CMA-ES when optimizing scaling-invariant functions.\del{ 
		which, a}\del{\new{A}s explained \new{in the paper},}
	\new{This is an important milestone}\del{can be determinant} to prove the linear convergence of CMA-ES.

	Our \new{stability analysis}\del{assumptions} encompasses more general processes than the one underlying CMA-ES \nnew{by considering}\del{.
	In particular,}\del{
		 similarly to what was done in previous work \cite{toure2023global},}\del{
	we consider} an abstract stepsize change function.\del{ $\Gamma$. 
	Besides}
	\nnew{Compared to previous work \cite{toure2023global},} we relax the assumption\niko{{\tiny condition/precondition/requirement/premise}}\del{ that}
	\nnew{on the stepsize change}\del{its needs to be} \nnew{from} $\cC^1$ to locally Lipschitz.
	This \nnew{now}\del{also} allows to \new{analyze}\del{include} the default stepsize change of CMA-ES.\del{ In addition,}
	We \new{also consider}\del{ assume that the sampling is via} an abstract sampling distribution $\nu_U$ \nnew{which} includes\del{ the} multivariate normal distributions \nnew{as} used in CMA-ES.\del{, including general stepsize change and sampling distribution.} 

	\nnew{We summarize the assumptions to prove stability of CMA-ES:}\del{ we need to pose a number of}\del{different}\del{ assumptions \new{which}\del{that} we summarize below:}\del{ and comment on the reason for posing them:}
	\del{For our analysis, we make different assumptions.}
	\begin{itemize}
		\item \new{The objective function is\del{ assumed to be} scaling-invariant.}
		This is \new{inherent to our\del{ whole} methodology because we}
		define a time-homogeneous Markov chain \nnew{based upon}\del{from} the normalization of the state variables of CMA-ES.\del{, we require that the objective function $f$ is scaling-invariant;}\del{ and thus inherent to our whole methodology.}
		\item\del{ Moreover, we \del{only }consider} \new{The} objective function\del{s with} \new{has} Lebesgue negligible level sets. \new{This is needed to}\del{ and we} obtain \del{consequently }a lower semicontinuous density for the distribution of the ranked \del{offspring}\new{candidate solutions}.
		This is a main assumption to deduce \del{the }irreducibility \del{of the fore-mentioned Markov chain }from the analysis of an underlying control model.
		\item The normalization \new{function $R(\cdot)$ is}\del{ of the covariance matrix of the algorithm that defines the normalized process is done via a} positively homogeneous and continuously Lipschitz,\del{ function} \new{but} $R(\cdot)$ may be nonsmooth.
		\new{This includes natural normalizations, e.g., by} the determinant (which is smooth) or an eigenvalue (which is nonsmooth).
		\new{Positive homogeneity is needed for building the normalized Markov chain and thus (\new{too}\del{again}) inherent to the Markov chain methodology.
		Lipschitz \new{continuity} \nnew{yields}\del{assumption}\del{ is posed in order}\del{allows to obtain} a locally Lipschitz function $F$ for the nonsmooth model \eqref{eq:deterministic-control-model} \nnew{and allows} \new{to connect}\del{and be able\niko{\tiny required} to use the tools connecting} irreducibility to the analysis of an underlying control model \cite{gissler2024irreducibility}.}
		\item The hyperparameter setting assumptions cover all practically relevant algorithm variants (with/without cumulation, with rank-one and rank-mu updates) except when  $c_1 + c_\mu = 1$, or\del{ $c_\mu=0$ and $c_c<1$, or $1-c_c = (1-c_\sigma)\sqrt{1-c_1-c_\mu} > 0$.}
		\new{when $c_c<1$ and either $c_\mu=0$ or $1-c_c = (1-c_\sigma)\sqrt{1-c_1-c_\mu}$.}
		\new{Without cumulation ($c_c = 1$), the rank-one update is sufficient to prove irreducibility and aperiodicity. However,}
		we need the rank-mu update for our proof when cumulation is used ($c_c < 1$).
		None of the above cases is \new{important}\del{essential} in practice.\del{ and}\del{
			Importantly,}\del{ while practically the rank-one update with cumulation \new{has also been}\del{can be} used without rank-mu update}\del{
		without cumulation\del{ however} ($c_c = 1$), the rank-one update is sufficient to prove irreducibility and aperiodicity.}
	\end{itemize}
	%
	
	\paragraph*{Limitations and perspectives}
	
	We \nnew{believe that}\niko{{\tiny suppose/suspect/conjecture/reckon}}\del{ however that \del{our}}
	some of the above assumptions can be \new{relaxed}\del{generalized} with further work, \new{specifically, and based on empirical observations, the assumptions that}
	\begin{itemize}
	\item the hyperparameters have to be chosen \new{suitably}\del{correctly} (in particular $0<c_\mu<1$),
	\item the objective function $f$ has Lebesgue negligible level sets, and
	\item the sampling distribution is positive and continuous on the entire search space \new{(which is not the case for a\del{ uniform} distribution on the unit sphere)}.
	\end{itemize}

\new{In order to conclude\del{ to}---with the approach pursued in this paper---the linear convergence of CMA-ES and its learning of the inverse Hessian,
\new{it still remains to be proven}\del{the other important milestone that is left is to prove} that the}
\del{	Yet, it remains to prove that this }normalized Markov chain converges geometrically fast to a stationary distribution and satisfies a Law of Large Numbers.
\nnew{This proof could be achieved by finding}\del{A promising approach would be to find} a potential function for which a geometric drift condition holds \cite{meyn2012markov}.
	
	\appendix

	\section{Proofs in \Cref{sec:steadily-attracting-state}}\label{app:proofs-steadily}
	\subsection{Proof of \Cref{l:path-step-3}}
	\begin{proof}[Proof of \Cref{l:path-step-3}]
		Let $e_k$ be the $k$-th vector of the basis $\cB$. 
		Let $\kappa$ be positive and $\kappa'$ be real.
		Consider the sequence $\{\theta_t\}_{t=0,1,2,3,4}$ defined by
		$$
		\theta_{t+1}=(\Z_{t+1},p_{t+1},q_{t+1}, \rncovmat_{t+1},r_{t+1}) = \F\left( \theta_t , \aalpha\left( \theta_t , v_{t+1} \right) \right)
		$$
		with
		$
		v_1 = [\kappa e_k]_{i=1,\dots,\mu} ,
		$
		$
		v_2 = - r_1^{-1/2} \Gamma(p_1)^{-1} v_1,
		$
		$
		v_3 = [\kappa' e_k]_{i=1,\dots,\mu} ,
		$
		and
		$
		v_4 = -r_3^{-1/2} \Gamma(p_3)^{-1} v_3 .
		$
		\del{By \Cref{p:characterization-control-sets-cma}, note that $v_k\in\overline{ \cO^1_{\theta_{k-1}}}$ for $k=1,2,3,4$, hence $v_{1:4}\in\overline{ \cO^4_{\theta_{0}}}$.
			Then,}
		\nnew{By \Cref{p:characterization-control-sets-cma},}
		\del{By \Cref{l:step-back-to-mean-0},} we have $v_{1:4}\in\overline{ \cO^4_{\theta_{0}}}$ and \nnew{by \Cref{l:step-back-to-mean-0} we obtain}
		$
		\Z_4= \Z_2 = 0.
		$
		Moreover,
		$$
		q_2= \kappa\times \sqrt{c_c(2-c_c)\mueff} r_1^{-1/2}\left[1-c_c-\Gamma(p_1)^{-1}\right] e_k .
		$$
		Let $\eta\in\bR$, and set $\kappa'=\eta\times\left(\kappa\sqrt{c_c(2-c_c)\mueff}r_1^{-1/2}\left[1-c_c-\Gamma(p_1)^{-1}\right]\right)\nnew{=\eta q_2}$.
		Then, similarly to the proof of \Cref{l:path-step-2}, we have, since $v_3=\eta [q_2,\dots,q_2]$:
		\begin{equation}\label{eq:q4}
			q_4 = r_3^{-1/2}\times\left(  (1-c_c)^2r_2^{-1/2} + (1-c_c-\Gamma(p_3)^{-1})\sqrt{c_c(2-c_c)\mueff}\eta   \right) \times q_2
		\end{equation}
		where
		\begin{multline*}
			p_3 = (1-c_c) p_2 +  \sqrt{c_\sigma(2-c_\sigma)\mueff R(\rncovmat_2) }  \rncovmat_2^{-1/2} \vwm^\top v_3 \\
			= (1-c_\sigma)p_2 + \sqrt{c_\sigma(2-c_\sigma)\mueff R(\rncovmat_2) } \eta \rncovmat_2^{-1/2} q_2
		\end{multline*}
		and thus
		$$
		\Gamma(p_3)^{-1} = \Gamma\left( (1-c_\sigma)p_2 + \sqrt{c_\sigma(2-c_\sigma)\mueff R(\rncovmat_2) } \eta \rncovmat_2^{-1/2} q_2 \right)^{-1} .
		$$
		\del{We find then a value of $\eta\in\bR$ (independently of $\kappa$), such that $$q_4=0.$$}%
		\del{\nnew{Similarly to}\del{We refer} to the proof of \Cref{l:path-step-2}, b}%
		We apply the intermediate value theorem to the function 
		$$
		\zeta: \eta \mapsto \left[1-c_c-\Gamma\left( (1-c_\sigma)p_2 + \sqrt{c_\sigma(2-c_\sigma)\mueff R(\rncovmat_2) } \eta \rncovmat_2^{-1/2} q_2  \right)^{-1}\right]\times\sqrt{c_c(2-c_c)\mueff}\eta \enspace,
		$$ 
		\nnew{which is such that $q_4 = r_3^{-1/2}\times\left(  (1-c_c)^2r_2^{-1/2} + \zeta(\eta)  \right) \times q_2  $}.
		Since $\Gamma$ is continuous by \Cref{R2:C1} \nnew{and such that when $\eta$ goes to $\pm \infty $, $\left[1-c_c-\Gamma\left( (1-c_\sigma)p_2 + \sqrt{c_\sigma(2-c_\sigma)\mueff R(\rncovmat_2) } \eta \rncovmat_2^{-1/2} q_2  \right)^{-1}\right]$ is strictly positive by \Cref{G2:unbounded}, we find that $\zeta$ is continuous and $\zeta(\eta)$ tends to $+\infty$ when $\eta$ to $+\infty$, and to $-\infty$ when $\eta$ to $-\infty$. Hence} we find $\eta\nnew{_\kappa}\in\bR$ (\del{independently of}\nnew{which depends continuously on} $\kappa$) such that
		\nnew{when $\eta=\eta_\kappa$, we have}
		$$q_4=0.$$
		\del{ in order to find such value for $\eta\in\bR$.}%
		For the covariance matrix $\rncovmat_1$, we have
		$$
		\rncovmat_1 = \frac{(1-c_1-c_\mu)  \ncovmat_0 + c_1 {c_c(2-c_c)\mueff} \kappa^2 e_ke_k^\top + c_\mu \kappa^2 e_ke_k^\top}{\rho((1-c_1-c_\mu)  \ncovmat_0 + c_1 {c_c(2-c_c)\mueff} \kappa^2 e_ke_k^\top + c_\mu \kappa^2 e_ke_k^\top)} \enspace,
		$$
		where $\ncovmat_0=R(\rncovmat_0)^{-1}\rncovmat_0$.
		Let $\rho_1= \rho((1-c_1-c_\mu)  \ncovmat_0 + c_1 {c_c(2-c_c)\mueff} \kappa^2 e_ke_k^\top + c_\mu \kappa^2 e_ke_k^\top)$ and $\omega_1(\kappa) = c_1 {c_c(2-c_c)\mueff} \kappa^2 + c_\mu \kappa^2$ such that 
		\begin{align*}
			\rncovmat_1 & = \rho_1^{-1} \left[ (1-c_1-c_\mu)  \ncovmat_0 + c_1 {c_c(2-c_c)\mueff} \kappa^2 e_ke_k^\top + c_\mu \kappa^2 e_ke_k^\top \right] \\
			& \eqqcolon \rho_1^{-1} \left[ (1-c_1-c_\mu)  \ncovmat_0 + \omega_1(\kappa) e_ke_k^\top \right] . 
		\end{align*}
		\del{		where $\rho_1>0$ is chosen such that $\rho(\rncovmat_1)=1$,}\new{ The map }$\omega_1$ is\del{ a} continuous\del{ map}\done{ we do not need the continuity of $\rho$ for the continuity of $\omega_1$}, with $\omega_1(0)=0$ and $\omega_1(\kappa)\to\infty$ when $\kappa\to\infty$. \del{Then,}
		\new{Similarly, setting 
			$$
			\rho_2 = 
			\rho\left(
			R( \rncovmat_1)^{-1}\rncovmat_1 
			+ \left(c_1 c_c(2-c_c)\mueff\kappa^2  \left(1-c_c-\Gamma(p_1)^{-1}\right)  +c_\mu\kappa^2 \Gamma(p_1)^{-2}\right) e_ke_k^\top 
			\right)
			$$ 
			we get} 
		\begin{align*}
			\rncovmat_2 & = \rho_2^{-1}\left[ R( \rncovmat_1)^{-1}(1-c_1-c_\mu)^2  \ncovmat_0 \right. \\
			+ & \left. \left(R( \rncovmat_1)^{-1}(1-c_1-c_\mu) \omega_1(\kappa) +c_1 c_c(2-c_c)\mueff\kappa^2  \left(1-c_c-\Gamma(p_1)^{-1}\right)  +c_\mu\kappa^2 \Gamma(p_1)^{-2}\right) e_ke_k^\top \right] \\
			& \eqqcolon \rho_2^{-1}\left[ R( \rncovmat_1)^{-1}(1-c_1-c_\mu)^2 \ncovmat_0 + \omega_2(\kappa) e_ke_k^\top \right] ,
		\end{align*}
		\done{the $\rncovmat_0$ should probably read $\ncovmat_0$}
		with\del{ $\rho_2>0$ such that $\rho(\rncovmat_2)=1$,} $\omega_2$ continuous \new{since $R$ and $\Gamma$ are continuous by \Cref{R2:C1} and \Cref{G1:C1}}, $\omega_2(0)=0$ and $\omega_2(\kappa)\to\infty$ when $\kappa\to\infty$.
		
		Likewise, for the next two steps, we find $\rho_4>0$, $\omega_4$ continuous such that $\omega_4(0)=0$, and $\omega_4(\kappa)\to\infty$ when $\kappa\to\infty$, and
		$$
		\rncovmat_4 = \rho_4^{-1} \left[ R( \rncovmat_1)^{-1}R( \rncovmat_2)^{-1}R( \rncovmat_3)^{-1}(1-c_1-c_\mu)^4 \ncovmat_0 + \omega_4(\kappa) e_k e_k^\top \right].
		$$
		Then, by the intermediate value theorem, there exists $\kappa>0$ such that 
		$$\omega_4(\kappa)=R( \rncovmat_0)^{-1}R( \rncovmat_1)^{-1}R( \rncovmat_2)^{-1}R( \rncovmat_3)^{-1}(1-c_1-c_\mu)^4 (\lambda_{k-1}-\lambda_k)>0 \enspace. $$
		Therefore,
		$$
		[\rncovmat_4]_{\mathcal{B}} =  \rho_4^{-1}R( \rncovmat_0)^{-1}R( \rncovmat_1)^{-1}R( \rncovmat_2)^{-1}R( \rncovmat_3)^{-1} (1-c_1-c_\mu)^4 \diag(\lambda_1,\dots,\lambda_{k-1},\lambda_{k-1},\lambda_{k+1},\dots,\lambda_d) \enspace.
		$$
		\del{		This ends the proof.}\new{Setting $\gamma =  \rho_4^{-1} R( \rncovmat_0)^{-1}R( \rncovmat_1)^{-1}R( \rncovmat_2)^{-1}R( \rncovmat_3)^{-1}(1-c_1-c_\mu)^4$, we have proven that we can reach $\theta_4$ with the matrix $\rncovmat_4$ defined in \eqref{eq:equal-eig}.}
	\end{proof}
	
	\section{Proofs in \Cref{sec:rank-condition-cma}}\label{app:proofs-controllability}
	\subsection{Proof of \Cref{l:contrabillity-condition-step0}}
	\begin{proof}[Proof of \Cref{l:contrabillity-condition-step0}]
		By \Cref{R3:differentiable}, 
		there exists $\mathbf C_0\in\Sdpp$, 
		such that $R$ is differentiable on a neighbor\-hood of $\mathbf C_0$. 
		Since $R$ is positively homogeneous by \Cref{R1:homogeneous}, 
		then \new{by \Cref{l:diff-pos-hom}} $R$ is also differentiable on a neighborhood of $\rncovmat_0\coloneqq\rho(\mathbf C_0)^{-1}\mathbf C_0$\done{We have a lemma to prove a property which is less strong than that and here we state the property without any proof. Formulate maybe in Lemma~\ref{l:diff-pos-hom} that if a positively homogeneous function is differentiable in a neighborhood of $C_0$, then it is differentiable in a neighborhood of $\gamma C_0$ for any $\gamma$.}. 
		Then, by \Cref{c:steadily-attracting-state-for-CMA}, there exists $\theta_0=(z_0,p_0,q_0,\rncovmat_0,r_0)$ with $z_0=q_0=0$ which is a steadily attracting state.
		
		Let $T\in\bN$, $v_{1:T}\in\overline{\mathcal{O}_{\theta_0}^T}$ and ${h}_{1:T}\in\cV^T$. 
		We denote for $t\in\{1,\dots,T\}$:
		$$
		\theta_t=(\Z_t,p_t,q_t, \rncovmat_t,r_t) = S^t_{\theta_0}(v_{1:t}) \quad \text{and} \quad \theta_t^h=(\Z_t^h,p_t^h,q_t^h, \rncovmat_t^h,r_t^h) = S^t_{\theta_0}(v_{1:t}+{h}_{1:t}) \enspace.
		$$
		We have that, if $v_{1:T}=0$, then \new{since $q_0=0$, $q_t = 0$ for $t=1,\ldots, T$ and using \eqref{eq:update-Sigma} we find } $\rncovmat_t=\rncovmat_0$.
		Since $\nnew{v \mapsto} S^t_{\theta_0}\nnew{(v)}$ is continuous, and since $R$ is differentiable in a neighborhood of $\rncovmat_0$, 
		then there exists $M_V >0$ such that, if $\|v_{1:T}\|^2\leqslant M_V$, then $R$ is differentiable at $\rncovmat_t$.
		Hence we impose that $\|v_t\|^2 \leqslant M_V/T$ \nnew{for all $t  \in\{1,\dots,T\}$}.
		
		Define, for $t=0,\dots,T$,
		$b_t = r_0\times\dots\times r_t \times R(\rncovmat_t)^{-1}$ and
		$b_t^h = r_0^h\times\dots r_t^h \times R({\rncovmat}_t^h)^{-1}$,
		and let
		$\mathbf B_t=b_t\rncovmat_t$ and likewise $\mathbf B_t^h = b_t^h\rncovmat_t^h$.
		\del{with $b_t = r_1\times\dots\times r_t \times R(\rncovmat_t)^{-1}$\todo{this is a bit confusing - can you detail how you go from $b_t = \rho(B_t)$ to this equation?} and $b_t^h = r_1^h\times\dots r_t^h \times R({\rncovmat}_t^h)^{-1}$,
			so that}%
		Therefore \nnew{by positive homogeneity of $\rho$, $\rho(\mathbf B_t) = \rho(b_t\rncovmat_t) = b_t \rho(\rncovmat_t) = b_t$ since $\rho(\rncovmat_t) = 1$. Similarly $\rho(\mathbf B_t^h) = b_t^h$ and thus}
		\begin{equation}\label{eq:rcovB}
			\rncovmat_t = \frac{\mathbf B_t}{ \rho(\mathbf B_t)} \quad \text{and} \quad \rncovmat_t^h =  \frac{\mathbf B_t^h}{\rho(\mathbf B_t^h)} \enspace.
		\end{equation}	
		Moreover, define 
			$\tilde q_t = \sqrt{\tilde r_{t-1}} q_t \quad\text{and}\quad \tilde q_t^h=\sqrt{\tilde r_{t-1}^h} q_t^h $ as well as
			$ \tilde v_{t+1} = \sqrt{\tilde r_{t}} v_{t+1} \quad\text{and}\quad \tilde h_{t+1} = \sqrt{\tilde r_{t}^h} h_{t+1}$
		where $\tilde r_t = r_0\times\dots\times r_t$ and $\tilde r_t^h=r_0^h\times \nnew{ \dots \times} r_t^h$.
		Hence, by applying \nnew{\eqref{eq:update-q}}\del{\eqref{eq:Fq}}:
		\begin{multline*}
			\tilde q_{t+1} = \sqrt{\tilde r_t} r_t^{-1/2} (1-c_c) q_t + \sqrt{c_c(2-c_c)\mueff} \sqrt{\tilde r_t} \vwm^\top v_{t+1} \\
			= (1-c_c) \sqrt{\tilde r_{t-1}} q_t + \sqrt{c_c(2-c_c)\mueff} \vwm^\top \tilde v_{t+1}  
			= (1-c_c)\tilde q_t + \sqrt{c_c(2-c_c)\mueff} \vwm^\top \tilde v_{t+1}
		\end{multline*}
		and likewise
		\begin{equation}\label{eq:inproof-qtt}
			\tilde q_{t+1}^h = (1-c_c)\tilde q_t^h + \sqrt{c_c(2-c_c)\mueff} \vwm^\top \left( (\tilde r_t^h/\tilde r_t)^{1/2}\tilde v_{t+1} + \tilde h_{t+1} \right) \enspace .
		\end{equation}
		\new{Denote $\mathbf A_{t+1}$ such that $\rncovmat_{t+1} = \mathbf A_{t+1}/\rho(\mathbf A_{t+1}) $ in \eqref{eq:update-Sigma}. \new{(Alternatively the matrix $\tilde\ncovmat_{t+1}$ \del{introduced in the proof of \Cref{l:homogeneous-MC}}\nnew{in \eqref{tilda-matrix} equals $\mathbf A_{t+1}$})}.\done{clarify whether they are the same or not} Then by positive homogeneity of $R$, $R(\rncovmat_{t+1})= R(\mathbf A_{t+1}) / \rho(\mathbf A_{t+1})$ such that
			\begin{equation}
				\frac{\mathbf A_{t+1}}{R(\mathbf A_{t+1})} = \frac{\mathbf A_{t+1}}{\rho(\mathbf A_{t+1}) R(\rncovmat_{t+1})} = \frac{\rncovmat_{t+1}}{R(\rncovmat_{t+1})}
			\end{equation}
		}
		Then, \nnew{using the previous equation and} \eqref{eq:update-Sigma}:
		\begin{align}
			\mathbf B_{t+1} & = \new{b_{t+1}} \rncovmat_{t+1} = r_0\times\dots\times r_{t+1} \times R(\rncovmat_{t+1})^{-1} \rncovmat_{t+1} = \tilde r_{t+1} R(\rncovmat_{t+1})^{-1} \rncovmat_{t+1} \\
			& = \nnew{\tilde r_{t+1} R(\mathbf A_{t+1})^{-1} \mathbf A_{t+1}} \\
			& = \underbrace{\tilde r_{t+1} \times r_{t+1}^{-1}}_{\tilde r_{t}} \times \left(  (1-c_1-c_\mu) R(\rncovmat_t)^{-1} \rncovmat_t
			+ c_1 q_{t+1}q_{t+1}^\top + c_\mu \sum_{i=1}^\mu \wic v_{t+1}^i (v_{t+1}^i)^\top  \right)  \\
			& = (1-c_1-c_\mu) \tilde r_tR(\rncovmat_t)^{-1} \rncovmat_t + c_1 \tilde r_t \del{\tilde q_{t+1}\tilde q_{t+1}^\top}\nnew{q_{t+1} q_{t+1}^\top} + c_\mu \sum_{i=1}^\mu \wic \tilde r_t v_{t+1}^i ( v_{t+1}^i)^\top \\
			& = (1-c_1-c_\mu) \mathbf B_t + c_1 \tilde q_{t+1}\tilde q_{t+1}^\top + c_\mu \sum_{i=1}^\mu \wic \tilde v_{t+1}^i (\tilde v_{t+1}^i)^\top  \enspace .
			\label{inproof-Btt}%
		\end{align}
		Likewise,
		\begin{align}\label{eq:inproof-tom}
			\mathbf B_{t+1}^h & = (1-c_1-c_\mu) \mathbf B_t^h 
			+ c_1 \tilde q_{t+1}^h(\tilde q_{t+1}^h)^\top 
			+ c_\mu \tilde r_t^h\sum_{i=1}^\mu \wic ( v_{t+1}^i+ h_{t+1}^i) ( v_{t+1}^i+ h_{t+1}^i)^\top
			\enspace .
		\end{align}
		Let $s=d(d+1)/2$ be the dimension of the set of symmetric matrices $\Sd$ as a real vector space. 
		Let $\psi_1 \in\bR^d$ be a nonzero vector and
		define then $\psi_2,\dots,\psi_s$ nonzero vectors of $\bR^d$, such that $(\psi_1\psi_1^\top,\dots,\psi_s\psi_s^\top)$ forms a basis of $\Sd$.
		\new{Scaling down the length of $\psi_k$ does not change that we have a basis of $\Sd$ and thus we}\del{We} impose that $\|\psi_k\|\leqslant \varepsilon$, where $\varepsilon$ is a positive constant that we precise in the next paragraph.
		Set $T=2s(s-1)+4$ and set $v_{1:T}$ as below.

		For $t\in\{0,\dots,s-1\}$, we set
		\begin{equation}\label{eq:v2t+1}
			v_{2t+1} = \tilde r_{2t}^{-1/2} [\psi_{t+1},\dots,\psi_{t+1}]\in\bR^{d\mu} \mbox{ and } v_{2t+2} = -(1-c_c)\tilde r_{2t+1}^{-1/2}[\psi_{t+1},\dots,\psi_{t+1}]\in\bR^{d\mu}
		\end{equation}
		\nnew{such that
			\begin{equation}\label{eq:tildev2t+1}
				\tilde v_{2t+1} = [\psi_{t+1},\dots,\psi_{t+1}] \mbox{ and }  \tilde v_{2t+2} = -(1-c_c) [\psi_{t+1},\dots,\psi_{t+1}].
			\end{equation}
		}
		Moreover, we choose $\varepsilon>0$ small enough so that 
		$\|v_k\|^2\leqslant M_V/T$ for all $k=1,\dots, 2s$.
		By definition of $M_V$ earlier in the proof, we have that $R$ is differentiable in $\rncovmat_t$ for $t=0,\dots,T$.
		If moreover $p_t\neq0$ for $t=1,\dots,T$, then by \Cref{l:S-differentiable-in}, $v \to S_{\theta_0}^T(v)$ is differentiable in $v_{1:T}$.
		Besides, by \Cref{p:characterization-control-sets-cma}, we have $v_{1:2s}\in\overline{\cO^{2s}_{\theta_0}}$.
		
		Observe now that there exists $\psi_1\in\bR^d$ such that $\|\psi_1\|\leqslant\varepsilon$ and $p_1,p_2$ are nonzero.
		Indeed, \new{$p_1  = (1-c_\sigma) p_0 +
			\sqrt{c_\sigma(2-c_\sigma)\mueff} R(\rncovmat_0)^{1/2}\rncovmat_0^{-1/2} \sum_{i=1}^\mu \wi v_{1}^i$ and using $\mathbf B_0 = r_0 R(\rncovmat_0)^{-1} \rncovmat_0$ we find $p_1 =   (1-c_\sigma) p_0 +
			\sqrt{c_\sigma(2-c_\sigma)\mueff} r_{0}^{1/2}\mathbf B_0^{-1/2} \sum_{i=1}^\mu \wi \tilde r_0^{-1/2} \psi_1 = (1-c_\sigma) p_0 +
			\sqrt{c_\sigma(2-c_\sigma)\mueff} \mathbf B_0^{-1/2}\psi_1$}
		and
		\begin{align}
			p_2  = (1-c_\sigma) p_1 & +
			\sqrt{c_\sigma(2-c_\sigma)\mueff} R(\rncovmat_1)^{1/2}\rncovmat_1^{-1/2} \sum_{i=1}^\mu \wi v_{2}^i \\
			= (1-c_\sigma)^2 p_0 & +
			(1-c_\sigma) \sqrt{c_\sigma(2-c_\sigma)\mueff} \mathbf B_0^{-1/2}\psi_1 \\ & - 
			\sqrt{c_\sigma(2-c_\sigma)\mueff} \tilde r_1^{1/2} \mathbf B_1^{-1/2} \sum \wi (1-c_c) \tilde r_1^{-1/2} \psi_1 \\
			= (1-c_\sigma)^2 p_0 & +
			[(1-c_\sigma)\mathbf B_0^{-1/2}-(1-c_c)\mathbf B_1^{-1/2}] \sqrt{c_\sigma(2-c_\sigma)\mueff} \psi_1  \enspace.
		\end{align}
		\done{Since $c_1+c_\mu>0$}
		Since $v_2 = -(1-c_c)\tilde r_1^{-1/2}[\psi_1,\dots,\psi_1]$ with $\psi_1\neq 0$, \new{given that according to \eqref{inproof-Btt}, $\mathbf B_1 = \alpha_1 \mathbf B_0 + \alpha_2 \psi_1 \psi_1^\top + \alpha_3 q_1 q_1^\top $, for $\alpha_1, \alpha_2,\alpha_3$ some nonnegative constants and $\alpha_2+\alpha_3 >0$ \nnew{since $c_1+c_\mu>0$, and $\psi_1, q_1\neq 0$ (see below)},} 
		we have $(1-c_\sigma)\mathbf B_0^{-1/2}\neq (1-c_c)\mathbf B_1^{-1/2}$.
		Moreover, up to scaling $\psi_2,\dots,\psi_s$ sufficiently smaller than $\psi_1$, 
		we can ensure that $p_t\neq 0$ for $t\in\{\del{1}\nnew{3},\dots,2s\}$. 
		Then, by \Cref{l:S-differentiable-in}, and by composition \nnew{since $S_{\theta_0}^{t+1}(v_{1:t+1}) = S_{S_{\theta_0}^{t}(v_{1:t})}^1(v_{t+1}) $ we find by induction that} $S_{\theta_0}^{2s}$ is differentiable at $v_{1:2s}$.		
		Then, by induction, \nnew{since $\tilde q_{t+1} 
			= (1-c_c)\tilde q_t + \sqrt{c_c(2-c_c)\mueff} \vwm^\top \tilde v_{t+1} $
			with
			$
			\tilde	v_{2t+1} =  [\psi_{t+1},\dots,\psi_{t+1}] 
			$
			and
			$
			\tilde v_{2t+2} = -(1-c_c)[\psi_{t+1},\dots,\psi_{t+1}]
			$,
		} we find that, for every $t\in\{0,\dots,s-1\}$, we have:
		\begin{equation}\label{innproof-1}
			\tilde q_{2t+1} =  \sqrt{c_c(2-c_c)\mueff}\psi_{t+1} \mbox{ and } q_{2t+2} = 0.
		\end{equation}
		\del{Besides, we have:
			$$
			B_{2t+1} = (1-c_1-c_\mu) B_{3t} + (c_1c_c(2-c_c)\mueff+c_\mu)\psi_{t+1}\psi_{t+1}^\top .
			$$}
		\del{$$
			B_{3t+2} = (1-c_1-c_\mu)^2 B_{3t} + \underbrace{\left[(1-c_1-c_\mu)(c_1c_c(2-c_c)\mueff+c_\mu)+c_\mu (1-c_c)^2r_{3t+1}^{-1}\todo{}\right]}_{\eqqcolon c_{t}} \psi_{t+1}\psi_{t+1}^\top ,
			$$
			and
			$$
			B_{3t+3} = (1-c_1-c_\mu)^3 B_{3t} + (1-c_1-c_\mu) c_t \psi_{t+1}\psi_{t+1}^\top .
			$$}%
		For $t=0,\dots,s-1$, let $\kappa_t^1\in\bR$ be arbitrary (we fix the value of $\kappa_t^1$ later in the proof).
		We set, given an arbitrary real number $\h_1\in\bR$, for $t=0,\dots,s-1$:
		$$
		{h}_{2t+1}= [(\tilde r_{2t}^h)^{-1/2}-\tilde r_{2t}^{-1/2} + (\tilde r_{2t}^h)^{-1/2}\kappa_t^1  \h_1 ] \times [\psi_{t+1} ,\dots,\psi_{t+1} ]\in\bR^{d\mu}
		$$
		which implies
		\begin{equation}
			\tilde h_{2t+1} = [ 1 - (\tilde r_{2t}^h/\tilde r_{2t})^{1/2}  + \kappa_t^1  \h_1]  \times [\psi_{t+1} ,\dots,\psi_{t+1} ]
		\end{equation}
		and
		\begin{align*}
			{h}_{2t+2} & =-(1-c_c) 
			[(\tilde r_{2t+1}^h)^{-1/2}-\tilde r_{2t+1}^{-1/2} + (\tilde r_{2t+1}^h)^{-1/2}\kappa_t^1  \h_1 ] \times [\psi_{t+1} ,\dots,\psi_{t+1} ]\in\bR^{d\mu}
			\enspace ,
		\end{align*}
		so that, by induction, \new{starting from \eqref{eq:inproof-qtt}} we have:
		\begin{align}
			\tilde q_{2t+1}^h & = (1-c_c) \tilde q_{2t}^h + \sqrt{c_c(2-c_c)\mueff} \vwm^\top ((\tilde r_{2t}^h/\tilde r_{2t})^{1/2}\tilde v_{2t+1} + \tilde h_{2t+1})  \\
			& = (1-c_c) \times 0 \\
			& ~ + \sqrt{c_c(2-c_c)\mueff} \vwm^\top \left(  (\tilde r_{2t}^h/\tilde r_{2t})^{1/2} + \left(1- (\tilde r_{2t}^h/\tilde r_{2t})^{1/2} + \kappa_t^1 \varepsilon_1\right) \right) [\psi_{t+1},\dots,\psi_{t+1}]  \\
			& = \sqrt{c_c(2-c_c)\mueff}\left(1 + \kappa_t^1 \varepsilon_1\right)   \psi_{t+1}
		\end{align}
		and
		\begin{align*}
			\tilde q_{2t+2}^h & = (1-c_c) \tilde q_{2t+1}^h + \sqrt{c_c(2-c_c)\mueff} \vwm^\top ((\tilde r_{2t+1}^h/\tilde r_{2t+1})^{1/2}\tilde v_{2t+2} + \tilde h_{2t+2}) \\
			& =  (1-c_c)\sqrt{c_c(2-c_c)\mueff}\left(1 + \kappa_t^1 \varepsilon_1\right)   \psi_{t+1}  \\
			&\quad -(1-c_c) \sqrt{c_c(2-c_c)\mueff}  \vwm^\top \left(  \sqrt{\frac{\tilde r_{2t+1}^h}{\tilde r_{2t+1}}} + \left(1- \sqrt{\frac{\tilde r_{2t+1}^h}{\tilde r_{2t+1}}} + \kappa_t^1 \varepsilon_1\right) \right) [\psi_{t+1},\dots,\psi_{t+1}] \\
			& = 0
		\end{align*}
		Note that, for $i=1,\dots,\mu$:\done{Armand: we had 3 times the same mistake I guess like below. I corrected the others, I let you triple check - as it is strange to see several times the same error (but maybe it's due to some cut and paste)}
		$$
		(\tilde r_{2t}^h)^{1/2} \times (v_{2t+1}^i+h_{2t+1}^i) = \left(  (\tilde r_{2t}^h/\tilde r_{2t})^{1/2} + \left(1- (\tilde r_{2t}^h/\tilde r_{2t})^{1/2}  + \kappa_t^1 \varepsilon_1\right) \right) \psi_{t+1} = \left(  1 + \kappa_t^1 \varepsilon_1 \right) \psi_{t+1}
		\enspace .
		$$
		Then, \nnew{using \eqref{eq:inproof-tom},} we obtain for $t\in\{0,\dots,s-1\}$, when $\h_1\to0$:
		\begin{align}
			\mathbf B_{2t+1}^h  =  (1-c_1-c_\mu) \mathbf B_{2t}^h &+ c_1 c_c(2-c_c)\mueff  (1+ \kappa_t^1 \h_1)^2 \psi_{t+1}  \psi_{t+1}^\top 
			\\ &+ c_\mu  \sum_{i=1}^\mu \wic  \left(  1 + \kappa_t^1 \varepsilon_1 \right)^2 \psi_{t+1}\psi_{t+1}^\top \\
			=(1-c_1-c_\mu) \mathbf B_{2t}^h  &
			+ \left[ c_1 c_c(2-c_c)\mueff +c_\mu\right]\times\left( 1+ 2\kappa_t^1\h_1  \right) \psi_{t+1}\psi_{t+1}^\top + o(\h_1) \label{inproof-BB}
		\end{align}
		From \eqref{inproof-Btt}, we have that 
		\begin{align*}
			\mathbf B_{2t+1}  = (1-c_1-c_\mu) \mathbf B_{2t}  & + c_1 \tilde q_{2t+1}\tilde q_{2t+1}^\top + c_\mu \sum_{i=1}^\mu \wic \tilde v_{2t+1}^i (\tilde v_{2t+1}^i)^\top \\
			= (1-c_1-c_\mu) \mathbf B_{2t}  & + c_1 c_c(2-c_c)\mueff \psi_{t+1}\psi_{t+1}^\top + c_\mu \sum_{i=1}^\mu \wic \psi_{t+1}\psi_{t+1}^\top \\
			= (1-c_1-c_\mu) \mathbf B_{2t}  & + \left[c_1 c_c(2-c_c)\mueff+c_\mu\right] \psi_{t+1}\psi_{t+1}^\top 
		\end{align*}
		that we use in \del{the previous equation}\nnew{\eqref{inproof-BB}} to obtain
		\begin{align} \nonumber
			\mathbf	B_{2t+1}^h	& = \mathbf B_{2t+1} 
			+  (1-c_1-c_\mu)\left( \mathbf B_{2t}^h -\mathbf B_{2t} \right) 
			+ \underbrace{\left[ c_1 c_c(2-c_c)\mueff +c_\mu\right] \times 2}_{:= c_b} \kappa_t^1 \h_1 \psi_{t+1}\psi_{t+1}^\top + o(\h_1) \\\label{inproof-conf}
			& =  \mathbf B_{2t+1} +  (1-c_1-c_\mu)\left(\mathbf  B_{2t}^h -\mathbf B_{2t} \right) 
			+ c_b \kappa_t^1 \h_1 \psi_{t+1}\psi_{t+1}^\top + o(\h_1) \enspace .
		\end{align}
		Moreover, for $i=1,\dots,\mu$:
		\begin{align*}
			(\tilde r_{2t+1}^h)^{1/2} \times (v_{2t+2}^i+h_{2t+2}^i) & = -(1-c_c) \left(  (\tilde r_{2t+1}^h/\tilde r_{2t+1})^{1/2} + \left(1- (\tilde r_{2t+1}^h/\tilde r_{2t+1})^{1/2} + \kappa_t^1 \varepsilon_1\right) \right) \psi_{t+1} 
			\\ & =  -(1-c_c) \left(  1 + \kappa_t^1 \varepsilon_1 \right) \psi_{t+1}
			\enspace .
		\end{align*}
		Thus, we obtain, by \eqref{eq:inproof-tom} and \eqref{inproof-conf}:
		\begin{align*}
			\mathbf B_{2t+2}^h & = (1-c_1-c_\mu) \mathbf B_{2t+1}^h + c_\mu \sum_{i=1}^\mu \wic (1-c_c)^2(  1 + \kappa_t^1 \varepsilon_1)^2 \psi_{t+1}\psi_{t+1}^\top \\
			& = (1-c_1-c_\mu)\mathbf B_{2t+1} +  (1-c_1-c_\mu)^2\left(\mathbf  B_{2t}^h -\mathbf B_{2t} \right) 
			+ (1-c_1-c_\mu) c_b \kappa_t^1 \h_1 \psi_{t+1}\psi_{t+1}^\top  \\
			& \quad + c_\mu (1-c_c)^2(  1 + \kappa_t^1 \varepsilon_1)^2 \psi_{t+1}\psi_{t+1}^\top + o (\varepsilon_1) \\
			& = (1-c_1-c_\mu)\mathbf B_{2t+1} + c_\mu(1-c_c)^2\psi_{t+1}\psi_{t+1}^\top + (1-c_1-c_\mu)^2\left(\mathbf  B_{2t}^h -\mathbf B_{2t} \right)
			\\ & \quad	+ d_b \kappa_t^1 \varepsilon_1 \psi_{t+1}\psi_{t+1}^\top + o(\varepsilon_1)
		\end{align*}
		with
		$$
		d_b = (1-c_1-c_\mu)c_b + 2c_\mu(1-c_c)^2 = 2(1-c_1-c_\mu)\left(c_1 c_c(2-c_c)\mueff +c_\mu\right)  + 2 c_\mu(1-c_c)^2  \enspace.
		$$
		Yet, by \eqref{inproof-Btt} since \nnew{by \eqref{innproof-1} $q_{2t+2} = 0$ and by \eqref{eq:tildev2t+1} $\tilde v_{2t+2}^i = - (1-c_c) \psi_{t+1}$}:
		$$
		\mathbf B_{2t+2} = (1-c_1-c_\mu)\mathbf B_{2t+1} + c_\mu(1-c_c)^2\psi_{t+1}\psi_{t+1}^\top \enspace.
		$$
		Therefore,
		$$
		\mathbf B_{2t+2}^h - \mathbf B_{2t+2} = (1-c_1-c_\mu)^2\left(\mathbf  B_{2t}^h -\mathbf B_{2t} \right)
		+ d_b \kappa_t^1 \varepsilon_1 \psi_{t+1}\psi_{t+1}^\top + o(\varepsilon_1) \enspace.
		$$	
		\del{Thus, there exists a continuous function $\beta_t^1 \colon \bR\to\bR$ such that $\beta_t^1 (0)=0$, $\beta_t^1 (\kappa)\to +\infty$ when $\kappa\to +\infty$, $\beta_t^1 (\kappa)\to -\infty$ when $\kappa\to -\infty$, and, for every $\kappa_t^1\in\bR$, when $\varepsilon_1\to0$, we have:
			\begin{align*}
				\mathbf B_{2t+2}^h & = \mathbf B_{2t+2} 
				+ (1-c_1-c_\mu)^2 \left(\mathbf B_{2t}^h-\mathbf B_{2t}\right) + \beta_t^1(\kappa_t^1)\varepsilon_1 \psi_{t+1}\psi_{t+1}^\top + o(\varepsilon_1) \enspace.
		\end{align*}}%
		Then, by induction, we get,
		$$
		\mathbf B_{2s}^h =\mathbf  B_{2s} 
		+ \h_1 \sum_{t=1}^s (1-c_1-c_\mu)^{2s-2t} d_b \kappa_{t-1}^1 \psi_t\psi_t^\top + (1-c_1-c_\mu)^{2s}(\mathbf{B}_0^h-\mathbf{B}_0) +o(\h_1) \enspace,
		$$		
		with $\mathbf{B}_0^h-\mathbf{B}_0=0$ by definition.
		By induction on $k\in\{1,\dots,s-2\}$, we set for $t\in\{0,\dots,s-1\}$:
		$$
		v_{2t+2ks+1} = \tilde r_{2t+2ks}^{-1/2} [\psi_{t+1},\dots,\psi_{t+1}]\in\bR^{d\mu} \enspace,
		$$
		and
		$$
		v_{2t+2ks+2} = -(1-c_c)\tilde r_{2t+2ks+1}^{-1/2}[\psi_{t+1},\dots,\psi_{t+1}]\in\bR^{d\mu} \enspace.
		$$
		We also set, given arbitrary real numbers $\h_k\in\bR$ and for some $\kappa_t^k\in\bR$ for $t\in\{0,\dots,s-1\}$:
		$$
		{h}_{2t+2ks+1}= [(\tilde r_{2t+2ks}^h)^{-1/2}-\tilde r_{2t+2ks}^{-1/2} + (\tilde r_{2t+2ks}^h)^{-1/2}\kappa_t^s  \h_1 ] \times [\psi_{t+1} ,\dots,\psi_{t+1} ]\in\bR^{d\mu}
		$$
		and
		\begin{align*}
			{h}_{2t+2ks+2} & =-(1-c_c) 
			[(\tilde r_{2t+2ks+1}^h)^{-1/2}-\tilde r_{2t+2ks+1}^{-1/2} + (\tilde r_{2t+2ks+1}^h)^{-1/2}\kappa_t^s  \h_1 ] \times [\psi_{t+1} ,\dots,\psi_{t+1} ]\in\bR^{d\mu}
			\enspace .
		\end{align*}
		Then, similarly to above, we obtain $q_{2(k+1)s}=0$ and
		$$
		\mathbf B_{2(k+1)s}^h =\mathbf  B_{2(k+1)s} 
		+ \h_k \sum_{t=1}^s (1-c_1-c_\mu)^{2s-2t} d_b \kappa_{t-1}^k \psi_t\psi_t^\top + (1-c_1-c_\mu)^{2s}(\mathbf{B}_{2ks}^h-\mathbf{B}_{2ks}) +o(\h_k) \enspace.
		$$
		Thus, by induction, we get
		\del{
			
			If we repeat this operation $s-1$ times  with arbitrary real numbers $\h_2,\dots,\h_{s-1}$ which tend to $0$ instead of $\h_1$, and arbitrary real numbers $\kappa_k^2,\dots,\kappa_k^{s-1}$ instead of $\kappa_k^1$, we get then }%
		$q_{2s(s-1)}^h=0$ and
		\begin{multline*}
			\mathbf B_{2s(s-1)}^h=\mathbf B_{2s(s-1)} 
			+ \sum_{k=1}^{s-1} \h_k (1-c_1-c_\mu)^{2s(s-1)-2ks} \sum_{t=1}^s  (1-c_1-c_\mu)^{2s-2t} d_b \kappa_{t-1}^k \psi_t\psi_t^\top  +o(\h_{1:s-1}) 
		\end{multline*}
		\del{where $\beta_k^t \colon \bR\to\bR$ is a continuous function such that $\beta_k^t(0)=0$, $\beta_k^t(\kappa)\to +\infty$ when $\kappa\to +\infty$, 
			and $\beta_k^t(\kappa)\to -\infty$ when $\kappa\to -\infty$. }%
		Note moreover that we can assume again that $p_t\neq 0$, 
		up to choosing again the $\psi_k$, $k\geqslant2$, sufficiently smaller than $\psi_1$.
		
		By \Cref{l:path-step-1,l:path-step-2}, 
		for any $v_{2s(s-1)+1}\in\overline{\cO^1_{\theta_{2s(s-1)}}}$, 
		there exists $v_{2s(s-1)+2:2s(s-1)+4}\in\overline{ \cO^3_{\theta_{2s(s-1)+1}}}$ such that $z_{2s(s-1)+4}=q_{2s(s-1)+4}=0$. 
		Moreover, when $v_{2s(s-1)+1}\to0$,
		then we have that $z_{2s(s-1)+1}$ and $q_{2s(s-1)+1}$ tend to $0$ and thus
		we can impose that $v_{2s(s-1)+2:2s(s-1)+4}\to 0$ as well.
		In particular, we can choose $v_{2s(s-1)+1}$ small enough such that $p_t\neq0$ for $t=2s(s-1)+1,\dots,2s(s-1)+4$.
		Hence, by \Cref{l:S-differentiable-in}, $S^{2s(s-1)+4}_{\theta_0}$ is differentiable at $v_{2s(s-1)+4}$ 
		(we have that $R$ is differentiable at $\rncovmat_t$ for all $t=1,\dots,T$ by imposing $v_{1:T}$ small enough, see the beginning of the proof).
		
		Consider then $(\mathbf S_1,\dots,\mathbf S_{s-1})$ a basis of $\ker\mathcal{D}\rho(\mathbf B_{2s(s-1)+4})$. 
		For $k=1,\dots,s-1$, we can choose then the $\kappa_t^k\in\bR$, $t=0,\dots,s-1$ so that we have
		\begin{equation}\label{eq:inproof-B2s(s-1)}
			(1-c_1-c_\mu)^{2s(s-1)-2ks} \sum_{t=1}^s  (1-c_1-c_\mu)^{2s-2t} d_b \kappa_{t-1}^k \psi_t\psi_t^\top = \mathbf S_k \enspace.
		\end{equation}
		This is possible since $(\psi_1\psi_1^\top,\dots,\psi_s\psi_s^\top)$ is a basis of $\Sd$,
		and by the intermediate value theorem
		applied to the LHS of \eqref{eq:inproof-B2s(s-1)},
		for $k=1,\dots,s$.
		Set $T=2s(s-1)+4$. Then, we have, when $\varepsilon_{1:s-1}\to0$,
		\begin{equation}\label{eq:BBB}
			\mathbf B_{T}^h=\mathbf B_{T} + \sum_{k=1}^{s-1} \h_k \mathbf S_k + o(\h_{1:s-1}) \enspace.
		\end{equation}
		Therefore, since $\mathbf S_k\in\ker\mathcal{D}\rho(\mathbf B_{T})$ for $k=1,\dots,s$, \nnew{we have $$\rho \left(\mathbf B_{T} + \sum_{k=1}^{s-1} \h_k \mathbf S_k + o(\h_{1:s-1})\right) = \rho(\mathbf B_{T}) + \mathcal \sum_{k=1}^{s-1} \h_k \underbrace{\mathcal{D}\rho(\mathbf B_{T}) \mathbf S_k}_{=0} + o(\h_{1:s-1}) =  \rho(\mathbf B_{T}) + o(\h_{1:s-1})$$ and using \eqref{eq:rcovB} and \eqref{eq:BBB}}
		\begin{align*}
			\rncovmat_T^h & = \frac{\mathbf B_{T} + \sum_{k=1}^{s-1} \h_k \mathbf S_k + o(\h_{1:s-1})}{\rho\left(\mathbf B_{T} + \sum_{k=1}^{s-1} \h_k \mathbf S_k + o(\h_{1:s-1})\right)} =  \frac{\mathbf B_{T} + \sum_{k=1}^{s-1} \h_k \mathbf S_k }{\rho\left(\mathbf B_{T}\right)}+ o(\h_{1:s-1})  \\
			& = \rncovmat_T + \sum_{k=1}^{s-1} \h_k \rho\left(\mathbf B_{T}\right)^{-1} \mathbf S_k + o(\h_{1:s-1}) .
		\end{align*}
		However, $(\mathbf{S}_1,\dots,\mathbf{S}_{s-1})$ is a basis of $\ker\mathcal{D}\rho(\mathbf{B}_{T})$, and 
		by \Cref{l:diff-pos-hom}, $\ker\mathcal{D}\rho(\mathbf{B}_{T})=\ker\mathcal{D}\rho(\rncovmat_{T})=\mathrm T_{\rncovmat_T}\rho^{-1}(\{1\})$. \nnew{Thus we have shown that for every $h_{\ncovmat} \in T_{\rncovmat_T}\rho^{-1}(\{1\})$ for which we can find $\h_k$ such that $h_{\ncovmat}=\sum_{k=1}^{s-1} \h_k \rho\left(\mathbf B_{T}\right)^{-1} \mathbf S_k$,}\del{Thus we have shown that for every $h_{\ncovmat}=\sum_{k=1}^{s-1} \h_k \rho\left(\mathbf B_{T}\right)^{-1} \mathbf S_k\in\mathrm T_{\rncovmat_T}\rho^{-1}(\{1\})$,} there exist $h_z,h_p\in\bR^d$, $h_r\in\bR$ and $h_{1:T}\in (\bR^{d\mu})^T$ such that $\mathcal D S_{\theta_0}^T(v_{1:T})h_{1:T} = [h_z,h_p,0,h_{\ncovmat},h_t]$\del{, ending the proof} \nnew{which is the statement (ii) of the lemma. The statements (i) and (iii) have been proven earlier in the proof}.
	\end{proof}

	\subsection{Proof of \Cref{l:contrabillity-condition-step1}}
		\begin{proof}[Proof of \Cref{l:contrabillity-condition-step1}]
		Let $\theta_0\in\cX$ be a steadily attracting state satisfying \Cref{l:contrabillity-condition-step0}.
		\new{Then, there exists $T_0 > 0$
			and $v_{1:T0} \in \overline{\cO^T_{\theta_0}}$ such that conditions (i), (ii), (iii) of \Cref{l:contrabillity-condition-step0} are satisfied.}
		Let $T>T_0$, $v_{1:T}\in\overline{\cO^T_{\theta_0}}$ and $h_{1:T}\in\cV^T$. We denote for every $t\in\{1,\dots,T\}$
		$$
		\theta_t = (\Z_t,p_t,q_t,\rncovmat_t,r_t) \coloneqq S_{\theta_0}^t(v_{1:t}) \quad\text{and}\quad \theta_t^h = (\Z_t^h,p_t^h,q_t^h,\rncovmat_t^h,r_t^h)\coloneqq S_{\theta_0}^t(v_{1:t}+ h_{1:t}) .
		$$
		
		Let $s=d(d+1)/2$ be the dimension of $\Sd$. Then, $\ker \mathcal{D}\rho(\rncovmat_{T_0}) = \mathrm T_{\rncovmat_{T_0}}\rho^{-1} (\{1\})$ is a vector
		space of dimension $s-1$. Let $(\mathbf S_1 , \dots , \mathbf S_{s-1} )$ be a basis of $\ker \mathcal{D}\rho(\rncovmat_{T_0})$. 
		Then for $k = 1, \dots , s -1$,
		by condition (ii) in \Cref{l:contrabillity-condition-step0}, there exists $\xi_{1:T_0}^k\in\cV^{T_0}$
		such that \nnew{$\mathcal{D}S_{\theta_0}^{T_0}(v_{1:{T_0}})\xi_{1:T_0}^k = [h_z, h_p, 0, \mathbf S_k, h_r]$ (for some $h_z, h_p, h_r$). If}\del{, if} $h_{1:T_0} = \varepsilon_k \xi_{1:T_0}^k\in\cV^{T_0}$
		for $\varepsilon_k \in\bR$, we have
		by Taylor expansion \nnew{and linearity of the differential}:
		$$
		\rncovmat_{T_0}^h = \rncovmat_{T_0} + \varepsilon_k\mathbf S_k + o(\varepsilon_k) \enspace.
		$$
		Set then 
		$ h_{1:T_0} = \sum_{k=1}^{s-1} \varepsilon_k \xi_{1:T_0}^k$,
		so that, by linearity of the differential:
		\begin{equation}\label{eq:toh-1}
			\rncovmat_{T_0}^h = \rncovmat_{T_0} + \sum_{k=1}^{s-1}\varepsilon_k\mathbf S_k + o(\varepsilon_{1:s-1}) \enspace.
		\end{equation}
		Moreover, by conditions (ii) and (iii) in \Cref{l:contrabillity-condition-step0},
		\del{Then, by \Cref{l:contrabillity-condition-step0}, we know that, 
			there exist $T_0\in\bN$, $v_{1:T_0}$ and $\xi^1_{1:T_0},\dots,\xi^{s-1}_{1:T_0}\in\cV^T$\todo{I don't see that we need those ones for the statement that $S_{\theta_0}^{T_0}$ is differentiable in $v_{1:T_0}$}
			such that $S_{\theta_0}^{T_0}$ is differentiable in $v_{1:T_0}$.
			By setting $h_{1:T_0} = \sum_{k=1}^{s-1} \varepsilon_k \xi^k_{1:T_0}$, }%
		we get 
		\begin{equation}\label{eq:theta-T0}
			z_{T_0}=q_{T_0}=0 \quad \text{and} \quad p\coloneqq p_{T_0}\neq 0 \enspace,
		\end{equation}
		and by Taylor expansion \nnew{since $S_{\theta_0}^{T_0}(v_{1:T_0}+h_{1:T_0}) = S_{\theta_0}^{T_0}(v_{1:T_0}) + \sum_{k=1}^{s-1} \varepsilon_k \mathcal{D}S_{\theta_0}^{T_0}(v_{1:{T_0}})\xi_{1:T_0}^k + o(\varepsilon_{1:s-1})$}
		\begin{equation}\label{eq:theta-T0-perturbed}
			z_{T_0}^h = 0 + O(\varepsilon_{1:s-1}) \quad \text{and} \quad q_{T_0}^h = 0 + o(\varepsilon_{1:s-1}) \quad \text{and} \quad  p_{T_0}^h = p + O(\varepsilon_{1:s-1}) \enspace.
		\end{equation}
		\del{$q_{T_0}^h =q_{T_0}=0$, $z_{T_0}=0$, $p\coloneqq p_{T_0}\neq 0$}\del{, and when $\varepsilon_{1:s-1}\to0$,
			$$
			\rncovmat_{T_0}^h = \rncovmat_{T_0} + \sum_{k=1}^{s-1} \varepsilon_k \mathbf S_k + o(\varepsilon_{1:s-1})
			$$	
			where $(\mathbf S_1,\dots,\mathbf S_{s-1})$ is a basis of $\ker \mathcal{D}\rho (\rncovmat_{T_0})$. }		
		Let $j\in\bN$ and set $T=T_0+j+5$ and 
		\begin{equation}\label{eq:choicev}
			v_{T_0+1:T_0+j+5}=0
		\end{equation}
		Then
		$v_{T_0+1:T_0+j+5}=0\in\overline{ \cO^{j+5}_{\theta_{T_0}}}$ by \Cref{p:characterization-control-sets-cma}. Since $z_{T_0}=q_{T_0}=0$, then we obtain by applying the update equations \eqref{eq:update-z} and \eqref{eq:update-q}, that $z_{T_0:T}=q_{T_0:T}=0$.
		By condition (i) in \Cref{l:contrabillity-condition-step0}, $S_{\theta_0}^{T_0}$ is differentiable at $v_{1:T_0}$.
		Moreover, for $t=T_0,\dots,T-1$, we have $z_t=q_t=0$ and $v_{t+1}=0$, hence by \Cref{l:S-differentiable-in} (case a) $S_{\theta_t}^1$ is differentiable at $v_{t+1}$.
		By chain rule, $S^{T}_{\theta_0}$ is differentiable at $v_{1:T}$.
		\del{
			
			so that, by \Cref{l:S-differentiable-in}, $S_{\theta_0}^T$ is differentiable at $v_{1:T}$.}%
		
		We set $h_{T_0+1:T_0+j}=0$\nnew{, then since $\theta_{T_0+j}^h  = S^j_{\theta^h_{T_0}}(v_{T_0+1:T_0+j}) $}\del{ so that,} by applying \eqref{eq:update-z}, \eqref{eq:update-p} and  \eqref{eq:update-q}
		\nnew{with $v=0$ and using}\del{ to} \eqref{eq:theta-T0-perturbed}, we have
		\begin{equation}\label{eq:theta-T1-perturbed}
			z_{T_0+j}^h = \del{z_{T_0}^h = }0 + O(\varepsilon_{1:s-1}) \quad \text{and} \quad q_{T_0+j}^h = \del{q_{T_0}^h = } 0 + o(\varepsilon_{1:s-1}) \quad \text{and} \quad  p_{T_0+j}^h = (1-c_\sigma)^jp + O(\varepsilon_{1:s-1}) \enspace.
		\end{equation}
		Moreover, set
		$$
		\left\{ 
		\begin{array}{ll}
			{h}_{T_0+j+1} = (H_1,\dots,H_1) \in\bR^{d\mu} & \text{for some } H_1\in\bR^d \\
			{h}_{T_0+j+2} = -(1-c_1-c_\mu)^{-1/2} \Gamma(p_{T_0+j+1})^{-1} {h}_{T_0+j+1} &\\ 
			{h}_{T_0+j+3} = (H_3,\dots,H_3) \in\bR^{d\mu} & \text{for some } H_3\in\bR^d \\
			{h}_{T_0+j+4} = -(1-c_1-c_\mu)^{-1/2} \Gamma(p_{T_0+j+3})^{-1} {h}_{T_0+j+3}&  \\
			{h}_{T_0+j+5} = (H_5,\dots,H_5) \in\bR^{d\mu} & \text{for some } H_5\in\bR^d
		\end{array}
		\right.
		$$
		Note that $p_{T_0+k}= (1-c_\sigma)^k p$ for $k=1,\dots,j+5$ \new{by \eqref{eq:update-p}, and by Taylor expansion:
			$$
			p_{T_0+j+1}^h  = (1-c_\sigma)^{j+1} p + O(\varepsilon_{1:s-1},H_1) \enspace.
			$$
			Yet, $\Gamma$ is locally Lipschitz by \Cref{G1:C1}, thus $\Gamma(p_{T_0+j+1}^h)=\Gamma((1-c_\sigma)^{j+1} p) + O(\varepsilon_{1:s-1},H_1)$.}
		Moreover, since by \Cref{R1:homogeneous}, we have $r_{T_0+k}= R((1-c_1-c_\mu)R(\rncovmat_{T_0+k})^{-1}\rncovmat_{T_0+k})=1-c_1-c_\mu$			
		and since $R$ is locally Lipschitz by \Cref{R2:C1}, we have 
		\begin{multline}
			\label{eq:rT0kh}
			r_{T_0+k}^h=R((1-c_1-c_\mu)R(\rncovmat_{T_0+k})^{-1}\rncovmat_{T_0+k} + O(h_{1:T_0+k}))=1-c_1-c_\mu+ O(h_{1:T_0+k}) \\
			\nnew{= 1-c_1-c_\mu + O(\h_{1:s-1}) + O(h_{T_0+1:k})}
			\enspace.
		\end{multline}
		When $H_1,H_3,H_5,\varepsilon_{1:s-1}\to0$, \nnew{since 
			\begin{equation}\label{eq:update-thetat0j1}
				\theta_{T_0+j+1}^h=S_{\theta_{T_0+j}^h}^1(\new{0+}h_{T_0+j+1})
			\end{equation}
		} by applying \eqref{eq:update-z} 
		\nnew{ we find
			$$
			\Z_{T_0+j+1}^h  = \frac{\Z_{T_0+j}^h + c_m H_1}{\sqrt{r_{T_0+j+1}^h}\Gamma(p_{T_0+j+1}^h)} 
			$$
		}
		and thus using \eqref{eq:theta-T1-perturbed} and \eqref{eq:rT0kh}:
		$$
		\Z_{T_0+j+1}^h =  c_m (1-c_1-c_\mu)^{-1/2}\Gamma((1-c_\sigma)^{j+1} p) ^{-1} H_1 + O(\h_{1:s-1})  + o(H_1)
		$$
		so,
		\begin{align*}
			\Z_{T_0+j+2}^h & = \frac{\Z_{T_0+j+1}^h - c_m (1-c_1-c_\mu)^{-1/2} \Gamma(p_{T_0+j+1})^{-1} H_1}{\sqrt{r_{T_0+j+2}^h}\Gamma(p_{T_0+j+2}^h)} \\
			& =   O(\h_{1:s-1}) + o(H_1) \enspace .
		\end{align*}
		Likewise,
		$$
		\Z_{T_0+j+4}^h =  O(\h_{1:s-1}) + o(H_1,H_3) ,
		$$
		so that, in the end, since $R$ is locally Lipschitz by \Cref{R2:C1} and $\Gamma$ is locally Lipschitz by \Cref{G1:C1}, then
		$$
		\Z_{T}^h=\Z_{T_0+j+5}^h = O(\h_{1:s-1}) + o(H_1,H_3) + c_m r_{T}^{-1/2}\Gamma(p_{T}) ^{-1} H_5 \new{+o(H_5)} .
		$$
		Furthermore \nnew{using \eqref{eq:update-thetat0j1} and \eqref{eq:theta-T1-perturbed}},
		\begin{align*}
			q_{T_0+j+1}^h & \new{= (1-c_c)(r_{T_0+j}^h)^{-1/2} q_{T_0+j}^h +  \sqrt{c_c(2-c_c)\mueff}(0+H_1) }  \\
			&  =  \sqrt{c_c(2-c_c)\mueff}H_1 + o(\varepsilon_{1:s-1})
		\end{align*}
		and \done{check: $\Gamma((1-c_\sigma)^{j+1} p)$}\armand{it was $\Gamma()^{-1}$}
		\begin{align*}
			q_{T_0+j+2}^h 
			&\new{ = (1-c_c)(r_{T_0+j+1}^h)^{-1/2} q_{T_0+j+1}^h} \\ 
			& \new{~ +  \sqrt{c_c(2-c_c)\mueff}(0-(1-c_1-c_\mu)^{-1/2}\Gamma((1-c_\sigma)^{j+1}p)^{-1}H_1) }
			\\
			& \new{ = (1-c_c)(1-c_1-c_\mu+O(\varepsilon_{1:s-1},H_1))^{-1/2} \sqrt{c_c(2-c_c)\mueff}H_1 + o(\varepsilon_{1:s-1})} \\ 
			& \new{~ +  \sqrt{c_c(2-c_c)\mueff}(1-c_1-c_\mu)^{-1/2}\Gamma((1-c_\sigma)^{j+1}p)^{-1}H_1  }
			\\
			& = \underbrace{(1-c_1-c_\mu)^{-1/2}\left[1-c_c- \Gamma((1-c_\sigma)^{j+1} p)^{-1}\right]\sqrt{c_c(2-c_c)\mueff}}_{\eqqcolon d_{j+1}^p} H_1 + o(\varepsilon_{1:s-1}\new{,H_1}) \enspace.
		\end{align*}
		Likewise,
		\begin{align*}
			q_{T_0+j+3}^h  &   =   (r_{T_0+j+2})^{-1/2}(1-c_c) q_{T_0+j+2}^h   +  \sqrt{c_c(2-c_c)\mueff} H_3 \\
			&    =    (1-c_1-c_\mu)^{-1/2} (1-c_c) d_{j+1}^p H_1   +  \sqrt{c_c(2-c_c)\mueff} H_3   + o(\varepsilon_{1:s-1},H_1)   
		\end{align*}
		and
		\begin{align*}
			q_{T_0+j+4}^h  &   =   (r_{T_0+j+3})^{-1/2}(1-c_c) q_{T_0+j+3}^h   -   \sqrt{c_c(2-c_c)\mueff}  (1-c_1-c_\mu)^{-1/2} \Gamma((1-c_\sigma)^{j+3})^{-1} H_3 \\
			&    =    (1-c_1-c_\mu)^{-1} (1-c_c)^2 d_{j+1}^p H_1   + (1-c_1-c_\mu)^{-1/2} (1-c_c)  \sqrt{c_c(2-c_c)\mueff} H_3   
			\\ & ~ -     \sqrt{c_c(2-c_c)\mueff}    (1-c_1-c_\mu)^{-1/2} \Gamma((1-c_\sigma)^{j+3})^{-1} H_3
			+ o(\varepsilon_{1:s-1},H_1, H_3)    \\
			&    = (1-c_1-c_\mu)^{-1} (1-c_c)^2 d_{j+1}^p H_1   +
			d_{j+3}^p H_3 + o(\varepsilon_{1:s-1},H_1, H_3) \enspace,
		\end{align*}
		where $d_{k}^p\coloneqq (1-c_1-c_\mu)^{-1/2}\left[1-c_c- \Gamma((1-c_\sigma)^kp)^{-1}\right]\sqrt{c_c(2-c_c)\mueff}$ for $k=j+1,j+3$.
		Then,
		\begin{align*}
			q_{T}^h & = q_{T_0+j+5}^h \new{= (r_{T_0+j+4}^h)^{-1/2}(1-c_c) q_{T_0+j+4}^h   +  \sqrt{c_c(2-c_c)\mueff} H_5} \\
			& 		
			= (1-c_1-c_\mu)^{-3/2}(1-c_c)^3 d_{j+1}^p H_1 \\
			& \quad + (1-c_1-c_\mu)^{-1/2}(1-c_c)d_{j+3}^p H_3 + O(H_5) + o(\varepsilon_{1:s-1},H_1,H_3) \enspace.
		\end{align*}
		For $t=0,\dots,T$, we denote $\ncovmat_t=R(\rncovmat_t)^{-1}\rncovmat_t$ and $\ncovmat_t^h=R(\rncovmat_t^h)^{-1}\rncovmat_t^h$.
		For $t=T_0,\dots,T-1$, \new{given the choice of $v_{T_0+1:T_0+j+5}=0$  in \eqref{eq:choicev},} we have then \nnew{$\rncovmat_{t+1}=(1-c_1-c_\mu) \ncovmat_t / \rho((1-c_1-c_\mu) \ncovmat_t)$ and} by \Cref{R1:homogeneous} 
		$$
		\ncovmat_{t+1}\new{= \frac{\rncovmat_{t+1}}{R(\rncovmat_{t+1})}} = \cfrac{(1-c_1-c_\mu)\ncovmat_t}{R((1-c_1-c_\mu)\ncovmat_t)} = \ncovmat_t \enspace.
		$$
		Thus, $\ncovmat_t=\ncovmat_T$ for $t=T_0,\dots,T$. 
		Moreover, we have for $k=0,\dots,j$:
		$$
		\rncovmat_{T_0+k+1}^h = \cfrac{(1-c_1-c_\mu)\ncovmat_{T_0+k}^h  + c_1 q_{T_0+k+1}^h(q_{T_0+k+1}^h)^\top  + c_\mu \sum_{i=1}^\mu \wic h_{T_0+k+1}^i (h_{T_0+k+1}^i)^\top  }
		{\rho( (1-c_1-c_\mu)\ncovmat_{T_0+k}^h  + c_1 q_{T_0+k+1}^h(q_{T_0+k+1}^h)^\top  + c_\mu \sum_{i=1}^\mu \wic h_{T_0+k+1}^i (h_{T_0+k+1}^i)^\top )}
		$$
		Since $\rho$ is $\cC^1$ by \Cref{rho2:C1}, hence locally Lipschitz, and $q_{T_0+k+1}^h(q_{T_0+k+1}^h)^\top = 0+O(\|h_{1:T_0+k}\|^2)=o(h_{1:T_0+k})$, we have then:
		$$
		\rncovmat_{T_0+k+1}^h = \cfrac{(1-c_1-c_\mu)\ncovmat_{T_0+k}^h}{\rho( (1-c_1-c_\mu)\ncovmat_{T_0+k}^h)} + o(h_{1:T_0+k}) = \cfrac{\ncovmat_{T_0+k}^h}{\rho( \ncovmat_{T_0+k}^h)} + o(h_{1:T_0+k})  = \rncovmat_{T_0+k}^h + o(h_{1:T_0+k}) \enspace ,
		$$
		where we have used \Cref{rho1:homogeneous} to simplify the above equation. 
		Therefore, we obtain by induction \new{and using \eqref{eq:toh-1}} that
		$$
		\rncovmat_{T_0+k}^h = \rncovmat_{T_0} + \sum_{k=1}^{s-1}\varepsilon_k \mathbf{S}_k + o(\varepsilon_{1:s-1},H_1,H_3,H_5) 
		$$
		and thus, using $\ncovmat_{T_0+k}=\ncovmat_T$, and since $R$ is locally Lipschitz by \Cref{R2:C1}:
		\begin{multline}
			\ncovmat_{T_0+k}^h = \cfrac{\rncovmat_{T_0+k}^h}{R(\rncovmat_{T_0+k}^h)}  = \nnew{ \frac{\rncovmat_{T_0} + \sum_{k=1}^{s-1}\varepsilon_k \mathbf{S}_k + o(\varepsilon_{1:s-1},H_1,H_3,H_5)  }{R(\rncovmat_{T_0} + \sum_{k=1}^{s-1}\varepsilon_k \mathbf{S}_k + o(\varepsilon_{1:s-1},H_1,H_3,H_5) )} } 
			\\ 
			= \nnew{ \underbrace{\frac{\rncovmat_{T_0}}{R(\rncovmat_{T_0})}}_{=\ncovmat_{T_0} 
				}  + O(\varepsilon_{1:s-1}) + o(H_1,H_3,H_5)  = }  
			\ncovmat_{T}+ O(\varepsilon_{1:s-1}) + o(H_1,H_3,H_5)  \enspace.
		\end{multline}
		It follows that:
		\begin{align*}
			p_{T_0+j+1}^h & = (1-c_\sigma) p_{T_0+j}^h + \sqrt{c_\sigma(2-c_\sigma)\mueff} (\ncovmat_{T_0+j}^h)^{-1/2} \times ( 0 + H_1) \\
			& = (1-c_\sigma)^{j+1} p  + \sqrt{c_\sigma(2-c_\sigma)\mueff} \ncovmat_{T}^{-1/2} H_1 + O(\varepsilon_{1:s-1}) + o(H_1)
		\end{align*}
		and
		\begin{align*}
			p_{T_0+j+2}^h  =~ & (1-c_\sigma) p_{T_0+j+1}^h \\
			& + \sqrt{c_\sigma(2-c_\sigma)\mueff} (\ncovmat_{T_0+j+1}^h)^{-1/2} \times ( 0 - (1-c_1-c_\mu)^{-1/2}\Gamma((1-c_\sigma)^{j+1}p)^{-1} H_1) \\
			=~ & (1-c_\sigma)^{j+1} p  + (1-c_\sigma)\sqrt{c_\sigma(2-c_\sigma)\mueff} \ncovmat_{T}^{-1/2} H_1 \\
			&   -  \sqrt{c_\sigma(2-c_\sigma)\mueff}  (1-c_1-c_\mu)^{-1/2}\Gamma((1-c_\sigma)^{j+1}p)^{-1} \ncovmat_T^{-1/2} H_1 \\ &+ O(\varepsilon_{1:s-1}) + o(H_1,H_3,H_5)  \\
			=~ & (1-c_\sigma)^{j+2} p  + c_{j+1}^p \ncovmat_T^{-1/2} H_1 +  O(\varepsilon_{1:s-1}) + o(H_1,H_3,H5) \enspace.
		\end{align*}
		\anne{The $o(H_1,H_3,H5) $ is to be consistent with equation 5.55 . Might not be needed for all $k$ but rather fix the inconsistency around equation 5.55 then }%
		\armand{it is not needed but it avoids more explanations so I put $o(H_1,H_3,H_5)$}
		Likewise,
		\begin{align*}
			p_{T_0+j+3}^h  =~ & (1-c_\sigma) p_{T_0+j+2}^h  + \sqrt{c_\sigma(2-c_\sigma)\mueff} (\ncovmat_{T_0+j+2}^h)^{-1/2} \times ( 0 + H_3) \\
			=~ & (1-c_\sigma)^{j+3} p  + (1-c_\sigma) c_{j+1}^p \ncovmat_{T}^{-1/2} H_1 \\
			&   +  \sqrt{c_\sigma(2-c_\sigma)\mueff}   \ncovmat_T^{-1/2} H_3 + O(\varepsilon_{1:s-1}) + o(H_1,H_3,H_5) \done{+o(H_1,H_3,H5) ??}
		\end{align*}
		and
		\begin{align*}
			p_{T_0+j+4}^h  =~ & (1-c_\sigma) p_{T_0+j+3}^h \\
			& + \sqrt{c_\sigma(2-c_\sigma)\mueff} (\ncovmat_{T_0+j+3}^h)^{-1/2} \times ( 0 - (1-c_1-c_\mu)^{-1/2}\Gamma((1-c_\sigma)^{j+3}p)^{-1} H_3) \\
			=~ & (1-c_\sigma)^{j+4} p  + (1-c_\sigma)^2 c_{j+1}^p \ncovmat_{T}^{-1/2} H_1  + (1-c_\sigma) \sqrt{c_\sigma(2-c_\sigma)\mueff}   \ncovmat_T^{-1/2} H_3   \\
			&   -  \sqrt{c_\sigma(2-c_\sigma)\mueff}  (1-c_1-c_\mu)^{-1/2}\Gamma((1-c_\sigma)^{j+3}p)^{-1} \ncovmat_T^{-1/2} H_3 \\ & + O(\varepsilon_{1:s-1}) + o(H_1,H_3,H_5)  \\
			=~ & (1-c_\sigma)^{j+4} p  + (1-c_\sigma)^2c_{j+1}^p \ncovmat_T^{-1/2} H_1 + c_{j+3}^p \ncovmat_T^{-1/2} H_3 +  O(\varepsilon_{1:s-1}) + o(H_1,H_3,H_5) \enspace,
		\end{align*}
		\nnew{where $c_k^p \coloneqq (1-c_\sigma -(1-c_1-c_\mu)^{-1/2}\Gamma((1-c_\sigma)^k p))^{-1}\sqrt{c_\sigma(2-c_\sigma)\mueff}$ for $k=j+1,j+3$.}
		Finally,
		\begin{align*}
			p_{T_0+j+5}^h  =~ & (1-c_\sigma) p_{T_0+j+4}^h  + \sqrt{c_\sigma(2-c_\sigma)\mueff} (\ncovmat_{T_0+j+4}^h)^{-1/2} \times ( 0 + H_5) \\
			=~ & (1-c_\sigma)^{j+5} p  + (1-c_\sigma)^3 c_{j+1}^p \ncovmat_{T}^{-1/2} H_1 \\ &+  (1-c_\sigma) c_{j+3}^p \ncovmat_{T}^{-1/2} H_3  + O(\varepsilon_{1:s-1}) + o(H_1,H_3) + O(H_5) \enspace.
		\end{align*}
		\done{we lose $O(H_5)$ in the last equation}
		\del{	Likewise, since $c_k^p \coloneqq (1-c_\sigma -(1-c_1-c_\mu)^{-1/2}\new{c_m}\Gamma((1-c_\sigma)^k p)^{-1})\sqrt{c_\sigma(2-c_\sigma)\mueff}$, we find
			$$
			p_{T}^h = p_{T}  + (1-c_\sigma)^3 c_{j+2}^p \ncovmat_{T_1}^{-1/2} H_1 + (1-c_\sigma)c_{j+4}^p\ncovmat_{T_1}^{-1/2} H_3 + O(\h_{1:s-1},H_5) + o(H_1,H_3)  \enspace .
			$$
			Finally,
			$$
			\rncovmat_{T}^h = \rncovmat_{T} + (1-c_1-c_\mu)^{j+5} \sum_{t=1}^{s-1} \h_t \mathbf S_t + o(\h_{1:T-1},H_1,H_3,H_5) \enspace.
			$$}%
		By \Cref{R2:C1}, we have
		\begin{align*}
			r_{T}^h & = R\left( (1-c_1-c_\mu) R(\rncovmat_{T-1}^h)^{-1}\rncovmat_{T-1}^h + o(H_1,H_3,H_5) \right) \\
			& = (1-c_1-c_\mu) R(\rncovmat_{T-1}^h)^{-1} R(\rncovmat_{T-1}^h) + o(H_1,H_3,H_5) \\
			& = 1-c_1-c_\mu + o(H_1,H_3,H_5) = r_{T} + o (H_1,H_3,H_5) \enspace,
		\end{align*}
		where we have used \Cref{R1:homogeneous} to simplify the first line into the second line in the above equation.
		All in all, when $H_1,H_3,H_5,\varepsilon_{1:s-1}\to0$,
		\begin{align*}
			\theta_{T}^h & = \theta_{T} +   
				\begin{bmatrix}
					O(\h_{1:s-1}) \\ O(\h_{1:s-1}) \\ 0 \\  \sum_{t=1}^{s-1} \h_t \mathbf S_t \\ 0
				\end{bmatrix}
				+o(\h_{1:s-1},H_1,H_3,H_5) \\
				& +  
					\begin{bmatrix}
						(1-c_1-c_\mu)^{-1/2}\Gamma(p_{T})^{-1}c_m H_5 \\ (1-c_\sigma) \times   \left[(1-c_\sigma)^2c_{j+1}^p \ncovmat_{T}^{-1/2} H_1 + c_{j+3}^p \ncovmat_{T}^{-1/2} H_3\right] + O(H_5) \\  (1-c_c)(1-c_1-c_\mu)^{-1/2} \times \left[(1-c_c)^2(1-c_1-c_\mu)^{-1} d_{j+1}^p H_1 + d_{j+3}^p H_3\right] + O(H_5) \\  0 \\ 0
					\end{bmatrix}
					\enspace .
				\end{align*}
				\del{This ends the proof by }\nnew{We identify the }Taylor expansion of $S_{\theta_0}^T(v_{1:T}+h_{1:T})$ in \eqref{eq:TE}, with
				$L_z= (1-c_1-c_\mu)^{-1/2}\Gamma(p_{T})^{-1}c_m$ and $\mathbf{L}_k^{\ncovmat}= \mathbf S_k$ for $k=1,\dots,s-1$\del{.			Moreover,} \new{and} \begin{multline*}
					\mathsf W_L= \mathrm{span}(\xi_{1:T_0}^1,\dots,\xi_{1:T_0}^{s-1}) \times \{0\}^j \times  \begin{pmatrix}
						1 \\ -(1-c_1-c_\mu)^{-1/2}\Gamma((1-c_\sigma)^{j+1}p)
					\end{pmatrix}^\top \bR^{d\mu} \\
					\times \begin{pmatrix}
						1 \\ -(1-c_1-c_\mu)^{-1/2}\Gamma((1-c_\sigma)^{j+3}p)
					\end{pmatrix}^\top\bR^{d\mu} \times \bR^{d\mu} 
				\end{multline*}
				and $L\colon\mathsf W_L\to\bR^{s-1}\times(\bR^{d})^3$ maps a vector $h_{1:T}\in\mathsf W_L$ to a vector
				$(\varepsilon_{1:s-1},H_1,H_3,H_5)\in\bR^{s-1}\times(\bR^{d})^3$ such that 
				\begin{align*}
					h_{1:s-1}  = \sum_{k=1}^{s-1} \varepsilon_k \xi_{1:T_0}^k ;
					h_{s+j:s+j+1}  = \begin{pmatrix}
						H_1 \\
						-(1-c_1-c_\mu)^{-1/2}\Gamma((1-c_\sigma)^{j+1}p) H_1
					\end{pmatrix} 
				\end{align*}
				\begin{align*}
					h_{s+j+2:s+j+3}  = \begin{pmatrix}
						H_3 \\
						-(1-c_1-c_\mu)^{-1/2}\Gamma((1-c_\sigma)^{j+1}p) H_3
					\end{pmatrix} ;
					h_{s+j+4}  = H_5 \enspace .
				\end{align*}
				%
				%
				%
				\nnew{Since the scalars $\varepsilon_1,\dots,\varepsilon_{s-1}\in\bR$ and the vectors $H_1,H_3,H_5\in\bR^{d}$ above can be chosen arbitrary and independently of each other, the linear application $L\colon\mathsf{W}_L\to\bR^{s-1}\times(\bR^{d})^3$ is surjective.}
				\done{"prove" at least comment on the subjectivity which is in the statement and crucial.}
			\end{proof}
		
		\subsection{Proof of \Cref{p:full-rank-control-matrix-cma}}\label{sec:proof-p-full-rank}
		
		\begin{proof}[Proof of \Cref{p:full-rank-control-matrix-cma}(a) and (c)]
			Suppose that either $c_c\neq 1$, $c_\sigma\neq1$ and $1-c_c\neq (1-c_\sigma)\sqrt{1-c_1-c_\mu}$, or that $c_c\neq1$, $c_\sigma=1$. Assume moreover that $c_\mu>0$. 
			Apply then \Cref{l:contrabillity-condition-step1,l:controllability-condition-step2} to get that there exist $T\in\bN$ and $v_{1:T}\in\overline{ \cO^T_{\theta_0}}$ with
			\begin{itemize}
				\item[(a)] in the case $c_\sigma\neq 1$, $\mathrm{rge}~\mathcal{D}S_{\theta_0}^T(v_{1:T})\supset \bR^d\times\bR^d\times\bR^d\times \mathrm T_{\ncovmat_{T}}\rho^{-1}(\{1\}) \times\{0\}$;
				\item[(c)] in the case $c_\sigma= 1$, $\mathrm{rge}~\mathcal{D}S_{\theta_0}^T(v_{1:T})\supset \bR^d\times\{0\}\times\bR^d\times \mathrm T_{\ncovmat_{T}}\rho^{-1}(\{1\}) \times\{0\}$.
			\end{itemize}
			In both cases, \del{let}\new{consider arbitrary} $h_z,h_p,h_q\in\bR^d$, $\mathbf H_{\ncovmat}\in T_{\ncovmat_{T}} \rho^{-1}(\{1\})$, with $h_p=0$ if $c_\sigma=1$, so that there exists ${h}_{1:T}\in \cV^T$ satisfying $ \mathcal{D}S_{\theta_0}^{T}(v_{1:T}){h}_{1:T} = (h_z, h_p, h_q, \mathbf H_{\ncovmat}, 0 )^\top$.
			By Taylor expansion, we have then
			$$
			S_{\theta_0}^T (v_{1:T}+ h_{1:T}) = \begin{bmatrix}
				z_T + h_z \\
				p_T + h_p \\
				q_T + h_q \\
				\rncovmat_T+\mathbf{H}_{\ncovmat} \\
				r_T
			\end{bmatrix} + o(h_{1:T}) \enspace.
			$$
			Since $R$ is positive and positively homogeneous, it is not constant around $\rncovmat_T$.
			Besides, $R$ is differentiable at $\rncovmat_T$ and thus there exists $w\in\bR^d$ such that $\mathcal D R(\rncovmat_T)(ww^\top)\neq 0$.
			Consider the nonconstant smooth function
			$$
			G_w \colon s\in\bR \mapsto \FK(q_t,R(\rncovmat_t)^{-1}\rncovmat_t,r_t ; s [w,\dots,w] )\enspace,
			$$
			see \eqref{eq:FSigma}.
			Since $R$ is locally Lipschitz on $\Sdpp$, then $R$ is locally Lipschitz on the submanifold $\mathrm{rge}~G_w$, which is nontrivial since $G_w$ is nonconstant.
			Then, by Rademacher's theorem~\cite[Corollary B.5]{gissler2024irreducibility}, $R$ is differentiable at $G_w(s)$ for almost every $s\in\bR$.
			Moreover, we know that $\rncovmat_T=G_w(0)$ and that $\mathcal D R(\rncovmat_T)(ww^\top)\neq 0$.
			Thus, by upper semicontinuity of Clarke's Jacobian~\cite[Proposition B.9]{gissler2024irreducibility}, there exists a sufficiently small $s>0$ such that 
			$\mathcal D R(G_w(s))(ww^\top)\neq 0$.

			Then,\del{ by \Cref{R2:C1} and Rademacher's theorem~\cite[Theorem 3.2]{gariepy2015measure},} 
			we can find $\epsilon=sw\in\bR^d$ a nonzero vector small enough so that, 
			if $v_{T+1}=[\epsilon,\dots,\epsilon]\in\overline{ \cO^1_{\theta_{T}}}$, then $R$ is differentiable at $\rncovmat_{T+1}$, 
			and $\mathcal{D}R(\rncovmat_{T+1})\del{\cdot}\nnew{(}\epsilon\epsilon^\top \nnew{)}\neq 0$.
			Moreover, up to taking $s>0$ smaller, we can assume that $\Gamma$ is differentiable at $p_{T+1}\neq 0$ by \Cref{G1:C1}.
			Hence
			by composition and by \Cref{l:S-differentiable-in},
			$S_{\theta_0}^{T+1}$ is differentiable at $v_{1:T+1}$. Indeed, by chain rule~\cite[Corollary 2.6.6]{clarke1990optimization}, we have 
			\del{	$$
				\mathcal{D} S_{\theta_0}^{T+1}(v_{1:T+1})=  \mathcal D F(S_{\theta_0}^{T}(v_{1:T}\new{)},v_{T+1}) \circ \mathcal{D}S_{\theta_0}^{T}(v_{1:T}) 
				$$}%
			$$
			\mathcal{D} S_{\theta_0}^{T+1}(v_{1:T+1}) {h}_{1:T+1} =  \mathcal D F(S_{\theta_0}^{T}(v_{1:T}\new{)},v_{T+1}) \left( \mathcal{D}S_{\theta_0}^{T}(v_{1:T}) (h_{1:T}),h_{T+1} \right)
			$$
			Let ${h}_{T+1}=[h,\dots,h]\in\bR^{d\mu}$ for some arbitrary $h\in\bR^d$.
			Then, 
			\anne{I understand well the derivation below but I did not check the equations line by line - Shan is ready to proofread, we should let him know that the equations of this new part were less proofread}
			\begin{multline*}
				S_{\theta_0}^{T+1} (v_{1:T+1}+h_{1:T+1}) \\ =   
				\begin{bmatrix}
					\Fz (z_T+h_z,p_T+h_p,R(\rncovmat_T+\mathbf H_{\ncovmat})^{-1}(\rncovmat_T+\mathbf H_{\ncovmat}),r_T ; [\epsilon+h,\dots,\epsilon+h]) \\
					\Fsig (p_T+h_p,R(\rncovmat_T+\mathbf H_{\ncovmat})^{-1}(\rncovmat_T+\mathbf H_{\ncovmat}) ;[\epsilon+h,\dots,\epsilon+h]) \\
					\Fq  (q_T+h_q,r_T ; [\epsilon+h,\dots,\epsilon+h] ) \\
					\FK (q_T+h_q,R(\rncovmat_T+\mathbf H_{\ncovmat})^{-1}(\rncovmat_T+\mathbf H_{\ncovmat}),r_T ; [\epsilon+h,\dots,\epsilon+h]) \\
					\Fr (q_T+h_q,R(\rncovmat_T+\mathbf H_{\ncovmat})^{-1}(\rncovmat_T+\mathbf H_{\ncovmat}),r_T ; [\epsilon+h,\dots,\epsilon+h]) 
				\end{bmatrix}
				+ o(h_{1:T+1})
			\end{multline*}
			see \eqref{eq:update-rnormalized-chain}-\eqref{eq:update-Fr}.
			Moreover, we have
			\begin{align*}
				\Fz & (z_T+h_z,p_T+h_p,R(\rncovmat_T+\mathbf H_{\ncovmat})^{-1}(\rncovmat_T+\mathbf H_{\ncovmat}),r_T ; [\epsilon+h,\dots,\epsilon+h])  \\
				&=
				\frac{z_T+h_z+c_m (h+\epsilon)}{r_{T+1}^{1/2}\Gamma(p_{T+1}) + O(h,h_p,h_q,\mathbf{H}_{\ncovmat})} + o (h_{1:T+1}) \\
				& = \Fz (z_T,p_T,R(\rncovmat_T)^{-1}(\rncovmat_T),r_T ; [\epsilon,\dots,\epsilon]) + r_{T+1}^{-1/2}\Gamma(p_{T+1})^{-1} h_z + O(h,h_p,h_q,\mathbf{H}_{\ncovmat})
				+ o (h_{1:T+1}) \enspace,
			\end{align*}
			\begin{align*}
				\Fsig & (p_T+h_p,R(\rncovmat_T+\mathbf H_{\ncovmat})^{-1}(\rncovmat_T+\mathbf H_{\ncovmat}) ;[\epsilon+h,\dots,\epsilon+h]) \\
				& = 
				(1-c_\sigma) (p_T+h_p) + \sqrt{c_\sigma(2-c_\sigma)\mueff} R(\rncovmat_T+\mathbf H_{\ncovmat})^{1/2}(\rncovmat_T+\mathbf H_{\ncovmat})^{-1/2} (\epsilon + h) + o(h_{1:T+1})  \\
				& = \Fsig (p_T,R(\rncovmat_T)^{-1}(\rncovmat_T) ;[\epsilon,\dots,\epsilon]) + (1-c_\sigma) h_p + O(h,\mathbf H_{\ncovmat}) + o(h_{1:T+1}) \enspace,
			\end{align*}
			\begin{align*}
				\Fq & (q_T+h_q,r_T ; [\epsilon+h,\dots,\epsilon+h] ) \\
				& =
				r_T^{-1/2} (1-c_c) (q_T+h_q) + \sqrt{c_c(2-c_c)\mueff} (\epsilon + h) + o(h_{1:T+1}) \\
				& = \Fq (q_T,r_T ; [\epsilon,\dots,\epsilon] ) + (1-c_c)r_T^{-1/2}h_q +  \sqrt{c_c(2-c_c)\mueff}h  + o(h_{1:T+1})  = q_{T+1} + h_q^+ +  o(h_{1:T+1})\enspace,
			\end{align*}
			where $h_q^+= (1-c_c)r_T^{-1/2}h_q +  \sqrt{c_c(2-c_c)\mueff}h$,
			\done{I guess you want to define $h_q^+$ via the $\eqqcolon$ ? I feel it's better to define just what you want to define (as I have a doubt now reading that) ? for instance you can write an equal and then a sentence "where $h_q^+$ is defined as XXX.}%
			\begin{align*}
				\Fr & (q_T+h_q,R(\rncovmat_T+\mathbf H_{\ncovmat})^{-1}(\rncovmat_T+\mathbf H_{\ncovmat}),r_T ; [\epsilon+h,\dots,\epsilon+h]) \\
				& = 
				R\left( (1-c_1-c_\mu)  R(\rncovmat_T+\mathbf H_{\ncovmat})^{-1}(\rncovmat_T+\mathbf H_{\ncovmat}) 
				+ 
				c_1[ q_{T+1} + h_q^+][q_{T+1}+h_q^+]^\top +
				c_\mu [\epsilon +h][\epsilon +h]^\top \right) \\ &  \quad + o(h_{1:T+1}) \\
				& = R(\tilde\ncovmat_{T+1})  
				+ \mathcal{D}R(\tilde\ncovmat_{T+1})	[ (1-c_1-c_\mu) R(\rncovmat_T+\mathbf H_{\ncovmat})^{-1}\mathbf H_{\ncovmat} ]	 \\
				& \quad +	  \mathcal{D}R(\tilde\ncovmat_{T+1}) [ c_1[ q_{T+1} (h_q^+)^\top + h_q^+q_{T+1}^\top]  ] 
				+ \mathcal{D}R(\tilde\ncovmat_{T+1}) [c_\mu [\epsilon h^\top + h\epsilon^\top]  ]
				+ o(h_{1:{T+1}}) \\
				& = \Fr (q_T,R(\rncovmat_T)^{-1}(\rncovmat_T),r_T ; [\epsilon,\dots,\epsilon])
				+(1-c_1-c_\mu) R(\rncovmat_T)^{-1} \mathcal{D}R(\tilde\ncovmat_{T+1})	\mathbf H_{\ncovmat}  \\
				& \quad +	c_1  \mathcal{D}R(\tilde\ncovmat_{T+1}) [ q_{T+1} (h_q^+)^\top + h_q^+q_{T+1}^\top] 
				+ c_\mu \mathcal{D}R(\tilde\ncovmat_{T+1}) [\epsilon h^\top + h\epsilon^\top] 
				+ o(h_{1:{T+1}})  \\
				& =
				\Fr (q_T,R(\rncovmat_T)^{-1}(\rncovmat_T),r_T ; [\epsilon,\dots,\epsilon]) + 
				\mathcal{D}R(\tilde\ncovmat_{T+1}) \tilde{\mathbf H}_{\ncovmat}^+ + o(h_{1:{T+1}})
				\enspace,
			\end{align*}
			where
			$$
			  \tilde{\mathbf H}_{\ncovmat}^+ = (1-c_1-c_\mu) R(\rncovmat_T)^{-1} 	\mathbf H_{\ncovmat} + 	c_1   [ q_{T+1} (h_q^+)^\top + h_q^+q_{T+1}^\top] 
			  + c_\mu  [\epsilon h^\top + h\epsilon^\top] 
			$$
			and
			$$
			\tilde\ncovmat_{T+1}  =   (1-c_1-c_\mu)  R(\rncovmat_T)^{-1}(\rncovmat_T) + c_1 q_{T+1}q_{T+1}^\top + c_\mu \epsilon \epsilon^\top \enspace.
			$$
			Let $(x,y)\in\bR^d\times\bR$.
				Since $c_\mu>0$ and $\mathcal{D}R(\tilde\ncovmat_{T+1})(\epsilon\epsilon^\top)\neq 0$ (since \nnew{as seen above} $\mathcal{D}R(\rncovmat_{T+1})(\epsilon\epsilon^\top)\neq 0$ and $\rncovmat_{T+1}$ is proportional to $\tilde\ncovmat_{T+1}$, see \Cref{l:diff-pos-hom}), there exists $h= l \varepsilon$, where $$l=(y-c_1\mathcal D R(\tilde \ncovmat_{T+1})[q_{T+1} x^\top + x q_{T+1}^\top]-(1-c_1-c_\mu)R(\hat\ncovmat_{T})\mathcal D R(\tilde \ncovmat_{T+1})\mathbf H_{\ncovmat})/(2c_\mu\mathcal D R(\tilde \ncovmat_{T+1})(\epsilon\epsilon^\top))\enspace,$$
			and $h_q = (1-c_c)^{-1}r_T^{1/2}(x-\sqrt{c_c(2-c_c)\mueff}l\epsilon)$ such that $h^+_q=x$ and
			$$
			\mathcal D R(\tilde\ncovmat_{T+1})\tilde{\mathbf H}_{\ncovmat}^+ = (1-c_1-c_\mu)D R(\tilde\ncovmat_{T+1}) R(\rncovmat_T)^{-1} 	\mathbf H_{\ncovmat} + 	c_1  D R(\tilde\ncovmat_{T+1}) [ q_{T+1} x^\top + xq_{T+1}^\top] 
			+ 2lc_\mu  [\epsilon \epsilon^\top ] = y
			$$
			Therefore, the linear map $(h_q,h)\mapsto (h_q^+,\mathcal D R(\tilde\ncovmat_{T+1})\tilde{\mathbf H}_{\ncovmat}^+)$ valued in $\bR^d\times\bR$ is surjective.
			\del{observe that $h_q^+$ and $\mathcal D R(\tilde\ncovmat_{T+1})\tilde{\mathbf H}_{\ncovmat}^+\nnew{= }$\done{}  are linearly independent in $h_q$ and $\|h\|$ when $h$ is proportional to $\epsilon$.}\done{not sure what is meant linearly independent in $h_q$ - explain better (instead of write "observe")}\done{do not understand what it means to be linearly independent in $\|h\|$.}%
			Besides,
			$$
			\rncovmat_{T+1}^h \coloneqq \FK (q_T+h_q,R(\rncovmat_T+\mathbf H_{\ncovmat})^{-1}(\rncovmat_T+\mathbf H_{\ncovmat}),r_T ; [\epsilon+h,\dots,\epsilon+h]) 
			= \frac{  \tilde\ncovmat_{T+1}^h  }{  \rho( \tilde\ncovmat_{T+1}^h ) } \enspace,
			$$
			where
			\begin{align*}
				\tilde\ncovmat_{T+1}^h &  = \tilde\ncovmat_{T+1} +(1-c_1-c_\mu) R(\rncovmat_T)^{-1}\mathbf H_{\ncovmat}
				+ c_1[ q_{T+1} (h_q^+)^\top + h_q^+q_{T+1}^\top] \\
				& \quad
				+ c_\mu [\epsilon h^\top + h\epsilon^\top] + o(h_{1:T+1})  = \tilde\ncovmat_{T+1} + \tilde{\mathbf{H}}_{\ncovmat}^+  + o(h_{1:T+1}) 
				\enspace.
			\end{align*}
			Therefore, 
			by using the Taylor expansion $	\rho( \tilde\ncovmat_{T+1}^h )^{-1} = \rho(\tilde\ncovmat_{T+1})^{-1} - \mathcal{D}\rho ( \tilde\ncovmat_{T+1})\tilde{\mathbf{H}}_{\ncovmat}^+ + o(h_{1:T+1})$ since $\rho$ is positive and continuously differentiable by \Cref{rho2:C1}, we get
			\done{detail the expansion you use below of  $\rho( \tilde\ncovmat_{T+1}^h )$}%
			\begin{align*}
				\FK & (q_T+h_q,R(\rncovmat_T+\mathbf H_{\ncovmat})^{-1}(\rncovmat_T+\mathbf H_{\ncovmat}),r_T ; [\epsilon+h,\dots,\epsilon+h])  \\
				& = \FK (q_T,R(\rncovmat_T)^{-1}(\rncovmat_T),r_T ; [\epsilon,\dots,\epsilon]) 
				+ \rho(\tilde\ncovmat_{T+1})^{-1} \tilde{\mathbf H}_{\ncovmat}^+ - (\mathcal D\rho(\tilde\ncovmat_{T+1}) \tilde{\mathbf H}_{\ncovmat}^+) \tilde\ncovmat_{T+1}
				+ o(h_{1:T+1}) \enspace.
			\end{align*}
			All in all, by Taylor expansion,
			$$
			\mathcal D S_{\theta_0}^{T+1} (v_{1:T+1}) h_{1:T+1}  =
			\begin{bmatrix}
				r_{T+1}^{-1/2}\Gamma(p_{T+1})^{-1} h_z + O(h,h_p,h_q,\mathbf{H}_{\ncovmat}) \\
				(1-c_\sigma) h_p + O(h,\mathbf H_{\ncovmat}) \\
				h_q^+ \\
				\rho(\tilde\ncovmat_{T+1})^{-1} \tilde{\mathbf H}_{\ncovmat}^+ - (\mathcal D\rho(\tilde\ncovmat_{T+1}) \tilde{\mathbf H}_{\ncovmat}^+) \tilde\ncovmat_{T+1} \\
				\mathcal{D}R(\tilde\ncovmat_{T+1}) \tilde{\mathbf H}_{\ncovmat}^+
			\end{bmatrix} \enspace,
			$$
			which proves that every element in $\bR^d\times\bR^d\times\bR^d\times \mathrm T_{\rncovmat_{T+1}}\rho^{-1}(\{1\})\times\bR$ is reached by the linear map $\mathcal D S_{\theta_0}^{T+1} (v_{1:T+1})$ when $c_\sigma\neq1$ \nnew{so $\mathcal D S_{\theta_0}^{T+1} (v_{1:T+1})$ is surjective, hence of maximal rank, which proves the statement (a) (with $T+1$ instead of $T$)}\done{Armand: check}.
			When $c_\sigma=1$, the statement (c) follows as well
			as there exists $p=0\in\bR^d$ such that for every $(z,q,\ncovmat,r)\in  \bR^d\times\bR^d\times \mathrm T_{\rncovmat_{T+1}} \rho^{-1}(\{1\})\times\bR$, $(z,p,q,\ncovmat,r)$ belongs to the range of $\mathcal D S_{\theta_0}^{T+1} (v_{1:T+1})$.
			\del{	The proof ends.}%
			\del{ we have
				$$
				\mathcal{D} S_{\theta_0}^{T+1}(v_{1:T+1}){h}_{1:T+1} = \left[\begin{array}{c}
					\Gamma(p_{T+1}^\sigma)^{-1}r_{T+1}^{-1/2}h_z + O(h) \\
					(1-c_\sigma)h_p +O(h,\mathbf H_{\ncovmat}) \\
					(1-c_c)r_{T}^{-1/2}h_q + \sqrt{c_c(2-c_c)\mueff} h \\
					(1-c_1-c_\mu)r_{T+1}^{-1}\mathbf H_{\ncovmat} +O(h,h_q) \\
					h_r
				\end{array}\right]
				$$
				with
				$$
				h_r = \mathcal{D}R(\rncovmat_{T+1})\left[(c_1c_c(2-c_c)\mueff+c_\mu)(h\epsilon^\top+\epsilon h^\top)+c_1(1-c_c)r_{T}^{-1/2}\sqrt{c_c(2-c_c)\mueff}(h_q\epsilon^\top+\epsilon h_q^\top)  \right].
				$$
				But when $c_\mu>0$, we have that $h_r$ and $(1-c_c)r_{T}^{-1/2}h_q + \sqrt{c_c(2-c_c)\mueff} h$ are linearly independent in $h_q$ and $h$. This ends the proof of (a) and (c).}%
		\end{proof}

	\bibliography{biblio}

\begin{thebibliography}{10}

\bibitem{aboyeji2024covariance}
Esther~Tolulope Aboyeji, Oladayo~S. Ajani, and Rammohan Mallipeddi.
\newblock Covariance matrix adaptation evolution strategy based on ensemble of
  mutations for parking navigation and maneuver of autonomous vehicles.
\newblock {\em Expert Systems with Applications}, 249:123565, September 2024.

\bibitem{akiba2019optuna}
Takuya Akiba, Shotaro Sano, Toshihiko Yanase, Takeru Ohta, and Masanori Koyama.
\newblock Optuna: {{A Next-generation Hyperparameter Optimization Framework}}.
\newblock In {\em Proceedings of the 25th {{ACM SIGKDD International
  Conference}} on {{Knowledge Discovery}} \& {{Data Mining}}}, {{KDD}} '19,
  pages 2623--2631, New York, NY, USA, July 2019. Association for Computing
  Machinery.

\bibitem{akimoto2022global}
Youhei Akimoto, Anne Auger, Tobias Glasmachers, and Daiki Morinaga.
\newblock Global {{Linear Convergence}} of {{Evolution Strategies}} on {{More}}
  than {{Smooth Strongly Convex Functions}}.
\newblock {\em SIAM Journal on Optimization}, 32(2):1402--1429, June 2022.

\bibitem{arnold2004performance}
D.V. Arnold and H.-G. Beyer.
\newblock Performance analysis of evolutionary optimization with cumulative
  step length adaptation.
\newblock {\em IEEE Transactions on Automatic Control}, 49(4):617--622, April
  2004.

\bibitem{auger2005convergence}
Anne Auger.
\newblock Convergence results for the (1, $\lambda$)-sa-es using the theory of
  $\phi$-irreducible markov chains.
\newblock {\em Theoretical Computer Science}, 334(1-3):35--69, 2005.

\bibitem{auger2016these}
Anne Auger.
\newblock Analysis of {{Comparison-based Stochastic Continuous Black-Box
  Optimization Algorithms}}.
\newblock Th{\`e}se d'habilitation {\`a} diriger des recherches, Universit{\'e}
  Paris-Sud, May 2016.

\bibitem{auger2013linear}
Anne Auger and Nikolaus Hansen.
\newblock Linear convergence on positively homogeneous functions of a
  comparison based step-size adaptive randomized search: the (1+ 1) es with
  generalized one-fifth success rule.
\newblock {\em arXiv preprint arXiv:1310.8397}, 2013.

\bibitem{auger2016linear}
Anne Auger and Nikolaus Hansen.
\newblock Linear {{Convergence}} of {{Comparison-based Step-size Adaptive
  Randomized Search}} via {{Stability}} of {{Markov Chains}}.
\newblock {\em SIAM Journal on Optimization}, 26(3):1589--1624, January 2016.

\bibitem{bieler2014robust}
Jonathan Bieler, Rosamaria Cannavo, Kyle Gustafson, Cedric Gobet, David
  Gatfield, and Felix Naef.
\newblock Robust synchronization of coupled circadian and cell cycle
  oscillators in single mammalian cells.
\newblock {\em Molecular Systems Biology}, 10(7):739, July 2014.

\bibitem{bienvenue2003global}
Alexis Bienven{\"u}e and Olivier Fran{\c c}ois.
\newblock Global convergence for evolution strategies in spherical problems:
  Some simple proofs and difficulties.
\newblock {\em Theoretical Computer Science}, 306(1):269--289, September 2003.

\bibitem{brooks2011handbook}
Steve Brooks, Andrew Gelman, Galin Jones, and Xiao~Li Meng.
\newblock {\em Handbook of {{Markov Chain Monte Carlo}}}.
\newblock CRC Press, May 2011.

\bibitem{chotard2019verifiable}
Alexandre Chotard and Anne Auger.
\newblock Verifiable conditions for the irreducibility and aperiodicity of
  {{Markov}} chains by analyzing underlying deterministic models.
\newblock {\em Bernoulli}, 25(1):112--147, February 2019.

\bibitem{clarke1990optimization}
Frank~H. Clarke.
\newblock {\em Optimization and {{Nonsmooth Analysis}}}.
\newblock SIAM, January 1990.

\bibitem{colutto2010cmaes}
Sebastian Colutto, Florian Fruhauf, Matthias Fuchs, and Otmar Scherzer.
\newblock The {{CMA-ES}} on {{Riemannian Manifolds}} to {{Reconstruct Shapes}}
  in 3-{{D Voxel Images}}.
\newblock {\em IEEE Transactions on Evolutionary Computation}, 14(2):227--245,
  April 2010.

\bibitem{gariepy2015measure}
Lawrence~Craig Evans and Ronald~F Gariepy.
\newblock {\em Measure {{Theory}} and {{Fine Properties}} of {{Functions}},
  {{Revised Edition}}}.
\newblock {Chapman and Hall/CRC}, New York, April 2015.

\bibitem{fujii2018exploring}
Garuda Fujii, Youhei Akimoto, and Masayuki Takahashi.
\newblock Exploring optimal topology of thermal cloaks by {{CMA-ES}}.
\newblock {\em Applied Physics Letters}, 112(6):061108, February 2018.

\bibitem{gallegos2023equivalences}
Marco~A. {Gallegos-Herrada}, David Ledvinka, and Jeffrey~S. Rosenthal.
\newblock Equivalences of {{Geometric Ergodicity}} of {{Markov Chains}}.
\newblock {\em Journal of Theoretical Probability}, May 2023.

\bibitem{gissler2023evaluation}
Armand Gissler.
\newblock Evaluation of the impact of various modifications to {{CMA-ES}} that
  facilitate its theoretical analysis.
\newblock In {\em {{GECCO}} 2023 - {{Genetic}} and {{Evolutionary Computation
  Conference}}}, July 2023.

\bibitem{gissler2024irreducibility}
Armand Gissler, Alain Durmus, and Anne Auger.
\newblock On the irreducibility and convergence of a class of nonsmooth
  nonlinear state-space models on manifolds, February 2024.

\bibitem{hansen2015evolution}
Nikolaus Hansen, Dirk~V Arnold, and Anne Auger.
\newblock Evolution {{Strategies}}.
\newblock 2015.

\bibitem{hansen2011cmaes}
Nikolaus Hansen and Anne Auger.
\newblock {{CMA-ES}}: Evolution strategies and covariance matrix adaptation.
\newblock In {\em Proceedings of the 13th Annual Conference Companion on
  {{Genetic}} and Evolutionary Computation}, {{GECCO}} '11, pages 991--1010,
  New York, NY, USA, July 2011. Association for Computing Machinery.

\bibitem{hansen2014principled}
Nikolaus Hansen and Anne Auger.
\newblock Principled {{Design}} of {{Continuous Stochastic Search}}: {{From
  Theory}} to {{Practice}}.
\newblock In Yossi Borenstein and Alberto Moraglio, editors, {\em Theory and
  {{Principled Methods}} for the {{Design}} of {{Metaheuristics}}}, Natural
  {{Computing Series}}, pages 145--180. Springer, Berlin, Heidelberg, 2014.

\bibitem{hansen2004evaluating}
Nikolaus Hansen and Stefan Kern.
\newblock Evaluating the {{CMA Evolution Strategy}} on {{Multimodal Test
  Functions}}.
\newblock In {\em Parallel {{Problem Solving}} from {{Nature}} - {{PPSN
  VIII}}}, Lecture {{Notes}} in {{Computer Science}}, pages 282--291, Berlin,
  Heidelberg, 2004. Springer.

\bibitem{hansen2003reducing}
Nikolaus Hansen, Sibylle~D. M{\"u}ller, and Petros Koumoutsakos.
\newblock Reducing the {{Time Complexity}} of the {{Derandomized Evolution
  Strategy}} with {{Covariance Matrix Adaptation}} ({{CMA-ES}}).
\newblock {\em Evolutionary Computation}, 11(1):1--18, March 2003.

\bibitem{hansen1996adapting}
Nikolaus Hansen and Andreas Ostermeier.
\newblock Adapting arbitrary normal mutation distributions in evolution
  strategies: The covariance matrix adaptation.
\newblock In {\em Proceedings of {{IEEE International Conference}} on
  {{Evolutionary Computation}}}, pages 312--317, May 1996.

\bibitem{hansen2001completely}
Nikolaus Hansen and Andreas Ostermeier.
\newblock Completely {{Derandomized Self-Adaptation}} in {{Evolution
  Strategies}}.
\newblock {\em Evolutionary Computation}, 9(2):159--195, June 2001.

\bibitem{hansen2010benchmarking}
Nikolaus Hansen and Raymond Ros.
\newblock Benchmarking a weighted negative covariance matrix update on the
  {{BBOB-2010}} noiseless testbed.
\newblock In {\em Proceedings of the 12th Annual Conference Companion on
  {{Genetic}} and Evolutionary Computation}, {{GECCO}} '10, pages 1673--1680,
  New York, NY, USA, July 2010. Association for Computing Machinery.

\bibitem{hastings1970monte}
W.~K. Hastings.
\newblock Monte {{Carlo}} sampling methods using {{Markov}} chains and their
  applications.
\newblock {\em Biometrika}, 57(1):97--109, April 1970.

\bibitem{horn2013matrix}
{Horn R. {and} Johnson C.}
\newblock {\em Matrix {{Analysis}}}.
\newblock Cambridge University Press, 2013.

\bibitem{jastrebski2006improving}
G.A. Jastrebski and D.V. Arnold.
\newblock Improving {{Evolution Strategies}} through {{Active Covariance Matrix
  Adaptation}}.
\newblock In {\em 2006 {{IEEE International Conference}} on {{Evolutionary
  Computation}}}, pages 2814--2821, July 2006.

\bibitem{lee2012introduction}
John~M. Lee.
\newblock {\em Introduction to {{Smooth Manifolds}}}, volume 218 of {\em
  Graduate {{Texts}} in {{Mathematics}}}.
\newblock Springer, New York, NY, 2012.

\bibitem{maki2020application}
Atsuo Maki, Naoki Sakamoto, Youhei Akimoto, Hiroyuki Nishikawa, and Naoya
  Umeda.
\newblock Application of optimal control theory based on the evolution strategy
  ({{CMA-ES}}) to automatic berthing.
\newblock {\em Journal of Marine Science and Technology}, 25(1):221--233, March
  2020.

\bibitem{metropolis1953equation}
Nicholas Metropolis, Arianna~W. Rosenbluth, Marshall~N. Rosenbluth, Augusta~H.
  Teller, and Edward Teller.
\newblock Equation of {{State Calculations}} by {{Fast Computing Machines}}.
\newblock {\em The Journal of Chemical Physics}, 21(6):1087--1092, June 1953.

\bibitem{meyn1991asymptotic}
S.~P. Meyn and P.~E. Caines.
\newblock Asymptotic {{Behavior}} of {{Stochastic Systems Possessing Markovian
  Realizations}}.
\newblock {\em SIAM Journal on Control and Optimization}, 29(3):535--561, May
  1991.

\bibitem{meyn2012markov}
Sean~P. Meyn and Richard~L. Tweedie.
\newblock {\em Markov {{Chains}} and {{Stochastic Stability}}}.
\newblock Springer Science \& Business Media, December 2012.

\bibitem{morinaga2024convergence}
Daiki Morinaga, Kazuto Fukuchi, Jun Sakuma, and Youhei Akimoto.
\newblock Convergence {{Rate}} of the (1+1)-{{ES}} on {{Locally Strongly
  Convex}} and {{Lipschitz Smooth Functions}}.
\newblock {\em IEEE Transactions on Evolutionary Computation}, 28(2):501--515,
  April 2024.

\bibitem{ostermeier1994stepsize}
Andreas Ostermeier, Andreas Gawelczyk, and Nikolaus Hansen.
\newblock Step-size adaptation based on non-local use of selection information.
\newblock In Yuval Davidor, Hans-Paul Schwefel, and Reinhard M{\"a}nner,
  editors, {\em Parallel {{Problem Solving}} from {{Nature}} --- {{PPSN III}}},
  pages 189--198, Berlin, Heidelberg, 1994. Springer.

\bibitem{roberts2004general}
Gareth~O. Roberts and Jeffrey~S. Rosenthal.
\newblock General state space {{Markov}} chains and {{MCMC}} algorithms.
\newblock {\em Probability Surveys}, 1(none):20--71, January 2004.

\bibitem{rodriguez2006hybrid}
Maria {Rodriguez-Fernandez}, Pedro Mendes, and Julio~R. Banga.
\newblock A hybrid approach for efficient and robust parameter estimation in
  biochemical pathways.
\newblock {\em Biosystems}, 83(2):248--265, February 2006.

\bibitem{serre2010matrices}
Denis Serre.
\newblock {\em Matrices: {{Theory}} and {{Applications}}}, volume 216 of {\em
  Graduate {{Texts}} in {{Mathematics}}}.
\newblock Springer, New York, NY, 2010.

\bibitem{tanabe2021level}
Takumi Tanabe, Kazuto Fukuchi, Jun Sakuma, and Youhei Akimoto.
\newblock Level generation for angry birds with sequential {{VAE}} and latent
  variable evolution.
\newblock In {\em Proceedings of the {{Genetic}} and {{Evolutionary Computation
  Conference}}}, {{GECCO}} '21, pages 1052--1060, New York, NY, USA, June 2021.
  Association for Computing Machinery.

\bibitem{toure2023global}
Cheikh Toure, Anne Auger, and Nikolaus Hansen.
\newblock Global linear convergence of evolution strategies with recombination
  on scaling-invariant functions.
\newblock {\em Journal of Global Optimization}, 86(1):163--203, May 2023.

\bibitem{toure2021scaling}
Cheikh Toure, Armand Gissler, Anne Auger, and Nikolaus Hansen.
\newblock Scaling-invariant {{Functions}} versus {{Positively Homogeneous
  Functions}}.
\newblock {\em Journal of Optimization Theory and Applications},
  191(1):363--383, October 2021.

\end{thebibliography}
	\bibliographystyle{plain}
	
\end{document}